\documentclass[english]{article}
\usepackage[latin9]{inputenc}
\usepackage{color}
\usepackage{babel}
\usepackage{array}
\usepackage{float}
\usepackage{mathtools}
\usepackage{multirow}
\usepackage{amsmath}
\usepackage{amsthm}
\usepackage{amssymb}
\usepackage{graphicx}
\usepackage{esint}
\usepackage{textcomp}
\usepackage{setspace}
\usepackage[unicode=true]
{hyperref}

\makeatletter

\providecommand{\tabularnewline}{\\}

\theoremstyle{remark}
\newtheorem{rem}{\protect\remarkname}
\theoremstyle{plain}
\newtheorem{lem}{\protect\lemmaname}
\theoremstyle{plain}
\newtheorem{thm}{\protect\theoremname}

\@ifundefined{showcaptionsetup}{}{%
	\PassOptionsToPackage{caption=false}{subfig}}
\usepackage{subfig}
\makeatother

\providecommand{\lemmaname}{Lemma}
\providecommand{\remarkname}{Remark}
\providecommand{\theoremname}{Theorem}

\begin{document}
	
	\title{Fast Computation of Steady-State Response\\
		for Nonlinear Vibrations of High-Degree-of-Freedom Systems}
	\date{\vspace{-5ex}}
	\author{Shobhit Jain, Thomas Breunung, George Haller}
	\maketitle
	\begin{center}
		Institute for Mechanical Systems, ETH Zürich \\
		Leonhardstrasse 21, 8092 Zürich, Switzerland
		\par\end{center}
	\begin{abstract}
		We discuss an integral equation approach that enables fast computation
		of the response of nonlinear multi-degree-of-freedom mechanical systems
		under periodic and quasi-periodic external excitation. The kernel of
		this integral equation is a Green's function that we compute explicitly
		for general mechanical systems. We derive conditions under which the
		integral equation can be solved by a simple and fast Picard iteration even
		for non-smooth mechanical systems. The convergence of this iteration
		cannot be guaranteed for near-resonant forcing, for which we
		employ a Newton\textendash Raphson iteration instead, obtaining
		robust convergence. We further show that this integral-equation approach
		can be appended with standard continuation schemes to achieve an additional, significant
		performance increase over common approaches to computing steady-state response.
	\end{abstract}
	
	\section{Introduction}
	
	Multi-degree-of-freedom nonlinear mechanical systems generally approach
	a steady-state response under periodic or quasi-periodic forcing. Determining
	this response is often the most important objective in analyzing nonlinear
	vibrations in engineering practice.
	
	Despite the broad availability of effective numerical packages and
	powerful computers, identifying the steady-state response simply by
	numerically integrating the equations of motion is often a poor choice.
	First, modern engineering structures tend to be very lightly damped,
	resulting in exceedingly long integration times before the steady
	state is reached. Second, structural vibrations problems to be analyzed
	are often available as finite-element models for which repeated evaluations
	of the defining functions are costly. These evaluations are inherently
	not parallelizable, thus, increasing the number of processors used in
	the simulation results in increased cross-communication times that
	slow down already slowly-converging runs even further. As
	a result, even with today's advances in computing, it may take days
	to reach a good approximation to a steady-state response in complex
	structural vibration problems (cf., e.g. \cite{Avery}). 
	
	To achieve feasible computation times for steady-state response
	in high-dimensional systems, reduced-order models (ROM) are often
	used to obtain a low-dimensional variant of the mechanical system.
	Various nonlinear normal modes (NNM) concepts have been used to describe
	such small-amplitude, nonlinear oscillations. Among these, the classic
	NNM defintion of Rosenberg \cite{rosenberg62} targets periodic
	orbits in a two-dimensional subcenter-manifold \cite{kelley_67} in the undamped limit of the oscillatory system. By contrast, Shaw
	\& Pierre \cite{ShawPierre} define NNMs as the invariant manifolds
	tangent to modal subspaces at an equilibrium point (cf. Avramov \&
	Mikhlin \cite{avramov10} for a review) allowing application
	to dissipative systems. Haller \& Ponsioen \cite{Haller_SSM} distinguish
	these two notions for dissipative systems under possible periodic/quasi-periodic
	forcing, by defining an NNM as a near-equilibrium trajectory with
	finitely many frequencies, and introducing a spectral submanifold
	(SSM) as the smoothest invariant manifold tangent to a spectral subbundle
	along such an NNM. 
	
	Alternatively, ROMs obtained using heuristic projection-based techniques are
	also used to approximate steady-state response of high-dimensional
	systems. These include sub-structuring methods such  as the Craig\textendash Bampton method~\cite{Craig} (cf. Theodosiou~et~al.~\cite{Theodosiou_09}),
	proper orthogonal decomposition \cite{POD} (cf. Kerchen~et~al.~\cite{Kerchen_POD}), reduction using natural modes (cf.
	Amabili~\cite{Amabili_Nat_modes}, Touz\'{e}~et~al.~\cite{Touze}) and
	the modal-derivative method of Idelsohn ~\&~Cardona~\cite{Cardona} (cf.
	Sombroek~et~al.~\cite{Cees}, Jain~et~al.~\cite{Jain_2017}). A common
	feature of these methods is their \emph{local} nature:
	they seek to approximate nonlinear steady-state response in
	the vicinity of an equilibrium. Thus, high-amplitude oscillations
	are generally missed by these approaches.
	
	On the analytic side, perturbation techniques relying on a small parameter have been widely used to approximate the steady-state response
	of nonlinear systems. Nayfeh et al. \cite{nayfeh04,nayfeh74} give
	a formal multiple-scales expansion applied to a system with small
	damping, small nonlinearities and small forcing. Their results are
	detailed amplitude equations to be worked out on a case-by-case
	basis. Mitropolskii and Van Dao~\cite{mitropolskii97} apply the method
	of averaging (cf. Bogoliubov and Mitropolsky~\cite{bogoliubov61}
	or, more recently, Sanders and Verhulst \cite{sanders85}) after a
	transformation to amplitude-phase coordinates in the case of small
	damping, nonlinearities and forcing. They consider single as well
	as multi-harmonic forcing of multi degree of freedom systems and obtain
	the solution in terms of a multi-frequency Fourier expansion. Their
	formulas become involved even for a single oscillator, thus
	condensed formulas or algorithms are unavailable for general systems. As conceded by Mitroposkii and Van Dao~\cite{mitropolskii97}, the series expansion is formal, as no attention is given
	to the actual existence of a periodic response. Existence is indeed a subtle question
	in this context, since the envisioned periodic orbits would perturb from a non-hyperbolic fixed point. 
	
	Vakakis~\cite{vakakis95} relaxes the small nonlinearity assumption
	and describes a perturbation approach for obtaining the periodic response
	of a single-degree-of-freedom Duffing oscillator subject to small
	forcing and small damping. A formal series expansion is performed
	around a conservative limit, where periodic solutions are explicitly
	known (elliptic Duffing oscillator). This approach only works for
	perturbations of integrable nonlinear systems.
	
	Formally applicable without any small parameter assumption is the
	harmonic balance method. Introduced first by Kryloff and Bogoliuboff
	\cite{Kryloff} for single-harmonic approximation of the forced response,
	the method has been gradually extended to include higher harmonics and quasi-periodic
	orbits (cf. Chua and Ushida \cite{Chua_harm_bal} and Lau and Cheung~\cite{Lau_QuasiP}). In the harmonic balance procedure, the assumed
	steady-state solution is expanded in a Fourier series which, upon substitution, turns the original differential equations into a set of nonlinear algebraic equations for the unknown Fourier coefficients after truncation to finitely many harmonics. The error arising
	from this truncation, however, is not well understood. For the periodic case, Leipholz~\cite{Leipholz_harmbal} and Bobylev~et~al.~\cite{bobylev94}
	show that the solution of the harmonic balance converges to the actual
	solution of the system if the periodic orbit exists and the
	number of harmonics considered tends to infinity. Explicit error
	bounds are only available as functions of the (a priori unknown) periodic
	orbit (cf. Bobylev~et~al.~\cite{bobylev94}, Urabe~\cite{Urabe_HarmBal},
	Stokes \cite{Stokes Harmbal} and  Garc{\'i}a-Saldaña~and~Gasull~\cite{Saldana_harmbal}). The quantities involved, however,
	generally require numerical integration to obtain. For quasi-periodic forcing, such error bounds remain unknown to the best of our knowledge. 
	
	The shooting method (cf. Keller \cite{Keller_shooting}, Peeters et
	al. \cite{peters09} and Sracic and Allen \cite{Allen_shooting})
	is also broadly used to compute periodic orbits of nonlinear system.
	In this procedure, the periodicity of the sought orbit is
	used to formulate a two-point boundary value problem. The solutions
	are initial conditions on the periodic orbit. Starting from an initial
	guess, one corrects the initial conditions iteratively until the boundary
	value problem is solved up to a required precision. The iterated correction
	of the initial conditions, however, requires repeated numerical integration
	of the equation of variations along the current estimate of the periodic
	orbit, as well as numerical integration of the full system. Albeit
	the shooting method has moderate memory requirements relative to that of
	harmonic balance due to its smaller Jacobian, this advantage
	is useful only for very high-dimensional systems with memory constraints.
	In practice, shooting is limited by the capabilities of the time integrator
	used and can be unsuitable for solutions with
	large Floquet multipliers, as observed by Seydel \cite{Seydel}.
	Furthermore, the shooting method is only applicable to periodic steady-state solutions, not to quasi-periodic ones. 
	
	The shooting method uses a time-march-type integration, i.e., the solution at each time step is solved sequentially after the previous one. In contrast, collocation approaches solve for the solution at all time steps in the orbit simultaneously. Collocation schemes mitigate all the drawbacks of the shooting method but can be computationally expensive for large systems since all unknowns need to be solved together over the full orbit. Popular software packages, such as AUTO~\cite{Auto}, MATCONT~\cite{Matcont} and the \texttt{po}~toolbox of \textsc{coco}~\cite{CoCo}, also use collocation schemes to continue periodic solutions of dynamical systems. Renson et~al.~\cite{Renson} provide a thorough review of the commonly used methods for computation
	of periodic orbits in multi-degree-of-freedom mechanical systems. 
	
	Constructing particular solutions using integral equations is textbook
	material in physics or vibration courses for impulsive forcing the (system is at rest at the
	initial time, prior to which the forcing is zero). Solving this problem
	with a classic Duhamel integral will produce a particular solution
	that approaches the steady-state response asymptotically. This approach,
	therefore, suffers from the slow convergence we have already discussed
	for direct numerical integration. 
	
	In this paper, assuming either periodicity or quasi-periodicity for
	the external forcing, we derive an integral equation whose zeros
	are the steady-state responses of the mechanical system. Along with a phase condition to ensure uniqueness, the same integral equation can also
	be used to obtain the (quasi-) periodic response in conservative,
	autonomous mechanical systems. 
	
	While certain elements of the integral equations approach outlined here for periodic forcing have been already discussed outside the structural vibrations literature, our treatment of quasi-periodic forcing appears to be completely new. We do not set any conceptual bounds on the number of independent frequencies allowed in such a forcing, which enables one to apply the results to more complex excitations mimicking stochastic forcing.
	
	First, we derive a Picard iteration approach with explicit convergence
	criteria to solve the integral equations for the steady-state response
	iteratively (Section \ref{sec:Piccard}). This fast iteration
	approach is particularly appealing for high-dimensional systems, since
	it does not require the construction and inversion of Jacobian matrices, and for non-smooth systems, as is does not rely on derivatives. At the same time, this Picard iteration will not converge near external resonances. Applying a Newton\textendash Raphson
	scheme to the integral equation, however, we can achieve convergence of
	the iteration even for near-resonant forcing (Section~\ref{sec:Newton-Iteration}).
	We additionally employ numerical continuation schemes to obtain forced
	response and backbone curves of nonlinear mechanical systems (Section~\ref{subsec:Numerical-continuation}). Finally, we illustrate the
	performance gain from our newly proposed approach on several multi-degree-of-freedom
	mechanical examples (Section~\ref{sec:Numerical-Examples}), using a MATLAB\textsuperscript{\tiny\textregistered}-based implementation\footnote{available at \url{www.georgehaller.com}}.
	
	\section{Set-up}
	\label{sec:setup}
	We consider a general nonlinear mechanical system of the form
	\begin{equation}
	\mathbf{M}\ddot{\mathbf{x}}+\mathbf{C}\dot{\mathbf{x}}+\mathbf{Kx}+\mathbf{S}(\mathbf{x},\dot{\mathbf{x}})=\mathbf{f}(t),\label{eq:0}
	\end{equation}
	where $\mathbf{x}(t)\in\mathbb{R}^{n}$ is the vector of generalized
	displacements; $\mathbf{M},\mathbf{C},\mathbf{K}\in\mathbb{R}^{n\times n}$
	are the symmetric mass, stiffness and damping matrices; $\mathbf{S}$
	is a nonlinear, Lipschitz-continuous function such that $\mathbf{S}=\mathcal{O}\left(\left|\mathbf{x}\right|^{2},\left|\mathbf{x}\right|\left|\dot{\mathbf{x}}\right|,\left|\dot{\mathbf{x}}\right|^{2}\right)$;
	$\mathbf{f}$ is a time-dependent, multi-frequency external forcing.
	Specifically, we assume that $\mathbf{f}(t)$ is quasi-periodic with
	a rationally incommensurate frequency basis $\boldsymbol{\Omega}\in\mathbb{R}^{k},\,k\ge1$
	which means 
	\begin{equation}
	\mathbf{f}(t)=\tilde{\mathbf{f}}(\boldsymbol{\Omega}t),\qquad\left\langle \boldsymbol{\kappa},\boldsymbol{\Omega}\right\rangle \neq0,\quad\boldsymbol{\kappa}\in\mathbb{Z}^{k}-\{\mathbf{0}\},\label{eq:freq_basis}
	\end{equation}
	for some continuous function $\tilde{\mathbf{f}}\colon\mathbb{T}^{k}\to\mathbb{R}^{n}$,
	defined on a $k-$dimensional torus $\mathbb{T}^{k}$. For $k=1$,
	$\mathbf{f}$ is periodic in $t$ with period $T=2\pi/\boldsymbol{\Omega}$,
	while for $k>1$, $\mathbf{f}$ describes a strictly quasi-periodic
	forcing. System~\eqref{eq:0} can be equivalently expressed in the
	first-order form as
	\begin{equation}
	\mathbf{B}\dot{\mathbf{z}}=\mathbf{A}\mathbf{z}+\mathbf{R}(\mathbf{z})+\mathbf{F}(t)\,,\label{eq:firstorder}
	\end{equation}
	with
	\begin{equation*}
	\mathbf{z}=\left[\begin{array}{c}
	\dot{\mathbf{x}}\\
	\mathbf{x}
	\end{array}\right],\quad\mathbf{B}=\left[\begin{array}{cc}
	\mathbf{0} & \mathbf{M}\\
	\mathbf{M} & \mathbf{C}
	\end{array}\right],\quad\mathbf{A}=\left[\begin{array}{cc}
	\mathbf{M} & \mathbf{0}\\
	\mathbf{0} & -\mathbf{K}
	\end{array}\right],\quad\mathbf{R}(\mathbf{z})=\left[\begin{array}{c}
	\mathbf{0}\\
	\mathbf{S}(\mathbf{x},\dot{\mathbf{x}})
	\end{array}\right],\quad\mathbf{F}(t)=\left[\begin{array}{c}
	\mathbf{0}\\
	\mathbf{f}(t)
	\end{array}\right]\,.
	\end{equation*}
	
	The first-order form in \eqref{eq:firstorder} ensures that the coefficient
	matrices $\mathbf{A}$ and $\mathbf{B}$ are symmetric, if the matrices
	$\mathbf{M},\mathbf{C}$ and $\mathbf{K}$ are symmetric, as is
	usually the case in structural dynamics applications (cf. G\'{e}rardin~\&~Rixen~\cite{Rixen}). We assume that
	the coefficient matrix of the linear system
	\begin{equation}
	\mathbf{B}\dot{\mathbf{z}}=\mathbf{A}\mathbf{z}+\mathbf{F}(t)\label{eq:firstorderlinear}
	\end{equation}
	can be diagonalized using the eigenvectors of the generalized eigenvalue
	problem 
	\begin{equation}
	\left(\mathbf{A}-\lambda_{j}\mathbf{B}\right)\mathbf{v}_{j}=\mathbf{0},\quad j=1,\dots,2n,\label{eq:evp}
	\end{equation}
	via the linear transformation $\mathbf{z}=\mathbf{V}\mathbf{w}$,
	where $\mathbf{w}\in\mathbb{C}^{2n}$ represents the modal variables
	and~\linebreak${\ensuremath{\mathbf{V}=\left[\mathbf{v}_{1},\ldots,\mathbf{v}_{2n}\right]\in\mathbb{C}^{2n\times2n}}}$
	is the modal transformation matrix containing the eigenvectors. The
	diagonalized linear version of \eqref{eq:firstorderlinear} with forcing
	is given by 
	\begin{align}
	\dot{\mathbf{w}} & =\boldsymbol{\Lambda}\mathbf{w}+\boldsymbol{\psi}(t),\label{eq:modalfirstorder}
	\end{align}
	where $\mbox{\ensuremath{\boldsymbol{\Lambda}}=\ensuremath{\mathrm{diag}\left(\lambda_{1},\ldots,\lambda_{2n}\right)}}$,
	$\mbox{\ensuremath{\psi_{j}}(t)=\ensuremath{\frac{\tilde{\mathbf{v}}_{j}\mathbf{F}(t)}{\tilde{\mathbf{v}}_{j}\mathbf{B}\mathbf{v}_{j}}}}$,
	where $\tilde{\mathbf{v}}_{j}$ denotes the $j^{th}$ row of the matrix
	$\mathbf{V}^{-1}$. Furthermore, if the matrices $\mathbf{A}$ and
	$\mathbf{\mathbf{B}}$ are symmetric, then $\mbox{\ensuremath{\mathbf{V}^{-1}}=\ensuremath{\mathbf{V}^{\top}}}$.
	\begin{rem}
		\small
		We have assumed autonomous nonlinearities $\mathbf{S},\mathbf{R}$
		in the equations~\eqref{eq:0} and \eqref{eq:firstorder} since this
		is relevant for structural dynamics systems, but the following treatment
		also allows for time-dependence in $\mathbf{S}$ or $\mathbf{R}$.
		Specifically, all the following results hold for nonlinearities with
		explicit time dependence as long as the time-dependence is quasi-periodic
		(cf. eq.~\eqref{eq:freq_basis}) with the same frequency basis $\boldsymbol{\Omega}$
		as that of the external forcing $\mathbf{f}(t)$. 
	\end{rem}
	
	\subsection{Periodically forced system\label{subsec:Forced-linear-system}}
	
	We first review a classic result for periodic solutions in periodically
	forced linear systems (cf. {Burd~\cite{burd07}}).
	\begin{lem}
		\label{Lemma:1}If the forcing $\mathbf{F}(t)$ is $T$-periodic,i.e.,
		$\mathbf{F}(t+T)=\mathbf{F}(t),\:t\in\mathbb{R},$ and the non-resonance
		condition
		\begin{equation}
		\lambda_{j}\neq i\frac{2\pi}{T}\ell,\qquad \ell \in\mathbb{Z},\label{eq:nonresonance-1}
		\end{equation}
		is satisfied for all eigenvalues $\lambda_{1},\dots,\lambda_{2n}$
		defined in \eqref{eq:evp}, then there exists a unique $T$-periodic
		response to \eqref{eq:firstorderlinear}, given by 
		\begin{align}
		\mathbf{z}(t) & =\mathbf{V}\int_{0}^{T}\mathbf{G}(t-s,T)\mathbf{V}^{-1}\mathbf{F}(s)\,ds,
		\end{align}
		where $\mathbf{G}(t,T)$ is the diagonal matrix of periodic Green's
		functions for the modal displacement variables, defined as 
		\[
		\mathbf{G}(t,T)=\mathrm{diag}\left(G_{1}(t,T),\ldots,G_{2n}(t,T)\right)\in\mathbb{C}^{2n\times2n},
		\]
		\begin{equation}
		G_{j}(t,T)=e^{\lambda_{j}t}\left(\frac{e^{\lambda_{j}T}}{1-e^{\lambda_{j}T}}+h(t)\right),\quad j=1,\dots,2n\,,\label{eq:G_firstorder}
		\end{equation}
		with the Heaviside function $h(t)$ given by
		\[
		h(t):=\begin{cases}
		1 & t\ge0\\
		0 & t<0
		\end{cases}\,.
		\]
		\begin{proof}
			We reproduce the proof for completeness in Appendix~\ref{sec:ProofL1}.
		\end{proof}
	\end{lem}
	\begin{rem}
		\small
		\label{rmk:Gamma_T}The uniform-in-time sup-norm of the Green's function~\eqref{eq:G_firstorder} can be bounded by the constant~$\Gamma(T)$ defined as
		\begin{equation}
		\varGamma(T):=\max_{1\leq j\leq n}\frac{T\max(\left|e^{\lambda_{j}T}\right|,1)}{\left|1-e^{\lambda_{j}T}\right|}\geq\max_{0\leq t<T}\left\Vert \int_{0}^{T}\left\Vert \mathbf{G}(t-s,T)\right\Vert _{0}\,ds\right\Vert _{0}.\label{eq:GammaT}
		\end{equation}
		We detail this estimate in Appendix \ref{app:Gamma_T_estimate}. 
	\end{rem}
	
	The Green's functions defined in \eqref{eq:G_firstorder} turn out
	to play a key role in describing periodic solutions of the full, nonlinear
	system as well. We recall this in the following result. 
	
	\begin{thm}
		\label{Theorem:IE_P}(i) If $\mathbf{z}(t)$ is a $T-$periodic solution
		of the nonlinear system \eqref{eq:firstorder}, then $\mathbf{z}(t)$
		must satisfy the integral equation
		\begin{equation}
		\mathbf{z}(t)=\mathbf{V}\int_{0}^{T}\mathbf{G}(t-s,T)\mathbf{V}^{-1}\left[\mathbf{F}(s)-\mathbf{R}(\mathbf{z}(s))\right]\,ds\,.\label{eq:int_periodic}
		\end{equation}
		(ii) Furthermore, any continuous, $T-$periodic solution $\mathbf{z}(t)$
		of \eqref{eq:int_periodic} is a $T-$periodic solution of the nonlinear
		system \eqref{eq:firstorder}.
		\begin{proof}
			See Appendix \ref{sec:ProofT1} for the proof, where the term~$ \mathbf{V}^{-1}\left[\mathbf{F}(t)-\mathbf{R}(\mathbf{z}(t))\right] $ is treated as a periodic forcing term in \eqref{eq:modalfirstorder} for a $ T $-periodic $ \mathbf{z}(t) $ and Lemma~\ref{Lemma:1} is used to prove (\textit{i}). Statement (\textit{ii}) is then a direct consequence of the Liebniz rule.
		\end{proof}
	\end{thm}

	\subsection{Quasi-periodically forced systems}
	\label{subsec:qper_systems}
	The above classic results on periodic steady-state solutions extend
	to quasi-periodic steady-state solutions under quasi-periodic forcing.
	This observation does not appear to be available in the literature,
	which prompts us to provide full detail. 
	
	Let the forcing $\mathbf{F}(t)$ be quasi-periodic with frequency
	basis $\boldsymbol{\Omega}\in\mathbb{R}^{k}$, i.e., 
	\begin{equation}
	\mathbf{F}(t)=\sum_{\boldsymbol{\kappa}\in\mathbb{Z}^{k}}\mathbf{F}_{\boldsymbol{\kappa}}e^{i\left\langle \boldsymbol{\kappa},\boldsymbol{\Omega}\right\rangle t},\label{eq:QP_forcing}
	\end{equation}
	where each member of this $k$-parameter summation represents a time-periodic
	forcing with frequency $\left\langle \boldsymbol{\kappa},\boldsymbol{\Omega}\right\rangle $,
	i.e., forcing with period 
	\[
	T_{\boldsymbol{\kappa}}=\frac{2\pi}{\left\langle \boldsymbol{\kappa},\boldsymbol{\Omega}\right\rangle }.
	\]
	Here $T_{\mathbf{0}}=\infty$ formally corresponds to the period
	of the mean $\mathbf{F}_{\boldsymbol{0}}$ of $\mathbf{F}(t)$.
	\begin{lem}
		\label{Lemma:QP_linear} If the forcing is quasi-periodic, as given
		by \eqref{eq:QP_forcing}, then under the non-resonance condition
		\begin{equation}
		\lambda_{j}\neq i\frac{2\pi}{T_{\boldsymbol{\kappa}}}\ell,\qquad\ell\in\mathbb{Z},\quad j\in\{1,\dots,2n\},\quad\boldsymbol{\kappa}\in\mathbb{Z}^{k}\,,\label{eq:nonresonance-1-1}
		\end{equation}
		there exists a unique quasi-periodic steady-state response to \eqref{eq:firstorderlinear}
		with the same frequency basis $\boldsymbol{\Omega}$. This steady-state
		response is given by 
		\begin{align}
		\label{eq:IE_lin_qp}
		\mathbf{z}(t) & =\mathbf{V}\sum_{\boldsymbol{\kappa}\in\mathbb{Z}^{k}}\int_{0}^{T_{\boldsymbol{\kappa}}}\mathbf{G}(t-s,T_{\boldsymbol{\kappa}})\mathbf{V}^{-1}\mathbf{F}(s)\,ds.
		\end{align}
		Furthermore, $ \mathbf{z}(t)$ is quasi-periodic with Fourier expansion 
		\begin{equation}
		\mathbf{z}(t)=\mathbf{V}\sum_{\boldsymbol{\kappa}\in\mathbb{Z}^{k}}\mathbf{H}(T_{\boldsymbol{\kappa}})\mathbf{V}^{-1}\mathbf{F}_{\boldsymbol{\kappa}}e^{i\left\langle \boldsymbol{\kappa},\boldsymbol{\Omega}\right\rangle t}\,,
		\end{equation}
		where $\mathbf{H}(T_{\boldsymbol{\kappa}})$ is the diagonal matrix
		of the amplification factors, defined as 
		\[
		\mathbf{H}(T_{\boldsymbol{\kappa}})=\mathrm{diag}\left(H_{1}(T_{\boldsymbol{\kappa}}),\ldots,H_{2n}(T_{\boldsymbol{\kappa}})\right)\in\mathbb{C}^{2n\times2n},
		\]
		\begin{equation}
		H_{j}(t,T)=\frac{1}{i\left\langle \boldsymbol{\kappa},\boldsymbol{\Omega}\right\rangle -\lambda_{j}},\quad j=1,\dots,2n\,.\label{eq:ampl_fkt}
		\end{equation}
	\end{lem}
	\begin{proof}
		The proof is a consequence of the linearity of \eqref{eq:firstorderlinear} along with Lemma~\ref{Lemma:1}, followed by the explicit evaluation of the integrals in \eqref{eq:IE_lin_qp}. We give the details in Appendix~\ref{sec:ProofLemma2}.
	\end{proof}
	\begin{rem}
		{\small
			The maximum of $H_{j}(T_{\boldsymbol{\kappa}})$ can be bounded by
			the constant $h_{max}$, defined as
			\begin{equation}
			\begin{split}\max_{\begin{array}{c}
				1\leq j\leq2n\\
				\boldsymbol{\kappa}\in\mathbb{Z}^{k}
				\end{array}}\left|H_{j}(T_{\boldsymbol{\kappa}})\right| & =\max_{\begin{array}{c}
				1\leq j\leq2n\\
				\boldsymbol{\kappa}\in\mathbb{Z}^{k}
				\end{array}}\left|\frac{1}{i\left\langle \boldsymbol{\kappa},\boldsymbol{\Omega}\right\rangle -\lambda_{j}}\right|=\max_{\begin{array}{c}
				1\leq j\leq2n\\
				\boldsymbol{\kappa}\in\mathbb{Z}^{k}
				\end{array}}\frac{1}{(\left\langle \boldsymbol{\kappa},\boldsymbol{\Omega}\right\rangle -\mathrm{Im}(\lambda_{j}))^{2}+\mathrm{Re}(\lambda_{j})^{2}}\\
			& \leq\max_{1\leq j\leq2n}\frac{1}{\mathrm{Re}(\lambda_{j})^{2}}=\frac{1}{\min_{1\leq j\leq2n}\mathrm{Re}(\lambda_{j})^{2}}\eqqcolon h_{max}.
			\end{split}
			\label{eq:h_max}
			\end{equation}
		}
	\end{rem}
	
	In analogy with Theorem \ref{Theorem:IE_P}, we present here an integral
	formulation for steady-state solutions of the nonlinear system \eqref{eq:firstorder} under quasi-periodic forcing.
	\begin{thm}
		\label{Theorem:IE_QP}(i) If $\mathbf{z}(t)$ is a quasi-periodic
		solution of the nonlinear system \eqref{eq:firstorder} with frequency
		basis $\boldsymbol{\Omega}$, then the nonlinear function $\mathbf{R}(\mathbf{z}(t))$
		is also quasi-periodic with the same frequency basis $\boldsymbol{\Omega}$
		and $\mathbf{z}(t)$ must satisfy the integral equation: 
		\begin{align}
		\mathbf{z}(t) =\mathbf{V}\sum_{\boldsymbol{\kappa}\in\mathbb{Z}^{k}}\int_{0}^{T_{\boldsymbol{\kappa}}}\mathbf{G}(t-s,T_{\boldsymbol{\kappa}})\mathbf{V}^{-1}\left[\mathbf{F}(s)-\mathbf{R}(\mathbf{z}(s))\right]\,ds\,. \label{eq:int_quasiper}
		\end{align}
		
		(ii) Furthermore, any continuous, quasi-periodic solution $\mathbf{z}(t)$
		of \eqref{eq:int_quasiper}, with frequency basis $\boldsymbol{\Omega}$,
		is a quasi-periodic solution of the nonlinear system \eqref{eq:firstorder}.
		\begin{proof}
			The proof is analogous to that for the periodic case (cf. Theorem \ref{Theorem:IE_P}).
			Again, the term ${\mathbf{F}(t)-\mathbf{R}(\mathbf{z}(t))}$ is treated
			as a quasi-periodic forcing term.
		\end{proof}
	\end{thm}
	\vspace{3mm}
	\begin{rem}
		{\small 
			\label{rem:qp_fourier}
			With the Fourier expansion $ {\mathbf{z}(t) = \sum_{\boldsymbol{\kappa}\in\mathbb{Z}} \mathbf{z}_{\boldsymbol{\kappa}} e^{i\left\langle \boldsymbol{\kappa},\boldsymbol{\Omega}\right\rangle t} } $, eq.  \eqref{eq:int_quasiper} can be equivalently written as 
			\begin{equation}
			\mathbf{z}_{\boldsymbol{\kappa}} =\mathbf{V}\mathbf{H}(T_{\boldsymbol{\kappa}})\mathbf{V}^{-1}\left[\mathbf{F}_{\boldsymbol{\kappa}}-\mathbf{R}_{\boldsymbol{\kappa}}\{\mathbf{z}\}\right],\qquad {\boldsymbol{\kappa}\in\mathbb{Z}^{k}}\,,\label{eq:int_quasiper_explicit}
			\end{equation}
			where 	$ \mathbf{R}_{\boldsymbol{\kappa}}\{\mathbf{z}\} $  are the Fourier coefficients of the quasi-periodic function $ \mathbf{R}(\mathbf{z}(t)) $, defined as
			\begin{equation}
			\mathbf{R}_{\boldsymbol{\kappa}}\{\mathbf{z}\}:=\lim_{t\rightarrow\infty}\frac{1}{2t}\intop_{-t}^{t}\mathbf{R}(\mathbf{z}(t))e^{-i\left\langle \boldsymbol{\kappa},\boldsymbol{\Omega}\right\rangle t}dt.\label{eq:Rkappa_def}
			\end{equation}
			If we express the quasi-periodic solution using toroidal coordinates $ \boldsymbol{\theta} \in \mathbb{T}^{k}$ such that $ {\mathbf{z}(t) = \mathbf{u}(\mathbf{\Omega}t)} $, where $ {\mathbf{u}:\mathbb{T}^k\mapsto \mathbb{R}^{2n}} $ is the torus function, then we can express the Fourier coefficients as 
			\begin{equation}
			\mathbf{R}_{\boldsymbol{\kappa}}\{\mathbf{u}\}:=\frac{1}{(2\pi)^k}\int_{\mathbb{T}^k}\mathbf{R}(\mathbf{u}(\boldsymbol{\theta}))e^{-i\left\langle \boldsymbol{\kappa},\boldsymbol{\theta}\right\rangle t}d\boldsymbol{\theta}.\label{eq:Rkappa_tans}
			\end{equation}
			This helps to avoid the infinite limit in the integral \eqref{eq:Rkappa_def} that can pose numerical difficulties (cf. Schilder~et~al.~\cite{Schilder}, Mondelo González \cite{Mondelo}}.
	\end{rem}

	\subsection{Special case: Structural damping and purely geometric nonlinearities\label{subsec:Proportional_damping}}
	The results in Sections~\ref{subsec:Forced-linear-system}-\ref{subsec:qper_systems} apply to general first-order systems of the form \eqref{eq:firstorder}. The special case of second-order mechanical systems with proportional damping and purely geometric nonlinearities, however, is of significant interest to structural dynamicists (cf. G\'{e}rardin~\&~Rixen~\cite{Rixen}). These general results can be simplified for such systems, resulting in integral equations with half the dimensionality of equations \eqref{eq:int_periodic} and \eqref{eq:int_quasiper}, as we discuss in this section.
	
	We assume that the damping matrix $\mathbf{C}$ satisfies the proportional
	damping hypothesis, i.e., can be expressed as a linear combination
	of $\mathbf{M}$ and $\mathbf{K}$. We also assume that the nonlinearities
	depend on the positions only, i.e., we can simply write $\mathbf{S}(\mathbf{x})$.
	The equations of motion are, therefore, given by 
	
	\begin{equation}
	\mathbf{M}\ddot{\mathbf{x}}+\mathbf{C}\dot{\mathbf{x}}+\mathbf{Kx}+\mathbf{S}(\mathbf{x})=\mathbf{f}(t).\label{eq:eqm_posdep}
	\end{equation}
	Then, the real eigenvectors $\mathbf{u}_{j}$ of the undamped eigenvalue
	problem satisfy 
	\begin{equation}
	\left(\mathbf{K}-\omega_{0,j}^{2}\mathbf{M}\right)\mathbf{u}_{j}=\mathbf{0}\quad\left(j=1,2,\dots,n\right),
	\end{equation}
	where $\omega_{0,j}$ is the eigenfrequency of the undamped vibration
	mode $\mathbf{u}_{j}\in\mathbb{R}^{n}$. These eigenvectors (or \emph{modes})
	can be used to diagonalize the linear part of (\ref{eq:eqm_posdep}) using
	the linear transformation ${\mathbf{x}=\mathbf{U}\mathbf{y}}$, where
	${\mathbf{y}\in\mathbb{R}^{n}}$ represents the modal variables and
	${\mathbf{U}=\left[\mathbf{u}_{1},\ldots,\mathbf{u}_{n}\right]\in\mathbb{R}^{n\times n}}$
	is the modal transformation matrix containing the vibration modes.
	Thus, the decoupled system of equations for the linear system, 
	\begin{equation}
	\mathbf{M}\ddot{\mathbf{x}}+\mathbf{C}\dot{\mathbf{x}}+\mathbf{Kx}=\mathbf{f}(t),\label{eq:xeq}
	\end{equation}
	is given by
	\begin{equation}
	\mathbf{U}^{\top}\mathbf{MU}\ddot{\mathbf{y}}+\mathbf{U}^{\top}\mathbf{CU}\dot{\mathbf{y}}+\mathbf{U}^{\top}\mathbf{KUy}=\mathbf{U}^{\top}\mathbf{f}(t).\label{eq:forced modal-1}
	\end{equation}
	Specifically, the $j^{\text{th}}$mode $(y_{j})$ of equation \eqref{eq:forced modal-1}
	is customarily expressed in the vibrations literature as 
	\begin{equation}
	\ddot{y}_{j}+2\zeta_{j}\omega_{0,j}\dot{y}_{j}+\omega_{0,j}^{2}y_{j}=\varphi_{j}(t),\quad j=1,...,n,\label{eq:forced_modal_component-1}
	\end{equation}
	where $\omega_{0,j}=\sqrt{\frac{\mathbf{u}_{j}^{\top}\mathbf{K}\mathbf{u}_{j}}{\mathbf{u}_{j}^{\top}\mathbf{M}\mathbf{u}_{j}}}$
	are the undamped natural frequencies; $\zeta_{j}=\frac{1}{2\omega_{0,j}}\left(\frac{\mathbf{u}_{j}^{\top}\mathbf{C}\mathbf{u}_{j}}{\mathbf{u}_{j}^{\top}\mathbf{M}\mathbf{u}_{j}}\right)$
	are the modal damping coefficients; and $\varphi_{j}(t)=\left(\frac{\mathbf{u}_{j}^{\top}\mathbf{F}(t)}{\mathbf{u}_{j}^{\top}\mathbf{M}\mathbf{u}_{j}}\right)$
	are the modal participation factors. The eigenvalues for the corresponding
	full system in phase space can be arranged as follows 
	\begin{align*}
	\lambda_{2j-1,2j} & =\left(-\zeta_{j}\pm\sqrt{\zeta_{j}^{2}-1}\right)\omega_{0,j},\quad j=1,\dots,n
	\end{align*}
	With the constants
	\begin{align}
	\alpha_{j} & :=\text{Re}(\lambda_{2j}),\quad\omega_{j}:=|\text{Im}(\lambda_{2j})|,\quad\beta_{j}:=\alpha_{j}+\omega_{j},\quad\gamma_{j}:=\alpha_{j}-\omega_{j,}\quad j=1,\dots,n,\label{eq:alpha_beta}
	\end{align}
	we can restate Lemma \ref{Lemma:1} specifically for linear systems
	with proportional damping as follows. 
	\begin{lem}
		\label{Lemma:special_case_QP_linear}For $T$-periodic forcing $\mathbf{f}(t)$
		$\mathbf{f}(t+T)=\mathbf{f}(t),\:t\in\mathbb{R},\:T>0$
		and under the non-resonance conditions \eqref{eq:nonresonance-1}, there exists a unique T-periodic response for system \eqref{eq:xeq},
		given by 
		\begin{align}
		\mathbf{x}(t) & =\mathbf{U}\int_{0}^{T}\mathbf{L}(t-s,T)\mathbf{U}^{\top}\mathbf{f}(s)\,ds,
		\end{align}
		where $\mathbf{L}(t,T)$ is the diagonal Green's function matrix for
		the modal displacement variables defined as 
		\[
		\mathbf{L}(t,T)=\mathrm{diag}\left(L_{1}(t,T),\ldots,L_{n}(t,T)\right)\in\mathbb{R}^{n\times n},
		\]
		\begin{equation}
		L_{j}(t,T)=\begin{cases}
		\frac{e^{\alpha_{j}t}}{\omega_{j}}\left[\frac{e^{\alpha_{j}T}\left[\sin\omega_{j}(T+t)-e^{\alpha_{j}T}\sin\omega_{j}t\right]}{1+e^{2\alpha_{j}T}-2e^{\alpha_{j}T}\cos\omega_{j}T}+h(t)\sin\omega_{j}t\right], & \zeta_{j}<1\\
		\frac{e^{\alpha_{j}(T+t)}\left[\left(1-e^{\alpha_{j}T}\right)t+T\right]}{\left(1-e^{\alpha_{j}T}\right)^{2}}+h(t)te^{\alpha_{j}t}\,, & \zeta_{j}=1\\
		\frac{1}{(\beta_{j}-\gamma_{j})}\left[\frac{e^{\beta_{j}(T+t)}}{1-e^{\beta_{j}T}}-\frac{e^{\gamma_{j}(T+t)}}{1-e^{\gamma_{j}T}}+h(t)\left(e^{\beta_{j}t}-e^{\gamma_{j}t}\right)\right], & \zeta_{j}>1
		\end{cases},\quad j=1,\dots,n\label{eq:G_pos}
		\end{equation}
		and $\boldsymbol{\varphi}(s)=[\varphi_{1}(s),\dots,\varphi_{n}(s)]^{\top}$
		is the forcing vector in modal coordinates. 
		\begin{proof}
			See Appendix~\ref{sec:ProofLemma3}. 
		\end{proof}
	\end{lem}
	The periodic Green's function $L_{j}(t,T)$ for a single-degree-of-freedom,
	underdamped harmonic oscillator has already been derived in the controls
	literature (see, e.g., Kovaleva \cite{kovaleva99},~p.~19., formula~(1.40), or Babitsky~\cite{Babistky Vibro Impact}~p.~90). They also
	note a simplification when the periodic forcing function has an odd
	symmetry with respect to half the period (e.g., sinusoidal forcing),
	in which case the integral can be taken over just half the period
	with another Green's function. Kovaleva \cite{kovaleva99} also lists
	the Green function without damping for the case of a multi-degree-of-freedom
	system without damping, in transfer-function notation. In summary, formula \eqref{eq:G_pos} does not
	seem to appear in the vibrations literature, but earlier controls
	literature has simpler forms of it (single-degree-of-freedom modal
	form with damping, or multi-dimensional form without damping in modal
	coordinates), albeit for the under-damped case only. 
	
	Kovaleva~\cite{kovaleva99} also observes for undamped
	multi-degree-of-freedom systems that an integral equation with this Green's function can be written out for nonlinear systems, then refers
	to Rosenwasser~\cite{rosenwasse69} for existence conditions and approximate solution methods. Chapter~4.2 of Babitsky~and~Krupenin~\cite{babitsky12}
	also discusses this material in the context of the response of linear
	discontinuous systems, citing Rosenwasser~\cite{rosenwasse69} for
	a similar formulation. We formalize and generalize these discussions as a theorem here: 
	\begin{thm}
		\label{Thm:special_case_IE_P} (i) If $\mathbf{x}(t)$ is a $T-$periodic
		solution of the nonlinear system \eqref{eq:eqm_posdep}, then $\mathbf{x}(t)$
		must satisfy the integral equation 
		
		\begin{equation}
		\mathbf{x}(t)=\mathbf{U}\int_{0}^{T}\mathbf{L}(t-s,T)\mathbf{U}^{\top}\left[\mathbf{f}(s)-\mathbf{S}(\mathbf{x}(s))\right]\,ds\,,\label{eq:int_periodic-1}
		\end{equation}
		with $\mathbf{L}$ defined in \eqref{eq:G_pos}. 
		
		(ii) Furthermore, any continuous, T-periodic solution of $\mathbf{x}(t)$
		of \eqref{eq:int_periodic-1} is a $T-$periodic solution of the nonlinear
		system \eqref{eq:eqm_posdep}.
	\end{thm}
	\begin{proof}
		This result is just a special case of Theorem \ref{Theorem:IE_P},
		with the specific form of the Green's function listed in \eqref{eq:G_pos}.
	\end{proof}
	\begin{rem}
		{\small
			Once a solution to \eqref{eq:int_periodic-1} is obtained for the
			position variables $\mathbf{x}$ (cf. Section \ref{sec:Iterative_Solution}
			for solution methods), the corresponding velocity $\dot{\mathbf{x}}$
			can be recovered as 
			
			\[
			\dot{\mathbf{x}}(t)=\mathbf{U}\int_{0}^{T}\mathbf{J}(t-s,T)\mathbf{U}^{\top}\left[\mathbf{f}(s)-\mathbf{S}(\mathbf{x}(s))\right]\,ds,
			\]
			where ${\mathbf{J}(t-s,T)=\mathrm{diag}\left(J_{1}(t-s,T),\ldots,J_{n}(t-s,T)\right)\in\mathbb{R}^{n\times n}}$
			is the diagonal Green's matrix whose diagonal elements are given by  
			\begin{equation}
			\begin{split}J_{j}(t,T) & =\begin{cases}
			\begin{array}{c}
			\frac{e^{\alpha_{j}t}}{\omega_{j}}\left[\frac{e^{\alpha_{j}T}\left[\omega_{j}\left(\cos\omega_{j}(T+t)-e^{\alpha_{j}T}\cos\omega_{j}t\right)+\alpha_{j}\left(\sin\omega_{j}(T+t)-e^{\alpha_{j}T}\sin\omega_{j}t\right)\right]}{1+e^{2\alpha_{j}T}-2e^{\alpha_{j}T}\cos\omega_{j}T}+\right.\\
			\left.h(t)\left(\omega_{j}\cos\omega_{j}t+\alpha_{j}\sin\omega_{j}t\right)\right]
			\end{array}\quad, & \zeta_{j}<1\\
			\frac{e^{\alpha_{j}(T+t)}\left[\left(1-e^{\alpha_{j}T}\right)\left(1+\alpha_{j}t\right)+\alpha_{j}T\right]}{\left(1-e^{\alpha_{j}T}\right)^{2}}+h(t)\left(e^{\alpha_{j}t}+\alpha_{j}te^{\alpha_{j}t}\right)\,, & \zeta_{j}=1\\
			\frac{1}{(\beta_{j}-\gamma_{j})}\left[\frac{\beta_{j}e^{\beta_{j}(T+t)}}{1-e^{\beta_{j}T}}-\frac{\gamma_{j}e^{\gamma_{j}(T+t)}}{1-e^{\gamma_{j}T}}+h(t)\left(\beta_{j}e^{\beta_{j}t}-\gamma_{j}e^{\gamma_{j}t}\right)\right], & \zeta_{j}>1
			\end{cases}\end{split}
			\,,\label{eq:J_vel}
			\end{equation}
			as shown in Appendix~\ref{sec:ProofLemma3}.}
	\end{rem}
	
	Finally, the following result extends the integral equation formulation
	of Theorem \ref{Thm:special_case_IE_P} to quasi-periodic forcing.
	\begin{thm}
		\label{Thm:special_case_IE_QP}(i) If $\mathbf{x}(t)$ is a quasi-periodic
		solution of the nonlinear system \eqref{eq:eqm_posdep} with frequency
		basis $\boldsymbol{\Omega}$, and the nonlinear function $\mathbf{S}(\mathbf{x}(t))$
		is also quasi-periodic with the same frequency basis $\boldsymbol{\Omega}$,
		then $\mathbf{x}(t)$ must satisfy the integral equation: 
		
		\begin{align}
		\mathbf{x}(t) & =\mathbf{U}\sum_{\boldsymbol{\kappa}\in\mathbb{Z}^{k}}\int_{0}^{T_{\boldsymbol{\kappa}}}\mathbf{L}(t-s,T_{\boldsymbol{\kappa}})\mathbf{U}^{\top}\left[\mathbf{f}(s)-\mathbf{S}(\mathbf{x}(s))\right]\,ds\,\label{eq:int_quasiperiodic_posdep}
		\end{align}
		(ii) Furthermore, any continuous quasi-periodic solution $\mathbf{x}(t)$
		to \eqref{eq:int_quasiperiodic_posdep}, with frequency basis $\boldsymbol{\Omega}$,
		is a quasi-periodic solution of the nonlinear system \eqref{eq:0}.
	\end{thm}
	\begin{proof}
		This theorem is just a special case of Theorem \ref{Theorem:IE_QP}. 
	\end{proof}
	
	In analogy with Remark~\ref{rem:qp_fourier}, we make the following remark for geometric nonlinearities and structural damping.
	\begin{rem}
		{\small
			\label{rem:Fourier_qper_posdep}
			With the Fourier expansion $ {\mathbf{x}(t) = \sum_{\boldsymbol{\kappa}\in\mathbb{Z}} \mathbf{x}_{\boldsymbol{\kappa}} e^{i\left\langle \boldsymbol{\kappa},\boldsymbol{\Omega}\right\rangle t} } $, eq.~\eqref{eq:int_quasiperiodic_posdep} can be equivalently written as the system
			\begin{equation}
			\mathbf{x}_{\boldsymbol{\kappa}}=\mathbf{U}\mathbf{Q}(T_{\boldsymbol{\kappa}})\mathbf{U}^{\top}\left[\mathbf{f}_{\boldsymbol{\kappa}}-\mathbf{S}_{\boldsymbol{\kappa}}\{\mathbf{x}\}\right],\quad \boldsymbol{\kappa}\in\mathbb{Z}^{k}\,, \label{eq:int_explicit_qp_posdep}
			\end{equation}
			where 	\[
			\mathbf{Q}(T_{\boldsymbol{\kappa}})=\mathrm{diag}\left(Q_{1}(T_{\boldsymbol{\kappa}}),\ldots,Q_{n}(T_{\boldsymbol{\kappa}})\right)\in\mathbb{C}^{n\times n},
			\] is the diagonal matrix
			of the amplification factors, which are explicitly given by  
			
			\begin{equation}
			Q_{j}(T_{\boldsymbol{\kappa}}):=\begin{cases}
			\frac{1}{(i\left\langle \boldsymbol{\kappa},\boldsymbol{\Omega}\right\rangle -\alpha_{j})^{2}+\omega_{j}^{2}}\,, & \zeta_{j}<1\\
			\frac{1}{(i\left\langle \boldsymbol{\kappa},\boldsymbol{\Omega}\right\rangle -\alpha_{j})^{2}}\,, & \zeta_{j}=1\\
			\frac{1}{(\beta_{j}-i\left\langle \boldsymbol{\kappa},\boldsymbol{\Omega}\right\rangle )(\gamma_{j}-i\left\langle \boldsymbol{\kappa},\boldsymbol{\Omega}\right\rangle )}\,, & \zeta_{j}>1,
			\end{cases},\quad j=1,\dots,n\,,\label{eq:ampl_fkt_posdep}
			\end{equation}
			as derived in Appendix~\ref{sec:Proof4}.}
	\end{rem}
	
	\subsection{The unforced conservative case}
	In contrast to dissipative systems, which have isolated (quasi-)periodic
	solutions in response to (quasi-) periodic forcing, unforced conservative
	systems will generally exhibit families of periodic or quasi-periodic
	orbits (cf. Kelley~\cite{kelley69} or Arnold~\cite{arnold89}). The calculation of (quasi-) periodic orbits in an autonomous system such as \begin{equation}
	\mathbf{B}\dot{\mathbf{z}}=\mathbf{A}\mathbf{z}+\mathbf{R}(\mathbf{z})\,,\label{eq:unforced_first_order}
	\end{equation}
	is different
	from that in the forced case mainly due to two reasons:
	\begin{enumerate}
		\item The frequencies of such (quasi-) periodic oscillations are intrinsic
		to the system. This means that the time period $T$, or the base frequency
		vector $\boldsymbol{\Omega}$, of the response is a priori unknown. 
		\item Any given (quasi-) periodic solution $\mathbf{z}(t)$ to the autonomous system \eqref{eq:unforced_first_order} is a part of a family of (quasi-) periodic solutions, with an arbitrary phase shift $\theta\in\mathbb{R}$. 
	\end{enumerate}
	Nonetheless, Theorems \ref{Theorem:IE_P}-\ref{Thm:special_case_IE_QP}
	still hold for system \eqref{eq:unforced_first_order} with the external
	forcing function set to zero. Special care needs to be taken, however,
	in the numerical implementation of these results for unforced mechanical systems, as we shall discuss in Section \ref{subsec:UnforcedAutoCont}.
	
	\section{Iterative solution of the integral equations}
	
	\label{sec:Iterative_Solution}
	
	We would like to solve integral equations of the form (cf. Theorems
	\ref{Theorem:IE_P} and \ref{Thm:special_case_IE_P}) 
	\begin{equation}
	\mathbf{z}(t)=\int_{0}^{T}\mathbf{V}\mathbf{G}(t-s,T)\mathbf{V}^{-1}\left[\mathbf{F}(s)-\mathbf{R}(\mathbf{z}(s))\right]\,ds,\quad t\in[0,T]\label{eq:IE_periodic}
	\end{equation}
	to obtain periodic solutions, or integral equations of the form (cf.
	Theorem \ref{Theorem:IE_QP} and \ref{Thm:special_case_IE_QP})
	\begin{equation}
	\mathbf{z}(t)=\mathbf{V}\sum_{\boldsymbol{\kappa}\in\mathbb{Z}^{k}}\mathbf{H}(T_{\kappa})\mathbf{V}^{-1}\left(\mathbf{F}_{\boldsymbol{\kappa}}-\mathbf{R}_{\boldsymbol{\kappa}}\{\mathbf{z}\}\right)e^{i\left\langle \boldsymbol{\kappa},\boldsymbol{\Omega}\right\rangle t}\label{eq:IE_QP}
	\end{equation}
	to obtain quasi-periodic solutions of system \eqref{eq:firstorder}.  In the following, we propose iterative methods to solve these equations. First, we discuss a Picard iteration, then subsequently, a Newton-Raphson scheme.

	\subsection{Picard iteration\label{sec:Piccard}}
	
	Picard~\cite{Picard} proposed an iteration scheme to show local existence of solutions to ordinary
	differential equations, which is also used as practical iteration
	scheme to approximate the solutions to boundary value problems in
	numerical analysis (cf. Bailey~et~al.~\cite{Bailey_NumBVP}). We derive
	explicit conditions on the convergence of the Picard iteration when applied to equations~\eqref{eq:IE_periodic}-\eqref{eq:IE_QP}.
	\subsubsection{Periodic response\label{subsec:Picard_periodic}}
	
	We define the right-hand side of the integral equation \eqref{eq:int_periodic}
	as the mapping~$\boldsymbol{\mathcal{G}}_{P}$ acting on the phase
	space vector $\mathbf{z}$, i.e.,
	
	\begin{equation}
	\mathbf{z}(t)=\mathbf{\boldsymbol{\mathcal{G}}}_{P}(\mathbf{z})\,(t):=\int_{0}^{T}\mathbf{V}\mathbf{G}(t-s,T)\mathbf{V}^{-1}\left[\mathbf{F}(s)-\mathbf{R}(\mathbf{z}(s))\right]\,ds,\quad t\in[0,T].\label{eq:Map_periodic}
	\end{equation}
	Clearly, a fixed point of the mapping $\mathbf{\boldsymbol{\mathcal{G}}}_{P}$
	in \eqref{eq:Map_periodic} corresponds to a periodic steady-state
	response of system \eqref{eq:0} by Theorem~\ref{Theorem:IE_P}.
	Starting with an initial guess $\mathbf{z}_{0}(t)$ for the periodic
	orbit, the Picard iteration applied to the mapping \eqref{eq:IE_periodic}
	is given by 
	\begin{equation}
	\mathbf{z}_{\ell+1}=\mathbf{\boldsymbol{\mathcal{G}}}_{P}(\mathbf{z}_{\ell})\,,\quad \ell \in \mathbb{N}.\label{eq:Picard_P}
	\end{equation}
	To derive a convergence criterion for the Picard iteration, we define the
	sup norm ${\left\Vert \cdot\right\Vert _{0}=\max_{t\in[0,T]}\left|\cdot\right|}$
	and consider a $\delta-$ball of $C^{0}$-continuous and $T$-periodic
	functions centered at $\mathbf{z}_{0}$:
	\begin{equation}
	C_{\delta}^{\mathbf{z}_{0}}[0,T]:=\left\{ \mathbf{z}\colon[0,T]\to\mathbb{R}^{2n}\,\,\vert\quad\mathbf{z}\in C^{0}[0,T],\quad\mathbf{z}(0)=\mathbf{z}(T),\quad\left\Vert \mathbf{z}-\mathbf{z}_{0}\right\Vert _{0}\leq\delta\right\} .\label{eq:Space_per_fcn}
	\end{equation}
	We further define the first iterate under the map $ \boldsymbol{\mathcal{G}}_{P} $ as
	\begin{equation}
	\boldsymbol{\mathcal{E}}(t)=\mathcal{\boldsymbol{G}}_{P}(\mathbf{z}_{0})(t) =\int_{0}^{T}\mathbf{V}\mathbf{G}(t-s,T)\mathbf{V}^{-1}\left[\mathbf{F}(s)-\mathbf{R}(\mathbf{z}_{0}(s))\right]\,ds,\quad t\in[0,T],\label{eq:initial_error}
	\end{equation}
	and denote with $L_{\delta}^{\mathbf{z}_{0}}$ a uniform-in-time Lipschitz
	constant for the nonlinearity $\mathbf{R}(\mathbf{z})$ with respect
	to its argument $\mathbf{z}$ within $C_{\delta}^{\mathbf{z}_{0}}[0,T]$.
	With that notation, we obtain the following theorem for the convergence
	of a Picard iteration performed on \eqref{eq:IE_periodic}
	\begin{thm}
		\label{Thm:Picard_it_periodic}If the conditions 
		\begin{eqnarray}
		L_{\delta}^{\mathbf{z}_{0}} & < & \frac{1}{a\left\Vert \mathbf{V}\right\Vert \left\Vert \mathbf{V}^{-1}\right\Vert \Gamma(T)}\,,\label{eq:Lipschitz_cond}\\
		\delta & \geq & \frac{\left\Vert \boldsymbol{\mathcal{E}}\right\Vert _{0}}{1-\left\Vert \mathbf{V}\right\Vert \left\Vert \mathbf{V}^{-1}\right\Vert L_{\delta}^{\mathbf{z}_{0}}\Gamma(T)}\,,\label{eq:delta_cond}
		\end{eqnarray}
		hold for some real number $a\geq1$, then the mapping $\boldsymbol{\mathcal{G}}_{P}$
		defined in equation~\eqref{eq:IE_periodic} has a unique fixed point
		in the space~\eqref{eq:Space_per_fcn} and this fixed point can be
		found via the successive approximation
		\begin{equation}
		\mathbf{z}_{\ell+1}(t)=\mathcal{\boldsymbol{\mathcal{G}}}_{P}(\mathbf{z}_{\ell})\,(t)=\int_{0}^{T}\mathbf{V}\mathbf{G}(t-s,T)\mathbf{V}^{-1}\left[\mathbf{F}(s)-\mathbf{R}(\mathbf{z}_{\ell}(s))\right]\,ds,\quad \ell\in\mathbb{N}\label{eq:Picard_it_periodic}
		\end{equation}
	\end{thm}
	\begin{proof}
		The proof relies on the Banach fixed point theorem. We establish that
		the mapping~\eqref{eq:IE_periodic} is well-defined on the space~\eqref{eq:Space_per_fcn}.
		Subsequently, we prove that under conditions~\eqref{eq:Lipschitz_cond}-\eqref{eq:delta_cond}, the mapping~\eqref{eq:IE_periodic} is
		a contraction. We detail all this in Appendix \ref{App:Proof_Picard_it_per}. 
	\end{proof}
	\begin{rem}
		{\small	
			\label{Rmk:Picard_C1_NLs}If the nonlinearity $\mathbf{R}(\mathbf{z})$
			is not only Lipschitz but also of class $C^{1}$ with respect to $\mathbf{z}$,
			then condition~\eqref{eq:Lipschitz_cond} can be more specifically written as
			\begin{equation}
			\max_{1\leq j\leq n\,\,\,}\max_{\left|\mathbf{z}-\mathbf{z}_{0}\right|\leq\delta}\left|DR_{j}(\mathbf{z})\right|<\frac{1}{a\varGamma(T)\left\Vert \mathbf{V}\right\Vert \left\Vert \mathbf{V}^{-1}\right\Vert }.\label{eq:Lipschitz_cond_smooth}
			\end{equation}}
	\end{rem}
	
	\begin{rem}
		{\small
			Conditions \eqref{eq:Lipschitz_cond} and \eqref{eq:delta_cond}
			can be merged by rewriting condition \eqref{eq:delta_cond} as 
			\[
			\left\Vert \boldsymbol{\mathcal{E}}\right\Vert _{0}\leq\delta\left[1-\left\Vert \mathbf{V}\right\Vert \left\Vert \mathbf{V}^{-1}\right\Vert L_{\delta}^{\mathbf{z}_{0}}\Gamma(T)\right],
			\]
			which, by condition \eqref{eq:Lipschitz_cond}, holds for sure if
			$a>1$ and $\left\Vert \boldsymbol{\mathcal{E}}\right\Vert _{0}\leq\delta\left(1-\frac{1}{a}\right),$
			or, equivalently, 
			\begin{equation}
			\delta\geq\frac{a\left\Vert \boldsymbol{\mathcal{E}}\right\Vert _{0}}{a-1}.\label{eq:fincon2}
			\end{equation}
			Inserting the condition \eqref{eq:fincon2} for $\delta$ into the
			condition \eqref{eq:Lipschitz_cond} on the Lipschitz constant $L_{\delta}^{\mathbf{z}_{0}}$,
			we obtain the single convergence condition
			
			\begin{equation}
			L_{\frac{a\left\Vert \boldsymbol{\mathcal{E}}\right\Vert _{0}}{a-1}}^{\mathbf{z}_{0}}<\frac{1}{a\varGamma(T)\left\Vert \mathbf{V}\right\Vert \left\Vert \mathbf{V}^{-1}\right\Vert }.\label{eq:single_cond}
			\end{equation}
			For differentiable nonlinearities (cf. Remark \eqref{Rmk:Picard_C1_NLs}),
			the equivalent condition is then given by
			\begin{equation}
			\max_{1\leq j\leq n\,\,\,}\max_{\left|\mathbf{z}-\mathbf{z}_{0}\right|\leq\frac{a\left\Vert \boldsymbol{\mathcal{E}}\right\Vert _{0}}{a-1}}\left|D_{\mathbf{z}}R_{j}(\mathbf{z})\right|<\frac{1}{a\varGamma(T)\left\Vert \mathbf{V}\right\Vert \left\Vert \mathbf{V}^{-1}\right\Vert }.\label{eq:finfincond}
			\end{equation}}
	\end{rem}
	\begin{rem}
		{\small\label{rmk:Picard_per_reduced}In case of geometric (purely position-dependent)
			nonlinearities and proportional damping (cf. Section \ref{subsec:Proportional_damping}),
			we can avoid iterating in the $2n-$dimensional phase space by defining
			the iteration as 
			
			\begin{equation}
			\mathbf{x}_{\ell+1}(t)=\mathcal{\boldsymbol{\mathcal{L}}}_{P}(\mathbf{x}_{\ell})\,(t):=\int_{0}^{T}\mathbf{U}\mathbf{L}(t-s,T)\mathbf{U}^{\top}\left[\mathbf{f}(s)-\mathbf{S}(\mathbf{x})\right]\,ds,\quad t\in[0,T].\label{eq:Picard_per_pos_only}
			\end{equation}
			The existence of the steady state solution and the convergence of
			the iteration \eqref{eq:Picard_per_pos_only} can be proven analogously.} 
	\end{rem}
	Babistky \cite{Babistky Vibro Impact} derives via transfer functions
	an iteration similar to \eqref{eq:Picard_it_periodic} but without an
	explicit convergence proof. He asserts that the iteration is sensitive
	to the choice of the initial conditions $\mathbf{z}_{0}$. We can
	directly confirm this by examining condition~\eqref{eq:delta_cond}.
	Indeed, the norm of the initial error~$\left\Vert \boldsymbol{\mathcal{E}}\right\Vert _{0}$ is small for a good initial guess. Therefore, the $\delta$-ball in
	which the condition~\eqref{eq:Lipschitz_cond} on the Lipschitz constant
	needs to be satisfied can be selected small. 
	
	When no a priori information about the expected steady-state response
	is available, we can select $\mathbf{z}_{0}(t)\equiv\mathbf{0}$.
	Then, the term $\boldsymbol{\mathcal{E}}(0,t)$ is equal to the forced
	response of the linear system (cf. eq. \eqref{eq:initial_error}).
	In this case, the Lipschitz constant needs to be calculated for a
	$\delta-$ball centered at the origin. 
	
	On the other hand, if one supplies an actual $T-$periodic solution $\hat{\mathbf{z}}(t)=\hat{\mathbf{z}}(t+T)$
	of system~\eqref{eq:0}, the condition~\eqref{eq:delta_cond} is trivially
	satisfied for $\delta>0$ and the iteration~\eqref{eq:Picard_it_periodic}
	converges to the unique fixed point, if the single condition 
	\begin{equation}
	L_{\delta}^{\hat{\mathbf{z}}}<\frac{1}{a\left\Vert \mathbf{V}\right\Vert \left\Vert \mathbf{V}^{-1}\right\Vert \Gamma(T)},\qquad a\in\mathbb{R},\:a>1,\label{eq:cond_conv_zhat}
	\end{equation}
	on the Lipschitz constant holds. In the case of differentiable nonlinearities,
	we can rewrite condition~\eqref{eq:cond_conv_zhat} as 
	\[
	\max_{1\leq j\leq n\,\,\,}\max_{0\leq t<T}\left|DR_{j}(\hat{\mathbf{z}}(t))\right|<\frac{1}{a\left\Vert \mathbf{V}\right\Vert \left\Vert \mathbf{V}^{-1}\right\Vert \Gamma(T)},\qquad a\in\mathbb{R},\:a>1.
	\]
	This implies that whenever condition~\eqref{eq:cond_conv_zhat}
	is satisfied, there exists an initial guess $\mathbf{z}_{0}$ such
	that the iteration~\eqref{eq:Picard_it_periodic} converges. However,
	the radius $\delta$ of the ball~\eqref{eq:Space_per_fcn} might be
	small. 
	
	The constant $\varGamma(T)$ (cf. eq.~\eqref{eq:GammaT}) affects
	the convergence of the iteration~\eqref{eq:Picard_it_periodic}. Larger
	damping (i.e., smaller $e^{\mbox{Re}(\lambda_{j})T})$, larger distance
	of the forcing frequency $2\pi/T$ from the natural frequencies (i.e.,
	larger $|1-e^{\lambda_{j}T}|$), and higher forcing frequencies (i.e.,
	smaller $T$) all make the right-hand side of \eqref{eq:finfincond}
	larger and hence are beneficial to the convergence of the iteration.
	Likewise, a good initial guess (i.e., smaller $\left\Vert \boldsymbol{\mathcal{E}}\right\Vert _{0})$
	and smaller nonlinearities (i.e., smaller $\left|DS_{j}(\mathbf{x})\right|$
	) all make the left-hand side of \eqref{eq:finfincond} smaller and
	hence are similarly beneficial to the convergence of the iteration.
	In the context of structural vibrations, higher frequencies, smaller
	forcing amplitudes, and forcing frequencies sufficiently separated
	from the natural frequencies of the system are realistic and affect
	the convergence positively. At the same time, low damping values in
	such systems are also typical and affect the convergence negatively. 
	
	An advantage of the Picard iteration approach we have discussed is that it
	converges monotonically, and hence an upper estimate for the error
	after a finite number of iterations is readily available as the sup
	norm of the difference of the last two iterations. This can be exploited
	in numerical schemes to stop the iteration once the required precision
	is achieved.
	
	\subsubsection{Quasi-periodic response}
	\label{sec:qp_picard}
	We now consider the existence of a quasi-periodic solution under a Picard iteration of equation~\eqref{eq:int_quasiper}, which has
	apparently been completely absent in the literature. We rewrite
	the right-hand side of the integral equation~\eqref{eq:int_quasiper}
	as the mapping 
	\begin{equation}
	\mathbf{z}(t)=\boldsymbol{\mathcal{G}}_{Q}(\mathbf{z})\,(t):=\mathbf{V}\sum_{\boldsymbol{\kappa}\in\mathbb{Z}^{k}}\mathbf{H}(T_{\kappa})\mathbf{V}^{-1}\left(\mathbf{F}_{\boldsymbol{\kappa}}-\mathbf{R}_{\boldsymbol{\kappa}}\{\mathbf{z}\}\right)e^{i\left\langle \boldsymbol{\kappa},\boldsymbol{\Omega}\right\rangle t},\label{eq:Map_quasiper}
	\end{equation}
	where we have made use of the Fourier expansion defined in Remark~\ref{rem:qp_fourier}. 
	
	We consider a space of quasi-periodic functions with the frequency
	base $\boldsymbol{\Omega}$. Similarly to the periodic case (cf. section
	\eqref{subsec:Picard_periodic}), we restrict the iteration to a $\delta-$ball
	$C_{\delta}^{\mathbf{z}_{0}}\left(\boldsymbol{\Omega}\right)$ centered
	at the inital guess $\mathbf{z}_{0}$ with radius $\delta$, i.e., 
	
	\begin{equation}
	C_{\delta}^{\mathbf{z}_{0}}\left(\boldsymbol{\Omega}\right):=\left\{ \mathbf{z}(\boldsymbol{\theta}):\:\mathbb{T}^{k}\to\mathbb{R}^{2n}\;\vert\quad\mathbf{z}\in C^{0},\quad\left\Vert \mathbf{z}-\mathbf{z}_{0}\right\Vert _{0}\leq\delta\right\} ,\label{eq:Space_qper_funct}
	\end{equation}
	where the sup norm $\left\Vert \cdot\right\Vert _{0}=\max_{\boldsymbol{\theta}\in\mathbb{T}^k}\left|\cdot\right|$
	is the uniform supremum norm over the torus $ \mathbb{T}^k $. We then have the following theorem.
	\begin{thm}
		\label{Thm:Picard_it_qper} If the conditions 
		\begin{eqnarray}
		L_{\delta}^{\mathbf{z}_{0}} & < & \frac{1}{a\left\Vert \mathbf{V}\right\Vert \left\Vert \mathbf{V}^{-1}\right\Vert h_{max}}\,,\label{eq:Lipschitz_cond_qper}\\
		\delta & \geq & \frac{\left\Vert \boldsymbol{\mathcal{E}}\right\Vert _{0}}{1-2\left\Vert \mathbf{V}\right\Vert \left\Vert \mathbf{V}^{-1}\right\Vert L_{\delta}^{\mathbf{z}_{0}}h_{max}}\,,\label{eq:delta_cond_qper}
		\end{eqnarray}
		hold for some real number $a\geq1$, then the mapping $\boldsymbol{\mathcal{G}}_{Q}$
		defined in equation \eqref{eq:Map_quasiper} has a unique fixed point
		in the space \eqref{eq:Space_qper_funct} and this fixed point can
		be found via the successive approximation
		\begin{equation}
		\mathbf{z}_{\ell+1}(t)=\boldsymbol{\mathcal{G}}_{Q}(\mathbf{z}_{\ell})\,(t),\quad \ell\in\mathbb{N}.\label{eq:Picard_it_qper}
		\end{equation}
	\end{thm}
	\begin{proof}
		The is analogous to the proof of Theorem \ref{Thm:Picard_it_periodic}.
		We first establish that the mapping \eqref{eq:Map_quasiper} is well-defined in the space \eqref{eq:Space_qper_funct}. In Appendix \ref{App:Proof_Picard_it_qper},
		we detail that the mapping \eqref{eq:Map_quasiper} is a contraction
		under the conditions \eqref{eq:Lipschitz_cond_qper} and \eqref{eq:delta_cond_qper}
		. 
	\end{proof}
	\begin{rem}
		{\small In case of geometric (position-dependent) nonlinearities and proportional
			damping, we can reduce the dimensionalilty of the iteration \eqref{eq:Picard_it_qper}
			by half, using \eqref{eq:int_explicit_qp_posdep}. This results in
			the following, equivalent Picard iteration:
			\begin{equation}
			\mathbf{x}_{\ell+1}(t)=\boldsymbol{\mathcal{L}}_{Q}(\mathbf{x}_{\ell})\,(t):=\mathbf{U}\sum_{\boldsymbol{\kappa}\in\mathbb{Z}^{k}}\mathbf{Q}(T_{\boldsymbol{\kappa}})\mathbf{U}^{\top}\left[\mathbf{f}_{\boldsymbol{\kappa}}-\mathbf{S}_{\boldsymbol{\kappa}}\{\mathbf{x}_{\ell}\}\right]e^{i\left\langle \boldsymbol{\kappa},\boldsymbol{\Omega}\right\rangle t}\,,\quad \ell\in\mathbb{N}.\label{eq:Picard_qper_pos_only}
			\end{equation}
			The existence of the steady state solution and the convergence of
			the iteration \eqref{eq:Picard_qper_pos_only} can be proven analogously. }
	\end{rem}
	As in the periodic case, the convergence of the iteration \eqref{eq:Picard_it_qper}
	depends on the quality of the initial guess and the constant $h_{max}$
	(cf. eq. \eqref{eq:h_max}), which is the maximum amplification factor.
	Low damping results in a higher amplification factor (cf. eq. \eqref{eq:h_max})
	and will therefore affect the iteration negatively, which is similar
	to the criterion derived in the periodic case. 
	
	\subsubsection{Unforced conservative case}
	
	In the unforced conservative case, $\mathbf{z}(t)\equiv\mathbf{0}$ is
	the trivial fixed point of the maps $\boldsymbol{\mathcal{G}}_{P}$
	and $\boldsymbol{\mathcal{G}}_{Q}$. Thus, by Theorems \ref{Thm:Picard_it_periodic}~and~\ref{Thm:Picard_it_qper}, the Picard iteration with an initial guess
	in the vicinity of the origin would make the iteration converge to
	the trivial fixed point. In practice, the simple Picard approach is
	found to be highly sensitive to the choice of initial guess for obtaining
	non-trivial solution in the case of unforced conservative systems.
	Thus, in such cases, more advanced iterative schemes equipped with
	continuation algorithms are desirable, such as the ones we describe
	next. 
	
	\subsection{Newton\textendash Raphson Iteration\label{sec:Newton-Iteration}}
	
	So far, we have described a fast iteration process and gave bounds
	on the expected convergence region of this iteration. We concluded
	that if the iteration converges, it leads to the unique \linebreak(quasi-) periodic
	solution of the system \eqref{eq:0}. As discussed previously (cf.
	Section~\ref{sec:Piccard}), our convergence criteria for the Picard iteration will not be satisfied for near-resonant forcing and low damping. However, even if the Picard iteration fails to converge, one or more periodic orbits may still exist. 
	
	A common alternative to the contraction mapping approach proposed
	above is the Newton\textendash Raphson scheme (cf., e.g., Kelley \cite{Newton}
	). An advantage of this iteration is its quadratic convergence if
	the initial guess is close enough to the actual solution of the problem.
	This makes this procedure also appealing for a continuation set-up.
	We first derive the Newton\textendash Raphson scheme to periodically
	forced systems and afterwards to quasi-periodically forced systems. 
	
	\subsubsection{Periodic case}
	
	To set up a Newton\textendash Raphson iteration, we reformulate the
	fixed-point problem \eqref{eq:IE_periodic} with the help of a functional
	$\boldsymbol{\mathcal{F}}_{P}$ whose zeros need to be determined:
	
	\begin{equation}
	\boldsymbol{\mathcal{F}}_{P}(\mathbf{z}):=\mathbf{z}-\boldsymbol{\mathcal{G}}_{P}(\mathbf{z})=\mathbf{z}-\int_{0}^{T}\mathbf{V}\mathbf{G}(t-s,T)\mathbf{V}^{-1}\left[\mathbf{F}(s)-\mathbf{R}(\mathbf{z}(s))\right]\,ds=\mathbf{0}.\label{eq:functional_periodic}
	\end{equation}
	Starting with an initial solution guess $\mathbf{z}_{0}$, we formulate
	the iteration for the zero of the functional $\boldsymbol{\mathcal{F}}_{P}$
	as 
	
	\begin{equation}
	\begin{split}\mathbf{z}_{l+1} & =\mathbf{z}_{l}+\boldsymbol{\mu}_{l},\\
	-\boldsymbol{\mathcal{F}}_{P}(\mathbf{z}_{l}) & =D\boldsymbol{\mathcal{F}}_{P}(\mathbf{z}_{l})\boldsymbol{\mu}_{l},\quad l\in\mathbb{N}\,,
	\end{split}
	\label{eq:NR_periodic}
	\end{equation}
	where the second equation in \eqref{eq:NR_periodic} can be written
	using the Gateaux derivative of $ \boldsymbol{\mathcal{F}}_{P} $ of
	\begin{equation}
	-\boldsymbol{\mathcal{F}}_{P}(\mathbf{z}_{l})=D\boldsymbol{\mathcal{F}}_{P}(\mathbf{z}_{l})\boldsymbol{\mu}_{l}=\left.\frac{d\mathcal{\boldsymbol{\mathcal{F}}}_{P}(\mathbf{z}_{l}+s\boldsymbol{\mu}_{l})}{ds}\right|_{s=0}=\boldsymbol{\mu}_{l}+\int_{0}^{T}\mathbf{V}\mathbf{G}(t-s,T)\mathbf{V}^{-1}\left.D\mathbf{R}\left(\mathbf{z}(s)\right)\right|_{\mathbf{z}=\mathbf{z}_{l}}\boldsymbol{\mu}_{l}(s)ds~.\label{eq:to_be_inverted}
	\end{equation}
	Equation \eqref{eq:to_be_inverted} is a linear integral equation
	in $\boldsymbol{\mu}_{l}$, where $\mathbf{z}_{l}$ is the known approximation to the solution of \eqref{eq:functional_periodic} at the $l^{\text{th}}$ iteration step. 
	
	\subsubsection{Quasi-periodic case\label{subsec:NR_qper}}
	
	In the quasi-periodic case, the steady-state solution of system \eqref{eq:0}
	is given by the zeros of the functional 
	\begin{equation}
	\mathcal{\boldsymbol{\mathcal{F}}}_{Q}(\mathbf{z}):=\mathbf{z}-\boldsymbol{\mathcal{G}}_{Q}(\mathbf{z})=\mathbf{z}-\sum_{\boldsymbol{\kappa}\in\mathbb{Z}^{k}}\mathbf{V}\mathbf{H}(T_{\boldsymbol{\kappa}})\mathbf{V}^{-1}\left(\mathbf{F_{\boldsymbol{\kappa}}}-\mathbf{R_{\boldsymbol{\kappa}}}\{\mathbf{z}_{l}\}\right)e^{i\left\langle \boldsymbol{\kappa},\boldsymbol{\Omega}\right\rangle t}\,.\label{eq:Functional_quasiperiodic}
	\end{equation}
	Analogous to the periodic case, the Newton\textendash Raphson scheme
	seeks to find a zero of $\mathcal{\boldsymbol{\mathcal{F}}}_{Q}$ via the iteration:
	\begin{equation}
	\begin{split}\mathbf{z}_{l+1} & =\mathbf{z}_{l}+\mathbf{\boldsymbol{\nu}}_{l},\\
	-\boldsymbol{\mathcal{F}}_{Q}(\mathbf{z}_{l}) & =D\boldsymbol{\mathcal{F}}_{Q}(\mathbf{z}_{l})\mathbf{\boldsymbol{\nu}}_{l}.
	\end{split}
	\label{eq:NR_qper}
	\end{equation}
	To obtain the correction step $\mathbf{\boldsymbol{\nu}}_{l}$, the
	linear system of equations 
	\begin{equation}
	-\boldsymbol{\mathcal{F}}_{Q}(\mathbf{z}_{l})=D\boldsymbol{\mathcal{F}}_{Q}(\mathbf{z}_{l})\mathbf{\boldsymbol{\nu}}_{l}=\mathbf{\boldsymbol{\nu}}_{l}+\mathbf{V}\sum_{\boldsymbol{\kappa}\in\mathbb{Z}^{k}}\mathbf{H}(T_{\boldsymbol{\kappa}})\mathbf{V}^{-1}\left(\mathbf{F_{\boldsymbol{\kappa}}}-\{D\mathbf{R}(\mathbf{z}_{l})\mathbf{\boldsymbol{\nu}}_{l}\}_{\boldsymbol{\kappa}}\right)e^{i\left\langle \boldsymbol{\kappa},\boldsymbol{\Omega}\right\rangle t}\label{eq:to_be_inverted_quasiperiodic}
	\end{equation}
	needs to be solved for $\mathbf{\boldsymbol{\nu}}_{l}$. The Fourier
	coefficients $\{D\mathbf{R}(\mathbf{z}_{l})\mathbf{\boldsymbol{\nu}}_{l}\}_{\boldsymbol{\kappa}}$
	in \eqref{eq:to_be_inverted_quasiperiodic} are then given by the formula
	\[
	\{D\mathbf{R}(\mathbf{z}_{l}(t))\mathbf{\boldsymbol{\nu}}_{l}\}_{\boldsymbol{\kappa}}=\lim_{\tau\rightarrow\infty}\frac{1}{2\tau}\int_{-\tau}^{\tau}D\mathbf{R}(\mathbf{z}(t))\boldsymbol{\nu}_{l}(t)e^{-i\langle\boldsymbol{\kappa},\boldsymbol{\Omega}\rangle t}dt\,.
	\]
	\begin{rem}
		{\small
			As noted in the Introduction, the results in Sections \ref{sec:qp_picard} and \ref{subsec:NR_qper} hold without any restriction on the number of independent frequencies allowed in the forcing function $ \mathbf{f}(t) $. This enables us to compute the steady-state response for arbitrarily complicated forcing functions, as long as they are well approximated by a finite Fourier expansion. Thus, the treatment of random-like steady state computations is possible with the methods proposed here. 	}
	\end{rem}
	\subsection*{Discussion of the iteration techniques}
	The Newton--Raphson iteration offers an alternative to the Picard iteration, especially
	when the system is forced near resonance, and hence the convergence of
	the Picard iteration cannot be guaranteed. At the same time, the Newton\textendash Raphson
	iteration is computationally more expensive than the Picard iteration
	for two reasons: First, the evaluation of the Gateaux derivatives
	\eqref{eq:to_be_inverted} and \eqref{eq:to_be_inverted_quasiperiodic}
	can be expensive, especially if the Jacobian of the nonlinearity is
	not directly available. Second, the correction step $\boldsymbol{\mu}_{l}$
	or $\boldsymbol{\nu}_{l}$ at each iteration involves the inversion
	of the corresponding linear operator in equation \eqref{eq:to_be_inverted}
	or \eqref{eq:to_be_inverted_quasiperiodic}, which is costly for large
	systems.
	
	Regarding the first issue above, the tangent stiffness is often available
	in finite-element codes for structural vibration. Nonetheless, there
	are many quasi-Newton schemes in the literature that circumvent this
	issue by offering either a cost-effective but inaccurate approximation
	of the Jacobian (e.g., the Broyden's method \cite{Broyden}), or avoid
	the calculation of the Jacobian altogether (cf. Kelley~\cite{Newton}). 
	
	For the second challenge above, one can opt for the iterative solution
	of the linear system \eqref{eq:to_be_inverted}~or~\eqref{eq:to_be_inverted_quasiperiodic},
	which would circumvent operator inversion when the system size is
	very large. A practical strategy for obtaining force response curves of high-dimensional systems would be to use the Picard iteration away from resonant forcing and switch towards the Newton\textendash Raphson approach when the Picard iteration fails. Even though the Newton\textendash Raphson approach has better rate of convergence (quadratic) as compared to the Picard approach (linear), the computational complexity of a single Picard iteration is an order of magnitude lower than that of Newton--Raphson method (simply because of the additional costs of Jacobian evaluation and inversion involved in the correction step). In our experience with high-dimensional systems, when the Picard iteration converges, it is significantly faster than the Newton\textendash Raphson in terms of CPU time, even though it takes significantly more number of iterations to converge.
	
	\section{Numerical solution procedure\label{subsec:NR_int_disc}}
	
	Note that $\mathbf{z}(t)$ in eqs. \eqref{eq:IE_periodic} and \eqref{eq:IE_QP}
	is a continuous function of time that cannot generally be obtained
	in a closed form. It must therefore be approximated via the finite
	sum 
	\[
	\mathbf{z}_{m}(t)=\sum_{j=1}^{m}\mathbf{c}_{mj}b_{mj}(t),
	\]
	where $\mathbf{c}_{mj}$ is the unknown coefficient attached to the
	chosen basis function $b_{mj}(t)$. The basic idea of all related
	numerical methods is to project the solution onto a suitable finite-dimensional
	subspace to facilitate numerical computations. These projection methods
	can be broadly divided into the categories of collocation and Galerkin
	methods (cf. Kress \cite{Kress}). 
	
	If the basis functions $b_{mj}$ are chosen to perform a collocation-type
	approximation, the integral equation will be guaranteed to be satisfied
	at a finite number of \emph{collocation points}. Specifically, if
	$m$ collocation points $t_{1}^{(m)},t_{2}^{(m)}\dots,t_{m}^{(m)}\in[0,T]$
	are chosen, the integral equation is reduced to solving a finite-dimensional
	system of (nonlinear) algebraic equations in the coefficients $\mathbf{c}_{mj}$:
	\begin{equation}
	\mathbf{z}_{m}(t_{j})=\boldsymbol{\mathcal{G}}_{P}\left(\mathbf{z}_{m}\right)\,(t_{j}),\quad j=1,\dots,m.\label{eq:collocation}
	\end{equation}
	Equation \eqref{eq:collocation} needs to be solved for the unknown
	coefficients $\mathbf{c}_{mj}\in\mathbb{R}^{2n}$ for all $j\in\{1,\dots,m\}$.
	Note that under the non-resonance conditions \eqref{eq:nonresonance-1},
	the Green's function $\mathbf{G}$ is bounded in time. Then the collocation
	method with linear splines will converge to the exact solution for
	the linear integral equation at each iteration step (cf. Kress \cite{Kress},
	chapter 13). Furthermore, if the exact solution $\mathbf{z}$ is twice
	continuously differentiable in $t$, we also obtain an error estimate
	for the linear spline collocation approximate solution $\mathbf{z}_{m}$
	as 
	\[
	\|\mathbf{z}_{m}-\mathbf{z}\|_{\infty}\le M\|\ddot{\mathbf{z}}\|_{\infty}\Delta^{2},
	\]
	where $\Delta$ is the spacing between the uniform collocation points,
	$M$ is an order constant depending upon the Green's function $\mathbf{G}$.
	
	Alternatively, the Galerkin method can be chosen to approximate the
	solution of eqs. \eqref{eq:IE_periodic}~and~\eqref{eq:IE_QP} using
	Fourier basis with harmonics of the base frequencies, followed by
	a projection onto the same basis vectors, given by
	
	\begin{align}
	\int_{0}^{T_{j}}\mathbf{z}_{m}(s)b_{mj}(s)\,ds & =\int_{0}^{T_{j}}\sum_{\boldsymbol{\kappa}\in\mathbb{Z}^{k}}\mathbf{V}\mathbf{H}(T_{\boldsymbol{\kappa}})\mathbf{V}^{-1}\left(\mathbf{F_{\boldsymbol{\kappa}}}-\mathbf{R_{\boldsymbol{\kappa}}}\{\mathbf{z}_{m}\}\right)e^{i\left\langle \boldsymbol{\boldsymbol{\kappa}},\boldsymbol{\Omega}\right\rangle s}b_{mj}(s)\,ds\,,\quad j=1,\dots,m,\label{eq:Fourierbasis}
	\end{align}
	where $b_{mj}(t)=e^{i\left\langle \boldsymbol{\boldsymbol{\kappa}}_{j},\boldsymbol{\Omega}\right\rangle t}$
	are the Fourier basis functions, with the corresponding time-periods~${T_{j}=\frac{2\pi}{\left\langle \boldsymbol{\boldsymbol{\kappa}}_{j},\boldsymbol{\Omega}\right\rangle }}$.
	Due to the orthogonality of these basis function, system \eqref{eq:Fourierbasis}
	simplifies to 
	\[
	\mathbf{c}_{mj}=\sum_{\boldsymbol{\kappa}\in\mathbb{Z}^{k}}\mathbf{V}\mathbf{H}(T_{\boldsymbol{\kappa}})\mathbf{V}^{-1}\left(\mathbf{F_{\boldsymbol{\kappa}}}-\mathbf{R_{\boldsymbol{\kappa}}}\{\mathbf{z}_{m}\}\right),\quad j=1,\dots,m
	\]
	which is a system of nonlinear equations for the unknown coefficients
	$\mathbf{c}_{mj}$.
	
	In the frequency domain, this system of coefficient equations are
	the same as that obtained from the multi-frequency harmonic balance
	method after finite truncation (see e.g. Chua and Ushida~\cite{Chua_harm_bal} or
	Lau~and~Cheung~\cite{Lau_QuasiP}). Our explicit formulas in~\eqref{eq:ampl_fkt_posdep},
	however, should
	speed up the calculations relative to a general application of the
	harmonic balance method. The same scheme applies in the periodic case.
	For both cases (periodic and quasi-periodic) the error due to the
	omission of higher harmonics in harmonic balance procedure is not
	well understood (see our review of the available results on the periodic
	case in the Introduction ). For the quasi-periodic case, no such error
	bounds are known to the best of our knowledge. 
	
	Equations \eqref{eq:IE_periodic} and \eqref{eq:IE_QP} are Fredholm
	integral equations of the second\emph{ }kind (cf. Zemyan~\cite{Zemyan}).
	Fortunately, the theory and solution methods for integral equations
	of the second kind are considerably simpler than for those of the first
	kind. We refer to Atkinson~\cite{Atkinson} for an exhaustive treatment of
	numerical methods for such integral equations. Our supplementary MATLAB\textsuperscript{\tiny\textregistered} code
	provides the implementation details of a simple yet powerful collocation-based
	approach for the periodic case, and a Galerkin projection using a
	Fourier basis for the quasi-periodic case in the frequency domain. 
	
	\subsection{Numerical continuation\label{subsec:Numerical-continuation}}
	
	Both the Picard and the Newton\textendash Raphson iteration are efficient
	solution techniques for specific forcing functions. However,
	to obtain the steady-state response as a function of forcing amplitudes
	and frequencies, numerical continuation is required. 
	
	A simple approach is to use \emph{sequential continuation}, in which
	the frequency (or time period) of oscillation is treated as a continuation
	parameter (multi-parameter in case of quasi-periodic oscillations).
	This parameter is varied in small steps and the corresponding (quasi) periodic
	response is iteratively computed at each step. The solution in any
	given step is typically used as initial guess for the next step. Such
	an approach will generally fail near a fold bifurcation with respect to one of the
	base frequencies $\boldsymbol{\Omega}$. In such cases, more advanced
	continuation schemes such as the \emph{pseudo arc-length continuation}
	are required.
	
	\paragraph{The pseudo-arc length approach.}
	
	We briefly explain the steps involved in numerical continuation for
	calculation of periodic orbits using the pseudo-arc-length approach,
	which is commonly used to track folds in a single-parameter solution
	branch. We seek the frequency response curve of system \eqref{eq:eqm_posdep}
	as a family of solutions to 
	\begin{equation}
	\boldsymbol{\mathcal{N}}(\mathbf{x},T):=\mathbf{x}-\int_{0}^{T}\mathbf{U}\mathbf{L}\left(t-s;T(p)\right)\mathbf{U}^{\top}\left(\mathbf{f}(s)-\mathbf{S}(\mathbf{x}(s))\right)\,ds=\mathbf{0},\label{eq:Continuation}
	\end{equation}
	with the fictitious variable $p$ parameterizing the solution curve
	$(\mathbf{x},T)$ of equation \eqref{eq:Continuation}. If $(\mathbf{x},T)$
	is a regular solution for equation \eqref{eq:Continuation}, then
	the implicit function theorem guarantees the existence of a nearby
	unique solution. The tangent vector $\mathbf{t}$ to the solution
	curve at $(\mathbf{x},T)$ is given by the null-space of the Jacobian
	$\left.D\boldsymbol{\mathcal{N}}\right|_{(\mathbf{x},T)}$, i.e., $\mathbf{t}$
	can be obtained by solving the system of equations 
	\begin{align*}
	\left.D\boldsymbol{\mathcal{N}}\right|_{(\mathbf{x},T)}\mathbf{t} & =\mathbf{0},\\
	\left\langle \mathbf{t},\mathbf{t}\right\rangle  & =1,
	\end{align*}
	where the second equation specifies a unity constraint on the length
	of the tangent vector to uniquely identify $\mathbf{t}$. This tangent
	vector gives a direction on the solution curve along which $p$ increases.
	Starting with a known solution $\mathbf{u}_{0}:=(\mathbf{x}_{0},T_{0})$
	on the solution branch with the corresponding tangent direction vector
	$\mathbf{t}_{0}$ and a prescribed arc-length step size $\Delta p$,
	we obtain a nearby solution $\mathbf{u}$ by solving the nonlinear
	system of equations
	\begin{align*}
	\boldsymbol{\mathcal{N}}\left(\mathbf{u}\right) & =\mathbf{0},\\
	\left\langle \mathbf{u}-\mathbf{u}_{0},\mathbf{t}_{0}\right\rangle  & =\Delta p.
	\end{align*}
	This system can again be solved iteratively using, e.g., the Newton\textendash Raphson
	procedure, with the Jacobian given by $[D\boldsymbol{\mathcal{N}},\,\mathbf{t}_{0}]^{\top}$.
	We need the derivative of $\boldsymbol{\mathcal{N}}(\mathbf{x},T)$
	with respect to $T$ to evaluate this Jacobian, which can also be
	explicitly computed using the derivative of the Green's function $\mathbf{L}$
	or $\mathbf{G}$ with respect to $T$. These expressions are detailed in the
	Appendix~\ref{app:derivative} and implemented in the supplementary MATLAB\textsuperscript{\tiny\textregistered} code.
	
	Although this pseudo-arc length continuation is able to capture folds
	in single-parameter solution branches, it is not the state-of-the-art
	continuation algorithm. More advanced continuation schemes, such as
	the atlas algorithm of $\textsc{coco}$ \cite{CoCo} enable continuation
	with multi-dimensional parameters required for quasi-periodic forcing. In this work, we have implemented our integral
	equations approach with the MATLAB\textsuperscript{\tiny\textregistered}-based continuation package $\textsc{coco}$
	\cite{CoCo} . 
	
	\subsubsection{Continuation of periodic orbits in the conservative autonomous setting}
	
	\label{subsec:UnforcedAutoCont}
	
	As remarked earlier, conservative autonomous systems have internally parameterized
	periodic orbits a priori unknown period. Moreover,
	each (quasi-) periodic orbit of such a system is part of a family
	of (quasi-) periodic trajectories with the solutions differing only
	in their phases.  A unique solution is obtained using a \emph{phase
		condition, }a scalar equation of the form 
	\begin{equation}
	\mathfrak{p}(\mathbf{z}(t))=0,\label{eq:phase_cond}
	\end{equation}
	which fixes a Poincare section transverse to the (quasi-)
	periodic orbit in the phase space. The choice of this section is arbitrary
	but often involves setting the initial velocity of one of the degrees
	of freedom to be zero, i.e., letting
	\[
	\mathfrak{p}(\mathbf{z})=\dot{x}_{k}(0)=0,
	\]
	asserting that the initial velocity at the $k^{\text{th}}$ degree
	of freedom is zero. As the solution time-period $T$ is unknown,
	eqs. \eqref{eq:Continuation}-\eqref{eq:phase_cond}
	have an equal number of
	equations and variables $(\mathbf{x},T)$. Thus, there is no parameter among the variables, which can
	be used for the continuation of a given solution. 
	
	In literature, this issue is avoided by introducing a fictitious damping
	parameter, say $d$, and establishing the existence of a periodic
	orbit if and only if $d=0$ , i.e., in the conservative limit (Muñoz-Almaraz et al. \cite{ConservativeSystems_continuation}). With this trick, we reformulate the integral equation~\eqref{eq:Continuation} (with forcing $\mathbf{f}(t)\equiv\mathbf{0}$
	) to find periodic orbits of the system 
	\begin{equation}
	\mathbf{M}\ddot{\mathbf{x}}+d\mathbf{K}\dot{\mathbf{x}}+\mathbf{K}\mathbf{x}+\mathbf{S}(\mathbf{x})=\mathbf{0}.\label{eq:Sys_artDamp}
	\end{equation}
	Periodic solutions to this system are created through a Hopf bifurcation
	that occurs precisely at the conservative limit of system \eqref{eq:Sys_artDamp}
	($d=0$). It can be shown that the damping parameter maintains a zero
	value along the periodic-solution-branch obtained at the bifurcation
	point. Advanced continuation algorithms can then be used to detect such a
	Hopf bifurcation and to make a switch to the periodic solution branch
	(cf. Galán-Vioque and Vanderbauwhede~\cite{Galan}). 
	
	A similar continuation procedure can be carried out in our integral-equation
	approach. The periodic solution, however, technically does not arise
	from a Hopf bifurcation, because the non-resonance conditions \eqref{eq:nonresonance-1}
	and \eqref{eq:nonresonance-1-1} are violated at the Hopf bifurcation
	point. Nonetheless, continuation of a given solution point on the
	solution branch is possible using standard continuation algorithms.
	In this conservative autonomous setting, a given solution of system
	\eqref{eq:Continuation}, \eqref{eq:phase_cond} may again be continued
	in the $(\mathbf{x},T,d)$ space using, e.g., the pseudo arc-length
	approach.
	
	\section{Numerical examples\label{sec:Numerical-Examples}}
	
	To illustrate the power of our integral-equation-based approach in locating
	the steady-state response, we consider two numerical examples. The
	first one is a two-degree-of-freedom system with geometric nonlinearity.
	We apply our algorithms under periodic and quasi-periodic forcing, as
	well as to the autonomous system with no external forcing. We also treat
	a case of non-smooth nonlinearities. For periodic forcing, we compare
	the computational cost  with algorithm implemented in the $\texttt{po}$~toolbox of the state-of-the-art continuation software $\textsc{coco}$
	\cite{CoCo}. Subsequently, we perform similar computations for a
	higher-dimensional mechanical system. 
	
	\subsection{2-DOF Example}
	
	\begin{figure}[H]
		\centering{}\includegraphics{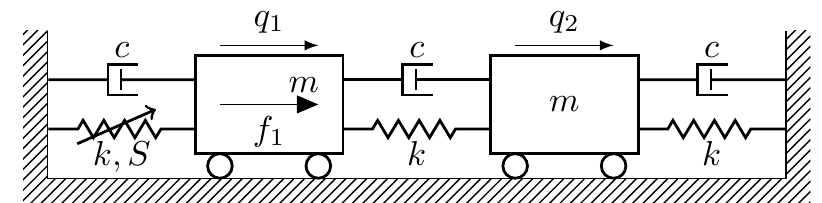}\caption{\label{fig:SP_example}Two-mass oscillator with the non-dimensional
			parameters $m$, $k$ and $c$. }
	\end{figure}
	
	We consider a two-degree-of-freedom oscillator shown in Figure \ref{fig:SP_example}.
	The nonlinearity $\mathbf{S}$ is confined to the first spring and
	depends on the displacement of the first mass only. The equations
	of motion are
	\begin{equation}
	\left[\begin{array}{cc}
	m & 0\\
	0 & m
	\end{array}\right]\ddot{\mathbf{q}}+\left[\begin{array}{cc}
	2c & -c\\
	-c & 2c
	\end{array}\right]\dot{\mathbf{q}}+\left[\begin{array}{cc}
	2k & -k\\
	-k & 2k
	\end{array}\right]\mathbf{q}+\left[\begin{array}{c}
	S(q_{1})\\
	0
	\end{array}\right]=\left[\begin{array}{c}
	f_{1}(t)\\
	f_2(t)
	\end{array}\right]\,.\label{eq:2Dof_sys}
	\end{equation}
	This system is a generalization of the two-degree-of-freedom
	oscillator studied by Szalai~et~al.~\cite{Szalai_SSM}, which is a slight modification of the classic example of Shaw~and~Pierre~\cite{ShawPierre}. Since the damping matrix $\mathbf{C}$
	is proportional to the stiffness matrix $\mathbf{K}$, we can employ
	the Green's function approach described in Section \ref{subsec:Proportional_damping}.
	The eigenfrequencies and modal damping of the linearized system at
	$q_{1}=q_{2}=0$ are given by 
	
	\[
	\begin{split}\omega_{0,1}=\sqrt{\frac{k}{m}},\qquad\omega_{0,2}=\sqrt{\frac{3k}{m}}, & \qquad \zeta_{1}=\frac{cm}{2\omega_{0,1}},\qquad \zeta_{2}=\frac{cm}{2\omega_{0,2}}.\end{split}
	\]
	With those constants, we can calculate the constants $\alpha_{j}$
	and $\omega_{j}$ (cf. eq. \eqref{eq:alpha_beta}) for the Green's
	function $L_{j}$ in eq. \eqref{eq:G_pos}. We will consider three different
	versions of system \eqref{eq:2Dof_sys}: smooth nonlinearity with
	periodic and subsequently quasi-periodic forcing; smooth nonlinearity
	without forcing; discontinuous nonlinearity without forcing. 
	
	\subsubsection{Periodic forcing\label{subsec:EX_2Dof_periodic} }
	
	First, we consider system \eqref{eq:2Dof_sys} with harmonic
	forcing of the form
	\begin{equation}
	f_{1}=f_2(t)=A\sin(\Omega t),\label{eq:singlharm_forc}
	\end{equation}
	and a smooth nonlinearity of the form
	
	\begin{equation}
	S(q_{1})=0.5q_{1}^{3}\,,\label{eq:cubic_NL}
	\end{equation}
	which is the same nonlinearity considered by Szalai~et~al.~\cite{Szalai_SSM}.
	The integral-equation-based steady-state response curves are shown
	in Figure \ref{fig:SP_BB-3} for a full frequency sweep and for different
	forcing amplitudes. As expected, our Picard iteration scheme (red)
	converges fast for all frequencies in case of low forcing amplitudes.
	For higher forcing amplitude, the method no longer converges in a
	growing neighborhood of the resonance. To improve the results close to the resonance, we employ the Newton\textendash Raphson scheme of Section~\ref{sec:Newton-Iteration}. We see that the latter iteration captures the periodic response even for larger amplitudes near resonances until a fold arises in the response
	curve. We need more sophisticated continuation algorithms to
	capture the response around such folds. 
	
	\begin{figure}[H]
		\begin{centering}
			\includegraphics[width=0.7\textwidth]{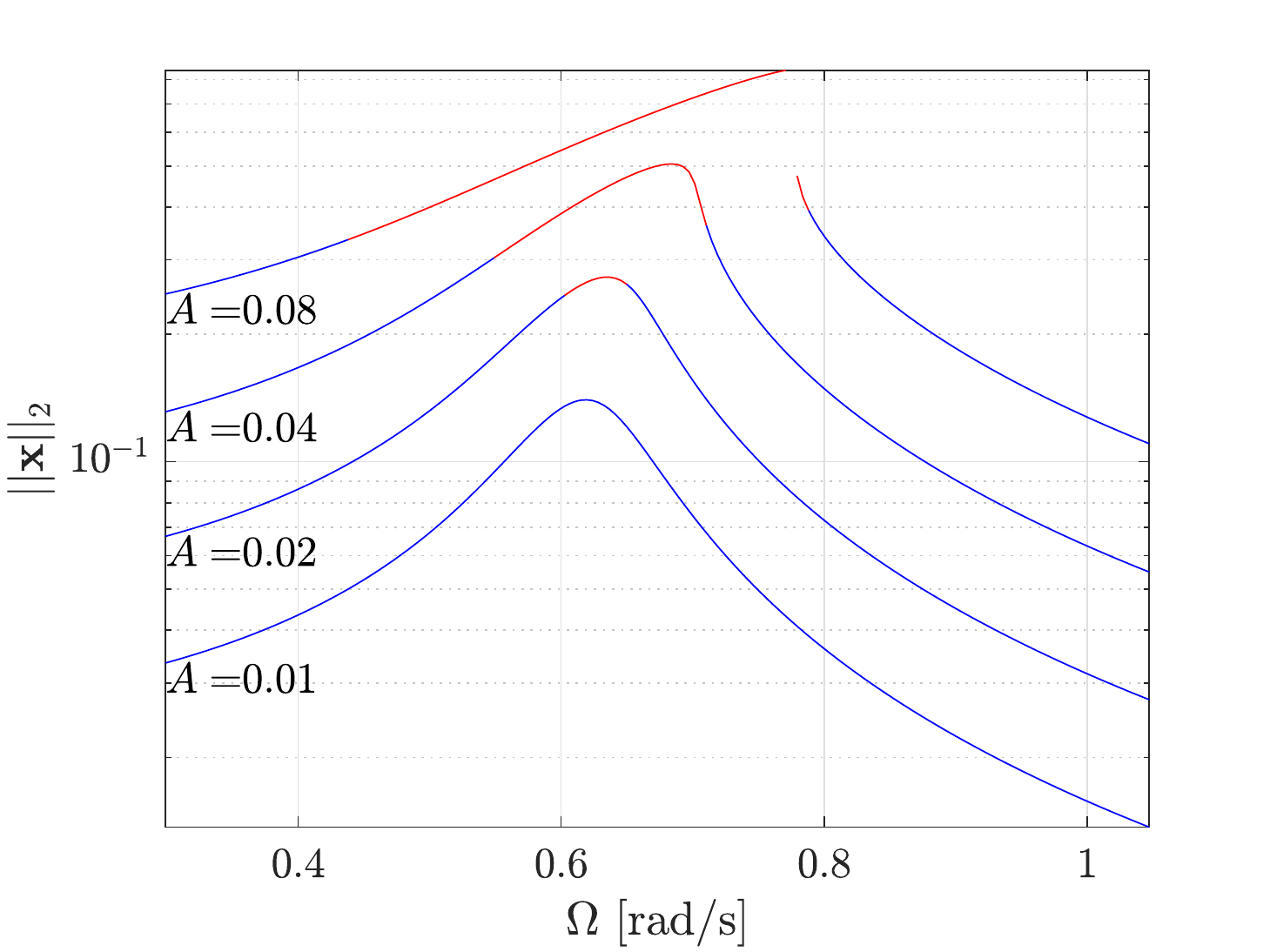}
			\par\end{centering}
		\caption{\label{fig:SP_BB-3} Example 1 with the nonlinearity \eqref{eq:cubic_NL}
			and the forcing \eqref{eq:singlharm_forc}, and the non-dimensional
			parameters $m=1$, $k=1$ and $c=0.3$: We use sequential continuation with the Picard iteration (blue) to capture the nonlinear frequency response curves for different amplitudes of loading. For relatively high amplitudes and near-resonant loading, the Picard approach fails to converge as expected. We then switch to the Newton\textendash Raphson method (red). Indeed, using sequential continuation, the Newton--Raphson approach is unable to capture the fold developing in the forced response for high amplitudes (cf.~Figure~\ref{fig:SP_Cont_Comp}).}
	\end{figure}
	
	\paragraph{Performance comparison between the integral equations and the \texttt{po} toolbox of}\textsc{coco}
	
	\begin{figure}[H]
		\begin{centering}
			\includegraphics[width=0.7\textwidth]{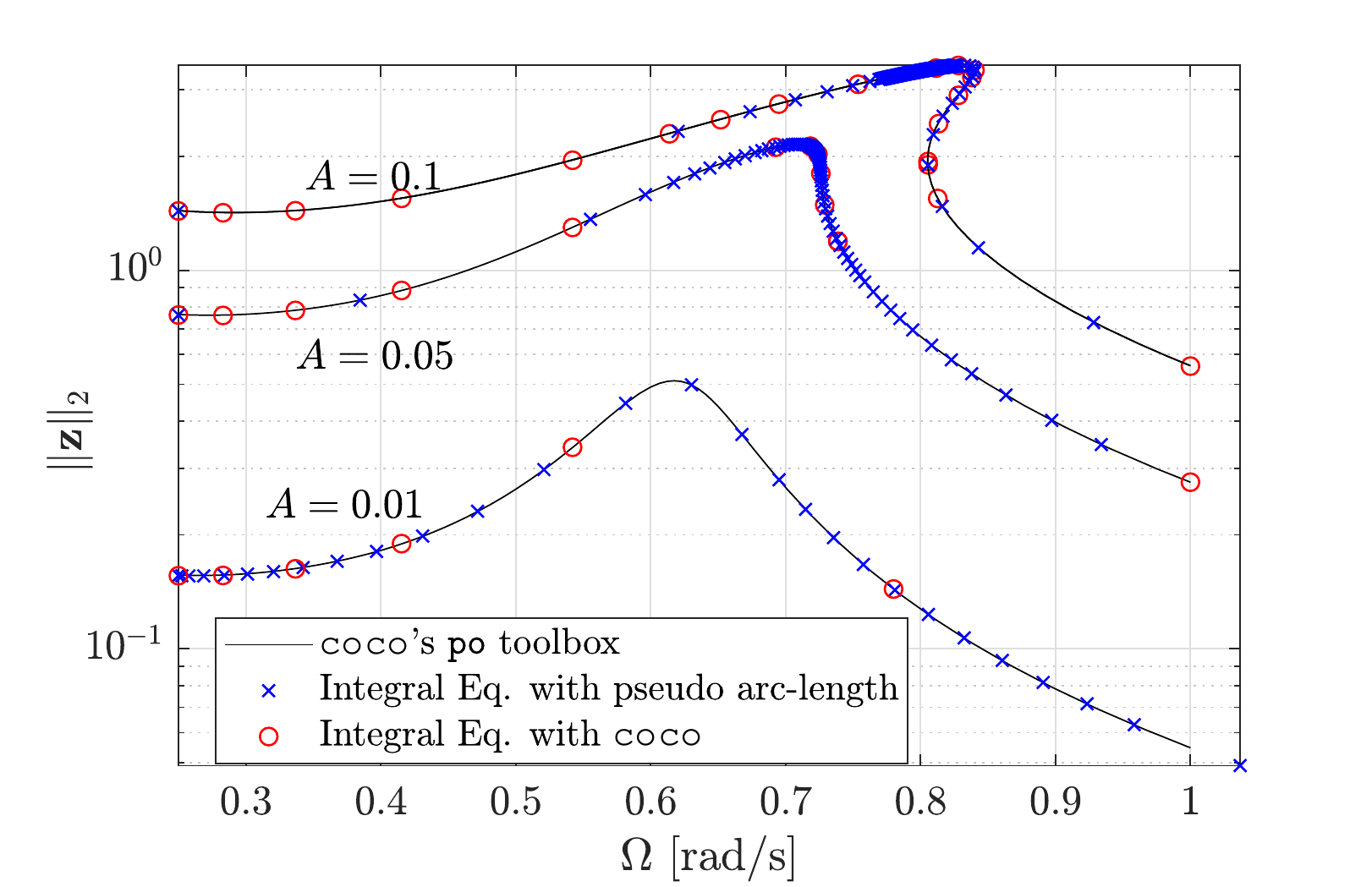}
			\par\end{centering}
		\caption{\label{fig:SP_Cont_Comp} Steady-state responses obtained for Example
			1 (with nonlinearity \eqref{eq:cubic_NL} and forcing \eqref{eq:singlharm_forc})
			from different continuation techniques. Numerical continuation using \textsc{coco} or pseudo-arc-length technique is able to capture the fold appearing in the response curve for $ A=1 $ (cf.~Figure~\ref{fig:SP_Cont_Comp}). Table~\ref{tab:CompContTimes} compares the computational performance of different methods on this example. }
	\end{figure}
	
	\begin{table}[H]
		\begin{centering}
			\begin{tabular}{|c|>{\centering}p{0.25\textwidth}|>{\centering}p{0.25\textwidth}|>{\centering}p{0.25\textwidth}|}
				\hline 
				\multirow{2}{*}{Forcing amplitude} & \multicolumn{3}{c|}{Computation time {[}seconds (\# continuation steps){]}}\tabularnewline
				\cline{2-4} 
				&  $\texttt{po}$-toolbox of \textsc{coco}  & Integral eq. with \emph{in-house} continuation & Integral eq. with \textsc{coco} continuation\tabularnewline
				\hline 
				\hline 
				0.01 & 16 (88 steps) & 0.15 (28 steps) & 2 (56 steps)\tabularnewline
				\hline 
				0.05 & 26 (124 steps) & 3.89 (700 steps) & 5 (110 steps)\tabularnewline
				\hline 
				0.1 & 32 (139 steps) & 298.73 (38537 steps) & 8 (160 steps)\tabularnewline
				\hline 
			\end{tabular}
			\par\end{centering}
		\caption{\label{tab:CompContTimes}Comparison of computational performance for different
			continuation approaches}
	\end{table}
	
	As shown in Table \ref{tab:CompContTimes}, the integral-equation
	approach proposed in the present paper is substantially faster
	than the $\texttt{po}$ toolbox for continuation of periodic orbits
	with the MATLAB\textsuperscript{\tiny\textregistered}-based continuation package $\textsc{coco}$ \cite{CoCo}
	for low enough amplitudes. However, as the frequency response starts
	developing complicated folds for higher amplitudes (cf. Figure \ref{fig:SP_Cont_Comp}),
	a much higher number of continuation steps are required for the convergence of our simple implementation of the pseudo-arc length continuation (cf. the third column in Table \ref{tab:CompContTimes}).
	Since $\textsc{coco}$ is capable of performing continuation on general
	problems with advanced algorithms, we have implemented our integral-equation
	approach in $\textsc{coco}$ in order to overcome this limitation.
	As shown in Table \ref{tab:CompContTimes}, the integral equation
	approach, along with $\textsc{coco}$'s built-in continuation scheme,
	is much more efficient for high-amplitude loading than any other
	method we have considered. 
	
	The integral-equation-based continuation was performed with $n_{t}=50$
	time steps to discretize the solution in the time domain. On the other
	hand, the $\texttt{po}$ toolbox in $\textsc{coco}$ performs collocation-based
	continuation of periodic orbits, whereby it is able to modify the
	time-step discretization in an adaptive manner to optimize performance.
	In principle, it is possible to build an integral-equation-based toolbox
	in $\textsc{coco}$, which would allow for the adaptive selection
	of the discretization steps. This is expected to further increase
	the performance of integral equations approach, when equipped with
	$\textsc{coco}$ for continuation.
	
	\subsubsection{Quasi-periodic forcing}
	
	Unlike the shooting technique reviewed earlier, our approach can also
	be applied to quasi-periodically forced systems (cf. Theorems~\ref{Theorem:IE_QP}~and~\ref{Thm:special_case_IE_QP}). Therefore, we can also choose
	a quasi-periodic forcing of the form 
	\begin{equation}
	f_2(t) =0,\quad f_{1}=0.01(\sin(\Omega_{1}t)+\sin(\Omega_{2}t))\,,\quad\kappa_{1}\Omega_{1}+\kappa_{2}\Omega_{2}\neq0\,,\quad\kappa_{1},\kappa_{2}\in\mathbb{Z}-\{0\},\label{eq:qper_forcing}
	\end{equation}
	in Example 2, with the nonlinearity still given by eq. \eqref{eq:cubic_NL}.
	Choosing the first forcing frequency $\Omega_{1}$ close to the first
	eigenfrequency $\omega_{1}$ and the second forcing frequency $\Omega_{2}$
	close to $\omega_{2}$, we obtain the results depicted in Figure \ref{fig:SP_qper}.
	We show the maximal displacement as a function of the two forcing
	frequencies, which are always selected to be incommensurate, otherwise
	the forcing would not be quasi-periodic. We nevertheless connect
	the resulting set of disrete points with a surface in Figure~\ref{fig:SP_qper} for better visibility. 
	
	\begin{figure}[h]
		\centering{}\includegraphics[scale=0.75]{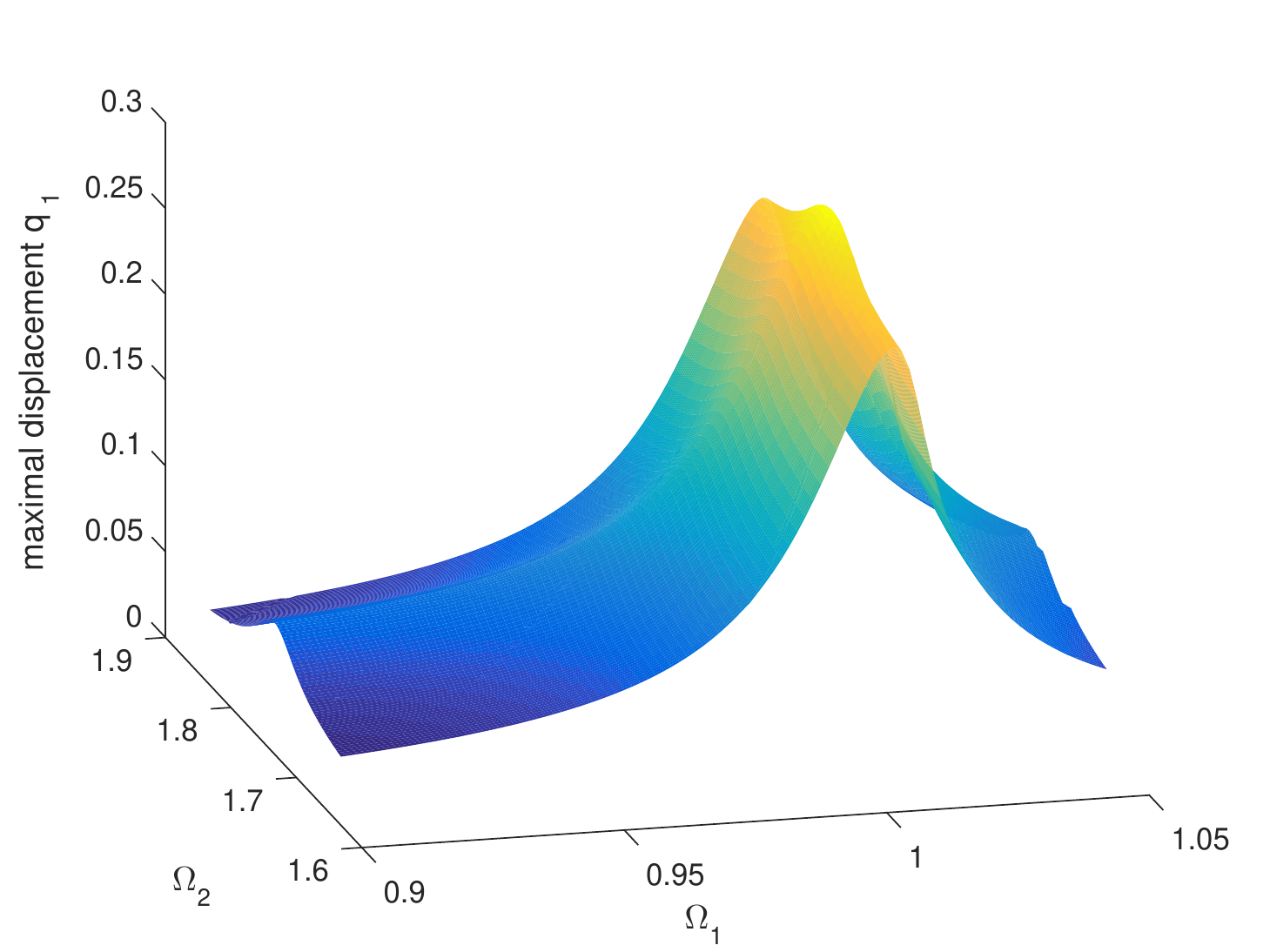}\caption{\label{fig:SP_qper}Response curve for Example 1 with the nonlinearity
			\eqref{eq:cubic_NL} and the forcing \eqref{eq:qper_forcing}, and the non-dimensional
			parameters $m=1$, $k=1$ and $c=0.02$.}
	\end{figure}
	
	To carry out the quasi-periodic Picard iteration~\eqref{eq:Picard_qper_pos_only},
	the infinite summation involved in the formula has to be truncated. We chose to truncate the Fourier expansion once
	its relative error is within $10^{-3}$. If the iteration \eqref{eq:Picard_qper_pos_only}
	did not converge, we switched to the Newton\textendash Raphson scheme
	described in Section~\ref{subsec:NR_qper}. In that case, we only
	kept the first three harmonics as Fourier basis.
	
	Figure \ref{fig:Num_it} shows the number of iterations needed to
	converge to a solution with this iteration procedure. Especially away
	from the resonances, a low number of iterations suffices for convergence
	to an accurate result. Also included in Figure \ref{fig:Num_it} are
	the conditions \eqref{eq:Lipschitz_cond_qper} and \eqref{eq:delta_cond_qper},
	which guarantee the convergence for the iteration to the steady state
	solution of system \eqref{eq:0}. Outside the two red curves both
	\eqref{eq:Lipschitz_cond_qper} and \eqref{eq:delta_cond_qper} are
	satisfied and, accordingly, the iteration is guaranteed to converge.
	Since these conditions are only sufficient for convergence, the iteration
	converges also for frequency pairs within the red curves. The number
	of iterations required increases gradually and within the white region bounded by green lines, the Picard iteration fails. In such cases, we proceed to employ the Newton\textendash Raphson scheme (cf. section \eqref{subsec:NR_qper}). 
	
	\begin{figure}[H]
		\begin{centering}
			\includegraphics[scale=0.75]{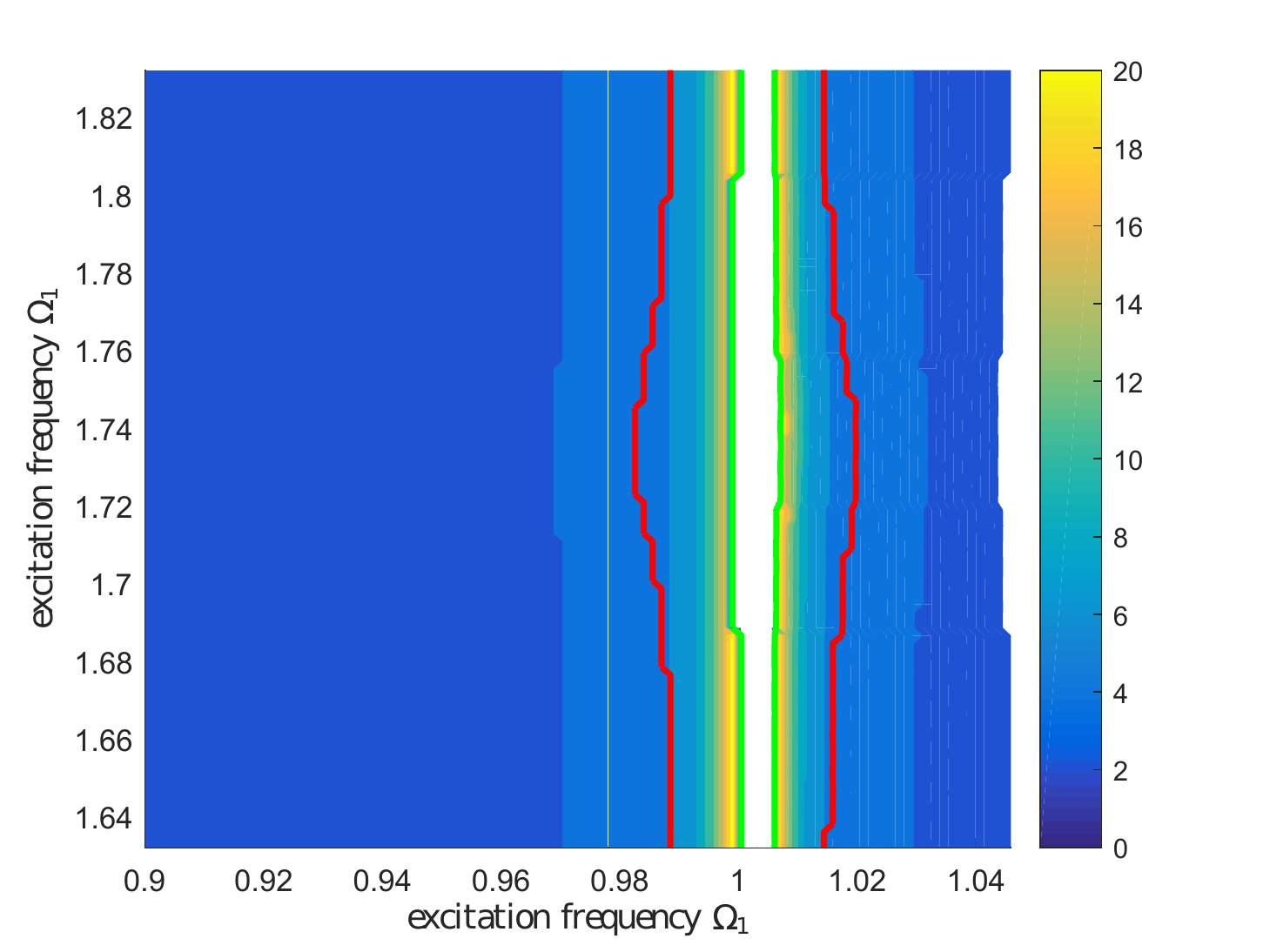}
			\par\end{centering}
		\caption{\label{fig:Num_it}Number of iterations needed on the construction
			of Figure \ref{fig:SP_qper}. Red curves bound the a priori guaranteed
			region of convergence for the Picard iteration. The white region
			is the domain where this iteration fails. In this region, we employ
			the Newton\textendash Raphson scheme (cf. Section~\ref{subsec:NR_qper}). }
	\end{figure}

	\subsubsection{Non-smooth nonlinearity\label{subsec:Ex_Nonsmooth}}
	
	As noted earlier, our iteration schemes are also applicable to non-smooth
	system as long as they are still Lipschitz. We select the nonlinearity
	of the form
	\begin{equation}
	S(q_{1})=\begin{cases}
	\alpha\text{sign}(q_{1})(|q_{1}|-\beta), & \textrm{for }|q_{1}|>\beta,\\
	0, & \textrm{otherwise},
	\end{cases}\label{eq:disconti_NL}
	\end{equation}
	which represents a hardening ($\alpha>0$) or softening ($\alpha<0$)
	spring with play $\beta>0$. The spring coefficient is given by $\tan^{-1}(\alpha)$,
	as depicted in Figure \ref{fig:disconti_force}. 
	
	\begin{figure}[H]
		\begin{centering}
			\includegraphics[width=0.5\linewidth]{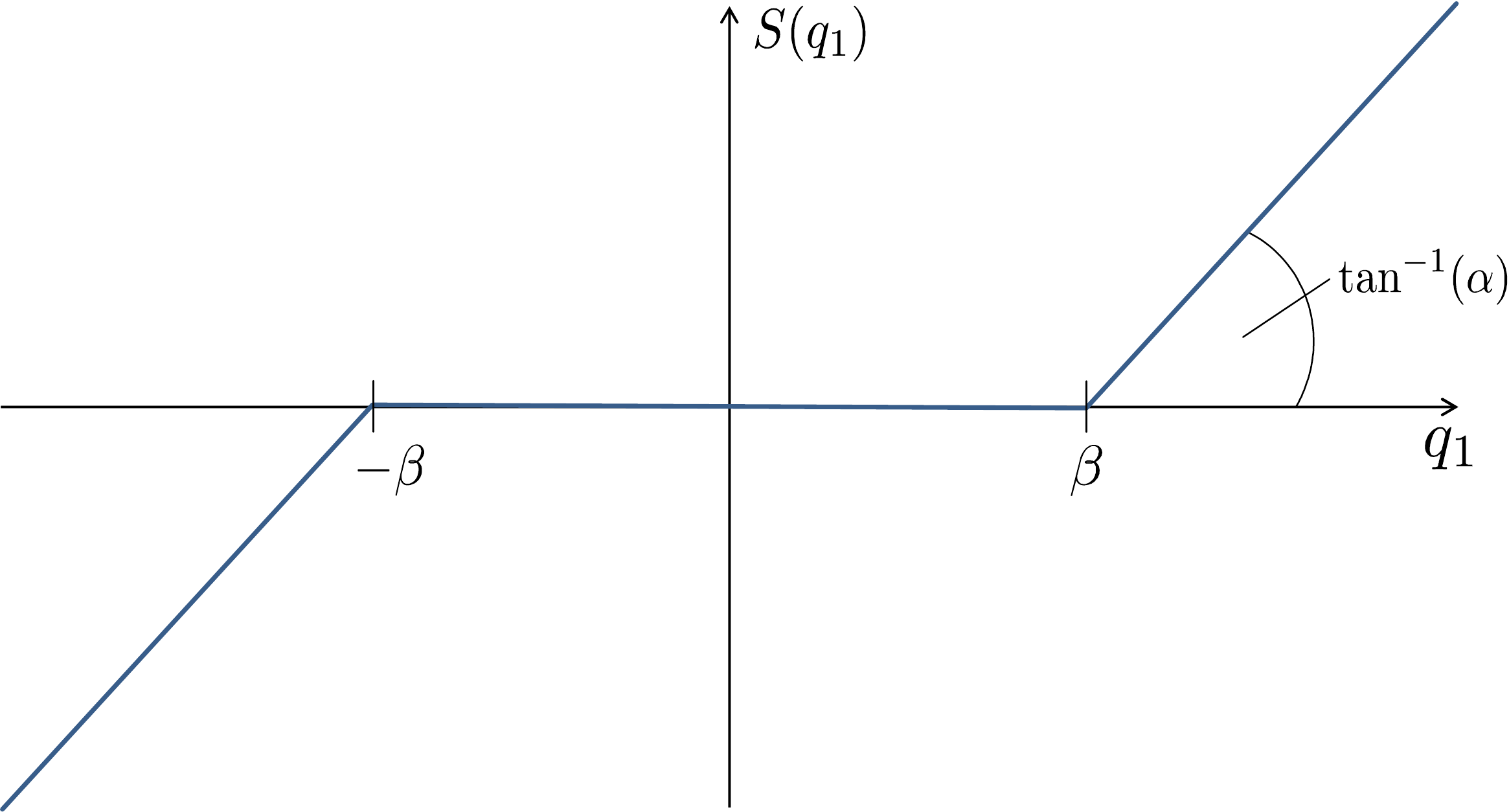}
			\par\end{centering}
		\caption{\label{fig:disconti_force}Graph of the non-smooth nonlinearity~\eqref{eq:disconti_NL}. }
	\end{figure}
	If we apply the forcing \[ f_1(t) = 0.02 \sin (\Omega t),\quad f_2(t) = 0 \] to system \eqref{eq:2Dof_sys}
	with the nonlinearity \eqref{eq:disconti_NL}, our iteration techniques
	yield the response curve depicted in Figure \ref{fig:disconti_NL}.
	The Picard iteration approach \eqref{eq:Picard_per_pos_only} converges
	for moderate amplitudes, also in the nonlinear regime ($|q_{1}|>\beta$).
	When the Picard iteration fails at higher amplitudes, we employ the
	Newton\textendash Raphson iteration. These results match
	closely with the amplitudes obtained by numerical integration, as seen
	in Figure~\ref{fig:disconti_NL}. 
	
	\begin{figure}[H]
		\begin{centering}
			\includegraphics[scale=0.6]{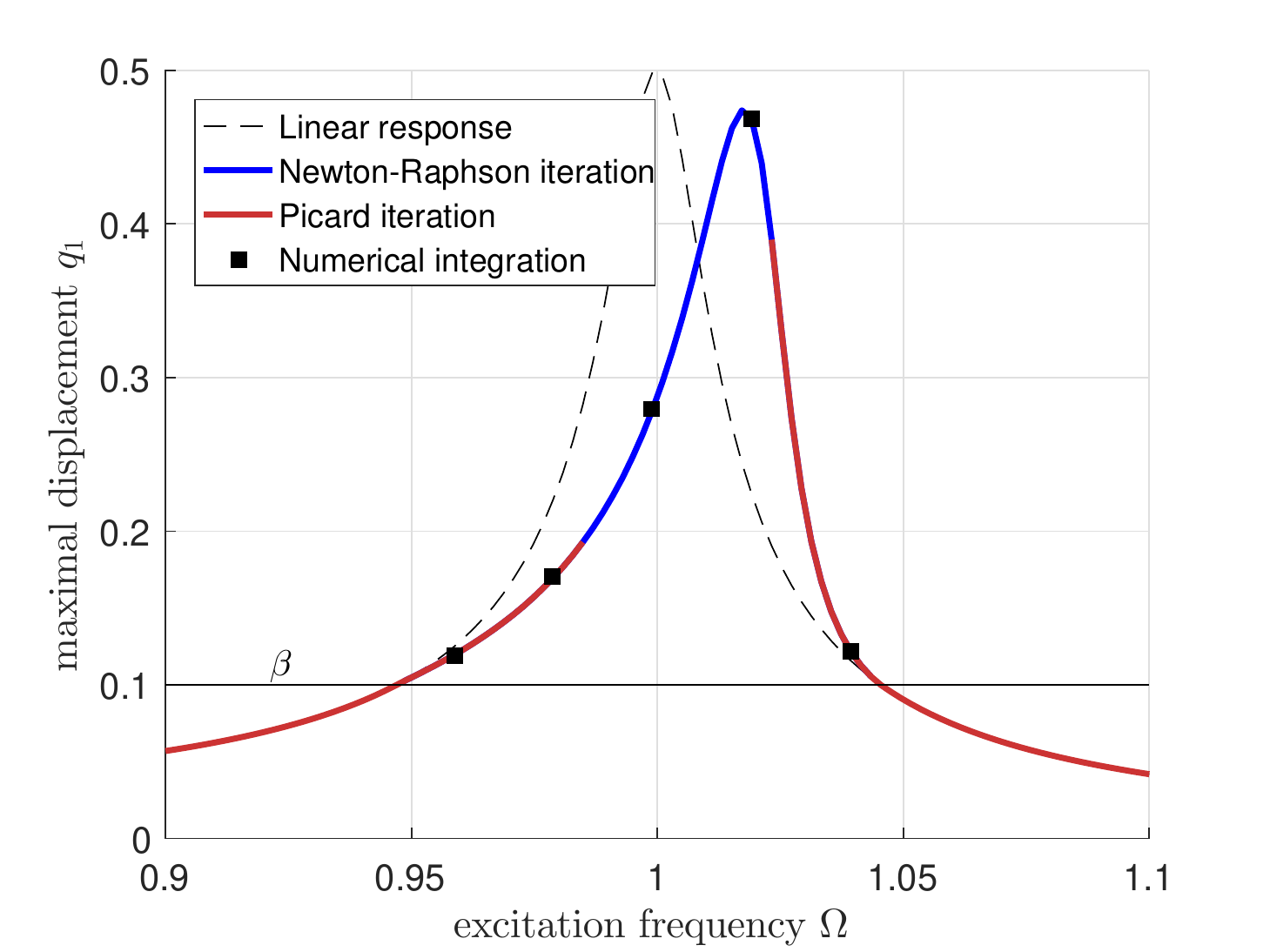}
			\par\end{centering}
		\caption{\label{fig:disconti_NL}Response curve for Example 1 with the nonlinearity
			\eqref{eq:disconti_NL} and the forcing \eqref{eq:singlharm_forc};
			parameters: $m=1$, $k=1$ and $c=0.02$, $ \alpha = \beta = 0.1 $. }
	\end{figure}
	
	\subsection{Nonlinear oscillator chain }
	
	\begin{figure}[H]
		\centering{}\includegraphics{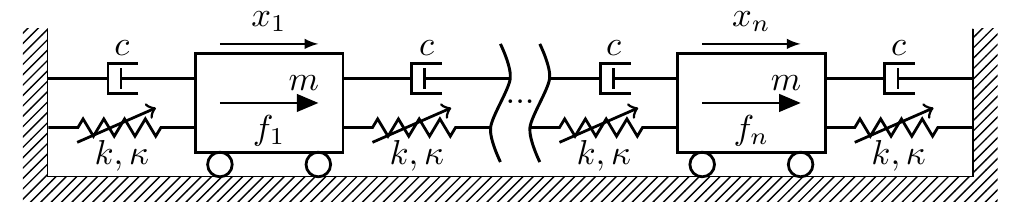}\caption{\label{fig:chain_example} An $n$-mass oscillator chain with coupled
			nonlinearity. We select the non-dimensional parameters $m=1$, $k=1,\kappa=0.5$
			and $c=1$. }
	\end{figure}
	
	To illustrate the applicability of our results to higher-dimensional systems and more complex nonlinearities, we consider a modification of the oscillator chain studied by Breunung and Haller \cite{Breunung_Backbone}. Shown Figure \eqref{fig:chain_example}, the oscillator chain consists of $n$ masses with linear and cubic nonlinear springs coupling every pair of adjacent masses. Thus, the nonlinear function $\mathbf{S}$ is given as: 
	\[
	\mathbf{S}(\mathbf{x})=\kappa\left[\begin{array}{c}
	x_{1}^{3}-\left(x_{2}-x_{1}\right)^{3}\\
	\left(x_{2}-x_{1}\right)^{3}-\left(x_{3}-x_{2}\right)^{3}\\
	\left(x_{3}-x_{2}\right)^{3}-\left(x_{4}-x_{3}\right)^{3}\\
	\vdots\\
	\left(x_{n-1}-x_{n-2}\right)^{3}-\left(x_{n}-x_{n-1}\right)^{3}\\
	\left(x_{2}-x_{1}\right)^{3}+x_{n}^{3}
	\end{array}\right].
	\]
	
	The frequency response curve obtained with the iteration described
	in section \ref{sec:Newton-Iteration} for harmonic forcing is shown
	in Figure \ref{fig:SP_coupled_nDOF} for $20$ degrees-of-freedom.
	We also include the frequency response obtained with the $\texttt{po}$-toolbox
	of $\textsc{coco}$ \cite{CoCo} with default settings for comparison. The integral equations
	approach gives the same solution as the $\texttt{po}$-toolbox of \textsc{coco},
	but the difference in run times is stunning: the \texttt{po}-toolbox of $\textsc{coco}$ takes about 12 minutes and 59
	seconds to generate this frequency response curve, whereas the integral-equation approach with a naive sequential continuation strategy takes 13 seconds to generate the same curve. This underlines the power of the approches proposed here for complex mechanical vibrations. 
	
	\begin{figure}[H]
		\begin{centering}
			\includegraphics[width=0.6\textwidth]{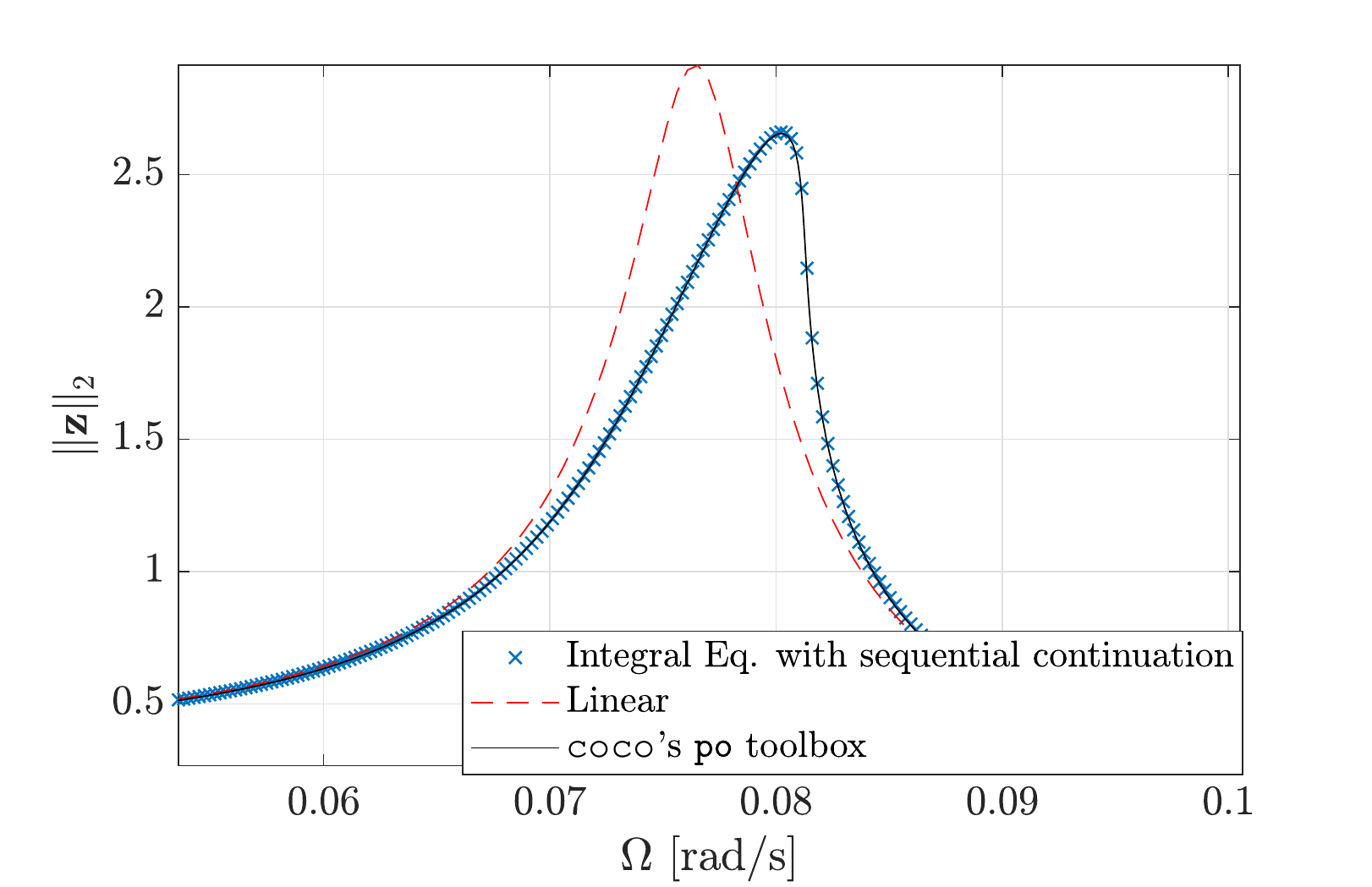}
			\par\end{centering}
		\caption{\label{fig:SP_coupled_nDOF} Comparison of frequency sweeps produced
			using different continuation techniques for a 20-DOF oscillator chain
			with coupled nonlinearity. }
	\end{figure}
	
	\section{Conclusion}
	
	We have presented an integral-equation approach for the fast computation
	of the steady-state response of nonlinear dynamical systems under
	external (quasi-) periodic forcing. Starting with a forced linear
	system, we derive integral equations that must be satisfied by the
	steady-state solutions of the full nonlinear system. The kernel of
	the integral equation is a Green's function, which we calculate explicitly
	for general mechanical systems. Due to these explicit formulae, the
	convolution with the Green's function can be performed with minimal effort, thereby making the solution of the equivalent integral
	equation significantly faster than full time integration of the dynamical
	system. We also show the applicability of the same equations to compute
	periodic orbits of unforced, conservative systems. 
	
	We employ a combination of Picard and the Newton\textendash Raphson
	iterations to solve the integral equations for the steady-state response.
	Since the Picard iteration requires only a simple application of a
	nonlinear map (and no direct solution via operator inversion), it
	is especially appealing for high-dimensional system. Furthermore,
	the nonlinearity only needs to be Lipschitz continuous, therefore
	our approach also applies to non-smooth systems, as we demonstrated
	numerically in section \ref{subsec:Ex_Nonsmooth}. We establish a
	rigorous a priori estimate for the convergence of the Picard iteration.
	From this estimate, we conclude that the convergence of the Picard
	iteration becomes problematic for high amplitudes and forcing frequencies
	near resonance with an eigenfrequency of the linearized system. This
	can also be observed numerically in Example \ref{subsec:EX_2Dof_periodic},
	where the Picard iteration fails close to resonance. 
	
	To capture the steady-state response for a full frequency sweep
	(including high amplitudes and resonant frequencies), we deploy the
	Newton\textendash Raphson iteration once the Picard iteration fails
	near resonance. The Newton\textendash Raphson formulation can be computationally-intensive
	as it requires a high-dimensional operator inversion, which would normally
	make this type of iteration potenially unfeasible for exceedingly high-dimensional
	systems. However, we circumvent this problem with the Newton\textendash Raphson
	method using modifications discussed in Section \ref{sec:Newton-Iteration}.

	We have further demonstrated that advanced numerical continuation
	is required to compute the (quasi-) periodic response when folds appear
	in solution branches. To this end, we formulated one such continuation
	scheme, i.e., the pseudo arc-length scheme, in our integral equations
	setting to facilitate capturing response around such folds. We also
	demonstrated that the integral equations approach can be coupled with
	existing state-of-the-art continuation packages to obtain better performance
	(cf. Section \ref{subsec:EX_2Dof_periodic}).
	
	Compared to well-established shooting-based techniques, our integral-equation approach also calculates quasi-periodic responses of
	dynamical systems and avoids numerical time integration. The latter
	can be computationally expensive for high-dimensional or stiff systems.
	In the case of purely geometric (position-dependent) nonlinearities,
	we can reduce the dimensionality of the corresponding integral iteration
	by half, by iterating on the position vector only. For numerical examples,
	we show that our integral equation approach equipped with numerical continuation
	outperforms available continuation packages significantly. As opposed
	to the broadly used harmonic balance procedure (cf. Chua and Ushida
	\cite{Chua_harm_bal} and Lau and Cheung \cite{Lau_QuasiP}), our
	approach also gives a computable and rigorous existence criterion
	for the (quasi-) periodic response of the system. 
	
	Along with this work, we provide a MATLAB\textsuperscript{\tiny\textregistered} code with a user-friendly
	implementation of the developed iterative schemes. This code implements
	the cheap and fast Picard iteration, as well as the robust Newton\textendash Raphson
	iteration, along with sequential/pseudo-arc length continuation. We
	have further tested our approach in combination with the MATLAB\textsuperscript{\tiny\textregistered}-based
	continuation package $\textsc{coco}$ \cite{CoCo} and obtained an
	improvement in performance. One could, therefore, further add an integral-equation-based
	toolbox to $\textsc{coco}$ with adaptive time steps in the discretization
	to obtain better efficiency. 
	
	\paragraph{Acknowledgments}
	
	We are thankful to Harry Dankowicz and Mingwu Li for clarifications
	and help with the continuation package $\textsc{coco}$ \cite{CoCo}. We also acknowledge helpful discussions with Mark Mignolet and Dane Quinn.
	
	\appendix
	
	\section{Proof of Lemma \ref{Lemma:1} \label{sec:ProofL1}}
	
	The general solution of \eqref{eq:modalfirstorder} is given by the
	classic variation of constants formula 
	\begin{equation}
	\mathbf{w}(t)=e^{\boldsymbol{\Lambda}t}\mathbf{w}(0)+\int_{0}^{t}e^{\boldsymbol{\Lambda}(t-s)}\boldsymbol{\psi}(s)\,ds.\label{eq:Duhamel_firstorder}
	\end{equation}
	If $\mathbf{w}_{0}(t)$ is the unique $T$-periodic solution of \eqref{eq:modalfirstorder},
	then we obtain from \eqref{eq:Duhamel_firstorder} that
	\[
	\mathbf{w}_{0}(T)=e^{\boldsymbol{\Lambda}T}\mathbf{w}_{0}(0)+\int_{0}^{T}e^{\boldsymbol{\Lambda}(T-s)}\boldsymbol{\psi}(s)\,ds=\mathbf{w}_{0}(0).
	\]
	This gives the initial condition for the periodic solution as
	\begin{equation}
	\mathbf{w}_{0}(0)=\left[\mathbf{I}-e^{\boldsymbol{\Lambda}T}\right]^{-1}\int_{0}^{T}e^{\boldsymbol{\Lambda}(T-s)}\boldsymbol{\psi}(s)\,ds.\label{eq:init_cond}
	\end{equation}
	Note that the matrix $\left[\mathbf{I}-e^{\boldsymbol{\Lambda}T}\right]$
	is invertible due to the non-resonance condition \eqref{eq:nonresonance-1}.
	Substitution of the initial condition from \eqref{eq:init_cond} into
	the general solution \eqref{eq:Duhamel_firstorder}, we obtain an
	expression for the unique $T$-periodic solution $\mathbf{w}_{0}(t)$
	to \eqref{eq:modalfirstorder} as 
	\begin{align}
	\mathbf{w}_{0}(t) & =e^{\boldsymbol{\Lambda}t}\left[\mathbf{I}-e^{\boldsymbol{\Lambda}T}\right]^{-1}\int_{0}^{T}e^{\boldsymbol{\Lambda}(T-s)}\boldsymbol{\psi}(s)\,ds+\int_{0}^{t}e^{\boldsymbol{\Lambda}(t-s)}\boldsymbol{\psi}(s)\,ds\nonumber \\
	& =e^{\boldsymbol{\Lambda}T}\left[\mathbf{I}-e^{\boldsymbol{\Lambda}T}\right]^{-1}\int_{0}^{T}e^{\boldsymbol{\Lambda}(t-s)}\boldsymbol{\psi}(s)\,ds+\int_{0}^{t}e^{\boldsymbol{\Lambda}(t-s)}\boldsymbol{\psi}(s)\,ds\label{eq:temp2}\\
	& =e^{\boldsymbol{\Lambda}T}\left[\mathbf{I}-e^{\boldsymbol{\Lambda}T}\right]^{-1}\int_{0}^{t}e^{\boldsymbol{\Lambda}(t-s)}\boldsymbol{\psi}(s)\,ds+e^{\boldsymbol{\Lambda}T}\left[\mathbf{I}-e^{\boldsymbol{\Lambda}T}\right]^{-1}\int_{t}^{T}e^{\boldsymbol{\Lambda}(t-s)}\boldsymbol{\psi}(s)\,ds+\int_{0}^{t}e^{\boldsymbol{\Lambda}(t-s)}\boldsymbol{\psi}(s)\,ds\nonumber \\
	& =\int_{0}^{t}\left[e^{\boldsymbol{\Lambda}T}\left[\mathbf{I}-e^{\boldsymbol{\Lambda}T}\right]^{-1}+\mathbf{I}\right]e^{\boldsymbol{\Lambda}(t-s)}\boldsymbol{\psi}(s)\,ds+\int_{t}^{T}e^{\boldsymbol{\Lambda}T}\left[\mathbf{I}-e^{\boldsymbol{\Lambda}T}\right]^{-1}e^{\boldsymbol{\Lambda}(t-s)}\boldsymbol{\psi}(s)\,ds\nonumber \\
	& =\int_{0}^{T}\underbrace{\left[e^{\boldsymbol{\Lambda}T}\left[\mathbf{I}-e^{\boldsymbol{\Lambda}T}\right]^{-1}+h(t-s)\mathbf{I}\right]e^{\boldsymbol{\Lambda}(t-s)}}_{\mathbf{G}(t-s,T)}\boldsymbol{\psi}(s)\,ds\,,\nonumber 
	\end{align}
	where $\mathbf{G}(t,T)$ is a diagonal matrix with the entries given
	by \eqref{eq:G_firstorder}. Using the linear modal transformation
	$\mathbf{z}_{P}=\mathbf{Vw}_{0}$, we find the unique $T$-periodic
	solution to \eqref{eq:firstorderlinear} in the form
	\[
	\mathbf{z}_{P}(t)=\mathbf{V}\int_{0}^{T}\mathbf{G}(t-s,T)\boldsymbol{\psi}(s)\,ds.
	\]
	
	\section{Proof of Theorem \ref{Theorem:IE_P}\label{sec:ProofT1}}
	
	If $\mathbf{z}(t)$ is a $T$-periodic solution of \eqref{eq:firstorder},
	then it satisfies the linear inhomogeneous differential equation
	\[
	\mathbf{B}\dot{\mathbf{z}}=\mathbf{A}\mathbf{z}+\mathbf{F}(t)-\mathbf{R}(\mathbf{z}(t)),
	\]
	where we view $\mathbf{F}(t)-\mathbf{R}(\mathbf{z}(t))$ as a $T-$periodic
	forcing term. Thus, according to Lemma \ref{Lemma:1}, we have
	\[
	\mathbf{z}(t)=\mathbf{V}\int_{0}^{T}\mathbf{G}(t-s,T)\mathbf{V}^{-1}\left[\mathbf{F}(s)-\mathbf{R}(\mathbf{z}(s))\right]\,ds,
	\]
	as claimed in statement (\emph{i}). 
	
	Now, let $\mathbf{z}(t)$ be a continuous, $T-$periodic solution to
	\eqref{eq:int_periodic}. After introducing the notation $\boldsymbol{\chi}(t)=\mathbf{V}^{-1}\left[\mathbf{F}(t)-\mathbf{R}(\mathbf{z}(t))\right]$,
	we have 
	\begin{align}
	\mathbf{z}(t) & =\mathbf{V}\int_{0}^{T}\mathbf{G}(t-s,T)\boldsymbol{\chi}(s)\,ds,\nonumber \\
	& =\mathbf{V}e^{\boldsymbol{\Lambda}t}\left(\mathbf{I}-e^{\boldsymbol{\Lambda}T}\right)^{-1}\int_{0}^{T}e^{\boldsymbol{\Lambda}(T-s)}\boldsymbol{\chi}(s)\,ds+\mathbf{V}\int_{0}^{t}e^{\boldsymbol{\Lambda}(t-s)}\boldsymbol{\chi}(s)\,ds\,,\label{eq:temp1}
	\end{align}
	where \eqref{eq:temp1} is a direct consequence of \eqref{eq:temp2}.
	By the continuity of $\mathbf{z}(t)$, $\boldsymbol{\chi}(t)$ is
	also at least $C^{0}$ ($\mathbf{F}$ is at least $C^{0}$ and $\mathbf{\ensuremath{R}}$
	is Lipschitz). Thus, for any $t\in[0,T]$, the right-hand side of
	\eqref{eq:temp1} can be differentiated with respect to $t$ according
	to the Leibniz rule, to obtain 
	\begin{align*}
	\frac{d\mathbf{z}(t)}{dt} & =\mathbf{V}\boldsymbol{\Lambda}e^{\boldsymbol{\Lambda}t}\left(\mathbf{I}-e^{\boldsymbol{\Lambda}T}\right)^{-1}\int_{0}^{T}e^{\boldsymbol{\Lambda}(T-s)}\boldsymbol{\chi}(s)\,ds+\mathbf{V}\boldsymbol{\Lambda}\int_{0}^{t}e^{\boldsymbol{\Lambda}(t-s)}\boldsymbol{\chi}(s)\,ds\,+\mathbf{V}\boldsymbol{\Lambda}\boldsymbol{\chi}(t)\\
	& =\mathbf{V}\boldsymbol{\Lambda}\mathbf{V}^{-1}\mathbf{z}(t)+\mathbf{V}\boldsymbol{\Lambda}\boldsymbol{\chi}(t)\qquad\text{(using \eqref{eq:temp1})}\\
	& =\mathbf{V}\boldsymbol{\Lambda}\mathbf{V}^{-1}\mathbf{z}(t)+\mathbf{V}\boldsymbol{\Lambda}\mathbf{V}^{-1}\left[\mathbf{F}(t)-\mathbf{R}(\mathbf{z}(t))\right],
	\end{align*}
	which implies
	\[
	\mathbf{B}\frac{d\mathbf{z}(t)}{dt}=\mathbf{A}\mathbf{z}(t)+\mathbf{F}(t)-\mathbf{R}(\mathbf{z}(t)),
	\]
	as claimed in statement (\emph{ii}). 
	\section{Proof of Lemma \ref{Lemma:QP_linear}\label{sec:ProofLemma2}}
	By the linearity of \eqref{eq:firstorderlinear}, one can verify that
	the sum of periodic solutions given by Lemma~\ref{Lemma:1} for each
	periodic forcing summand in \eqref{eq:QP_forcing} is the unique,
	bounded solution of \eqref{eq:firstorderlinear}. In case of a forcing
	written as a Fourier series, we can carry out the integration appearing
	in Lemma \ref{Lemma:1} for each summand in this bounded solution
	explicitly in diagonalized coordinates $\mathbf{z}=\mathbf{V}\mathbf{w}$.
	With the notation $\boldsymbol{\psi_{\boldsymbol{\kappa}}}=\mathbf{V}^{-1}\mathbf{F}_{\boldsymbol{\kappa}}$,
	we then obtain for the $j^{\text{th}}$ degree of freedom: 
	{\small
		\begin{align*}
		w_{j}(t) & =\sum_{\boldsymbol{\kappa}\in\mathbb{Z}^{k}}\int_{0}^{T_{\boldsymbol{\kappa}}}G_{j}(t-s,T_{\boldsymbol{\kappa}})\psi_{\boldsymbol{\kappa},j}e^{i\left\langle \boldsymbol{\kappa},\boldsymbol{\Omega}\right\rangle s}\,ds\\
		& =\sum_{\boldsymbol{\kappa}\in\mathbb{Z}^{k}}\int_{0}^{T_{\boldsymbol{\kappa}}}e^{\lambda_{j}t}\left(\frac{e^{\lambda_{j}T_{\boldsymbol{\kappa}}}}{1-e^{\lambda_{j}T_{\boldsymbol{\kappa}}}}+h(t)\right)\psi_{\boldsymbol{\kappa},j}e^{i\left\langle \boldsymbol{\kappa},\boldsymbol{\Omega}\right\rangle s}\,ds\\
		& =\sum_{\boldsymbol{\kappa}\in\mathbb{Z}^{k}}\int_{0}^{t}e^{\lambda_{j}(t-s)}\left(\frac{1}{1-e^{\lambda_{j}T_{\boldsymbol{\kappa}}}}\right)\psi_{\boldsymbol{\kappa},j}e^{i\left\langle \boldsymbol{\kappa},\boldsymbol{\Omega}\right\rangle s}\,ds+\sum_{\boldsymbol{\kappa}\in\mathbb{Z}^{k}}\int_{t-s}^{T_{\boldsymbol{\kappa}}}e^{\lambda_{j}(t-s)}\left(\frac{e^{\lambda_{j}T_{\boldsymbol{\kappa}}}}{1-e^{\lambda_{j}T_{\boldsymbol{\kappa}}}}\right)\psi_{\boldsymbol{\kappa},j}e^{i\left\langle \boldsymbol{\kappa},\boldsymbol{\Omega}\right\rangle s}\,ds\\
		&=\sum_{\boldsymbol{\kappa}\in\mathbb{Z}^{k}}e^{\lambda_{j}t}\left(\frac{1}{1-e^{\lambda_{j}T_{\boldsymbol{\kappa}}}}\right)\psi_{\boldsymbol{\kappa},j}\int_{0}^{t}e^{-\lambda_{j}s}e^{i\left\langle \boldsymbol{\kappa},\boldsymbol{\Omega}\right\rangle s}\,ds+\sum_{\boldsymbol{\kappa}\in\mathbb{Z}^{k}}e^{\lambda_{j}t}\left(\frac{e^{\lambda_{j}T_{\boldsymbol{\kappa}}}}{1-e^{\lambda_{j}T_{\boldsymbol{\kappa}}}}\right)\psi_{\boldsymbol{\kappa},j}\int_{t}^{T_{\boldsymbol{\kappa}}}e^{-\lambda_{j}s}e^{i\left\langle \boldsymbol{\kappa},\boldsymbol{\Omega}\right\rangle s}\,ds\\
		& =\sum_{\boldsymbol{\kappa}\in\mathbb{Z}^{k}}e^{\lambda_{j}t}\left(\frac{1}{1-e^{\lambda_{j}T_{\boldsymbol{\kappa}}}}\right)\psi_{\boldsymbol{\kappa},j}\frac{1}{i\left\langle \boldsymbol{\kappa},\boldsymbol{\Omega}\right\rangle -\lambda_{j}}\left.e^{(i\left\langle \boldsymbol{\kappa},\boldsymbol{\Omega}\right\rangle -\lambda_{j})s}\right|_{s=0}^{s=t} +
		\end{align*}
		\begin{align*}
		\qquad& \quad \sum_{\boldsymbol{\kappa}\in\mathbb{Z}^{k}}e^{\lambda_{j}t}\left(\frac{e^{\lambda_{j}T_{\boldsymbol{\kappa}}}}{1-e^{\lambda_{j}T_{\boldsymbol{\kappa}}}}\right)\psi_{\boldsymbol{\kappa},j}\frac{1}{i\left\langle \boldsymbol{\kappa},\boldsymbol{\Omega}\right\rangle -\lambda_{j}}\left.e^{(i\left\langle \boldsymbol{\kappa},\boldsymbol{\Omega}\right\rangle -\lambda_{j})s}\right|_{s=t}^{s=T_{\boldsymbol{\kappa}}}\\
		&=\sum_{\boldsymbol{\kappa}\in\mathbb{Z}^{k}}e^{\lambda_{j}t}\left(\frac{1}{1-e^{\lambda_{j}T_{\boldsymbol{\kappa}}}}\right)\psi_{\boldsymbol{\kappa},j}\frac{1}{i\left\langle \boldsymbol{\kappa},\boldsymbol{\Omega}\right\rangle -\lambda_{j}}\left[e^{(i\left\langle \boldsymbol{\kappa},\boldsymbol{\Omega}\right\rangle -\lambda_{j})t}-1+e^{\lambda_{j}T_{\boldsymbol{\kappa}}}(e^{(i\left\langle \boldsymbol{\kappa},\boldsymbol{\Omega}\right\rangle -\lambda_{j})T_{\boldsymbol{\kappa}}}-e^{(i\left\langle \boldsymbol{\kappa},\boldsymbol{\Omega}\right\rangle -\lambda_{j})t})\right]\\
		& =\sum_{\boldsymbol{\kappa}\in\mathbb{Z}^{k}}e^{\lambda_{j}t}\psi_{\boldsymbol{\kappa},j}\frac{1}{i\left\langle \boldsymbol{\kappa},\boldsymbol{\Omega}\right\rangle -\lambda_{j}}\frac{1}{1-e^{\lambda_{j}T_{\boldsymbol{\kappa}}}}\left[(1-e^{\lambda_{j}T_{\boldsymbol{\kappa}}})(e^{(i\left\langle \boldsymbol{\kappa},\boldsymbol{\Omega}\right\rangle -\lambda_{j})t})-1+e^{i\left\langle \boldsymbol{\kappa},\boldsymbol{\Omega}\right\rangle T_{\boldsymbol{\kappa}}}\right]\\
		& =\sum_{\boldsymbol{\kappa}\in\mathbb{Z}^{k}}\frac{1}{i\left\langle \boldsymbol{\kappa},\boldsymbol{\Omega}\right\rangle -\lambda_{j}}\psi_{\boldsymbol{\kappa},j}e^{i\left\langle \boldsymbol{\kappa},\boldsymbol{\Omega}\right\rangle t}\,.
		\end{align*}}
	\section{Explicit Green's function for mechanical systems: Proof of Lemma~\ref{Lemma:special_case_QP_linear} \label{sec:ProofLemma3}}
	
	The first-order ODE formulation for \eqref{eq:forced_modal_component-1}
	is given by
	\begin{equation}
	\frac{d}{dt}\left(\begin{array}{c}
	y_{j}\\
	\dot{y}_{j}
	\end{array}\right)=\underbrace{\left(\begin{array}{cc}
		0 & 1\\
		-\omega_{0,j}^{2} & -2\zeta_{j}\omega_{0,j}
		\end{array}\right)}_{\mathbf{A}_{j}}\left(\begin{array}{c}
	y_{j}\\
	\dot{y}_{j}
	\end{array}\right)+\left(\begin{array}{c}
	0\\
	\varphi_{j}(t)
	\end{array}\right),\qquad j=1,\ldots,n.\label{eq:modalsys}
	\end{equation}
	By the classic variation of constants formula for first-order systems
	of ordinary differential equations, the general solution of \eqref{eq:modalsys}
	is of the form 
	\begin{equation}
	\left(\begin{array}{c}
	y_{j}(t)\\
	\dot{y}_{j}(t)
	\end{array}\right)=\mathbf{N}^{(j)}(t)\left(\begin{array}{c}
	y_{j}(0)\\
	\dot{y}_{j}(0)
	\end{array}\right)+\int_{0}^{t}\mathbf{N}^{(j)}(t-s)\left(\begin{array}{c}
	0\\
	\varphi_{j}(s)
	\end{array}\right)ds,\qquad j=1,\ldots,n,\label{eq:Duhamel_j}
	\end{equation}
	with $\mathbf{N}(t)=e^{\mathbf{A}_{j}t}$ denoting the fundamental
	matrix solution for the $j^{th}$ mode with $\mathbf{N}(0)=\mathbf{I}.$
	Thus, the homogeneous (unforced) version of \eqref{eq:forced_modal_component-1},
	the explicit solution can be obtained as
	\begin{equation}
	\left(\begin{array}{c}
	y_{j}(t)\\
	\dot{y}_{j}(t)
	\end{array}\right)=\mathbf{N}^{(j)}(T)\left(\begin{array}{c}
	y_{j}(0)\\
	\dot{y}_{j}(0)
	\end{array}\right),\qquad j=1,\ldots,n.\label{eq:Ndef}
	\end{equation}
	Since $\mathbf{F}(t)$ is uniformly bounded for all times and all
	$\mathbf{A}_{j}$ matrices are \emph{hyperbolic} ($\zeta_{j}>0$ for
	$j=1,...,n$), then a unique uniformly bounded solution exists for
	the $2n$-dimensional system of linear ordinary differential equations
	(ODEs) \eqref{eq:modalsys} (see, e.g., Burd \cite{burd07}). The
	initial condition $\left(y_{j}(0),\dot{y}_{j}(0)\right)$ for the
	unique $T$-periodic solution of \eqref{eq:Duhamel_j} is obtained
	by imposing periodicity, i.e., $y_{j}(0)=y_{j}(T)$ for $j=1,\dots,n$
	and is given by
	{\small
		\begin{align}
		\left(\begin{array}{c}
		y_{j}(0)\\
		y_{j}(0)
		\end{array}\right) & =\frac{1}{1-\text{Trace}\left(\mathbf{N}^{(j)}(T)\right)+\text{det}\left(\mathbf{N}^{(j)}(T)\right)}\int_{0}^{T}\left(\begin{array}{c}
		\left(1-N_{22}^{(j)}(T)\right)N_{12}^{(j)}(T-s)+N_{12}^{(j)}(T)N_{22}^{(j)}(T-s)\\
		\left(1-N_{11}^{(j)}(T)\right)N_{22}^{(j)}(T-s)+N_{21}^{(j)}(T)N_{12}^{(j)}(T-s)
		\end{array}\right)\varphi_{j}(s)\,ds\,.\label{eq:IC}
		\end{align}
	}
	Finally, the unique periodic response $\left(y_{j}(t),\dot{y}_{j}(t)\right)$
	is obtained by substituting the initial condition \eqref{eq:IC} into
	the Duhamel's integral formula \eqref{eq:Duhamel_j} as
	{\small
		\begin{align}
		\left(\begin{array}{c}
		y_{j}(t)\\
		\dot{y}_{j}(t)
		\end{array}\right) & =\frac{\mathbf{N}^{(j)}(t)}{1-\text{Trace}\left(\mathbf{N}^{(j)}(T)\right)+\text{det}\left(\mathbf{N}^{(j)}(T)\right)}\int_{0}^{T}\left(\begin{array}{c}
		\left(1-N_{22}^{(j)}(T)\right)N_{12}^{(j)}(T-s)+N_{12}^{(j)}(T)N_{22}^{(j)}(T-s)\\
		\left(1-N_{11}^{(j)}(T)\right)N_{22}^{(j)}(T-s)+N_{21}^{(j)}(T)N_{12}^{(j)}(T-s)
		\end{array}\right)\varphi_{j}(s)\,ds\nonumber \\
		& \quad+\int_{0}^{t}\mathbf{N}^{(j)}(t-s)\left(\begin{array}{c}
		0\\
		\varphi_{j}(s)
		\end{array}\right)ds,\qquad j=1,\ldots,n\,.\label{particular_per_sol}
		\end{align}
	}
	With the notation introduced in \eqref{eq:alpha_beta}, i.e., 
	\[
	\alpha_{j}:=\text{Re}(\lambda_{2j}),\quad\omega_{j}:=|\text{Im}(\lambda_{2j})|,\quad\beta_{j}:=\alpha_{j}+\omega_{j},\quad\gamma_{j}:=\alpha_{j}-\omega_{j,}\quad j=1,\dots,n,
	\]
	the specific expressions for the fundamental matrix of solutions of
	\eqref{eq:modalsys} in the under-damped, the critically-damped and
	the over-damped case are given by 
	\begin{equation}
	\mathbf{N}^{(j)}(t)=\begin{cases}
	\frac{e^{\alpha_{j}t}}{\omega_{j}}\left(\begin{array}{cc}
	\omega_{j}\cos\omega_{j}t-\alpha_{j}\sin\omega_{j}t & \sin\omega_{j}t\\
	-\alpha_{j}^{2}\sin\omega_{j}t-\omega_{j}^{2}\sin\omega_{j}t & \omega_{j}\cos\omega_{j}t+\alpha_{j}\sin\omega_{j}t
	\end{array}\right), & \zeta_{j}<1\\
	\left(\begin{array}{cc}
	e^{\alpha_{j}t}-\alpha_{j}te^{\alpha_{j}t} & te^{\alpha_{j}t}\\
	-\alpha_{j}^{2}te^{\alpha_{j}t} & e^{\alpha_{j}t}+\alpha_{j}te^{\alpha_{j}t}
	\end{array}\right)\,, & \zeta_{j}=1\\
	\frac{1}{\beta_{j}-\gamma_{j}}\left(\begin{array}{cc}
	\beta_{j}e^{\gamma_{j}t}-\alpha_{j}e^{\beta_{j}t} & e^{\beta_{j}t}-e^{\gamma_{j}t}\\
	\gamma_{j}\beta_{j}\left(e^{\gamma_{j}t}-e^{\beta_{j}t}\right) & \beta_{j}e^{\beta_{j}t}-\gamma_{j}e^{\gamma_{j}t}
	\end{array}\right)\,, & \zeta_{j}>1
	\end{cases},\quad j=1,\dots,n\,.\label{eq:fund_sol}
	\end{equation}
	Furthermore, we have 
	\begin{align}
	\mathrm{Trace\,}\mathbf{N}^{(j)}(t) & =\begin{cases}
	2e^{\alpha_{j}t}\cos\omega_{j}t\,, & \zeta_{j}<1\\
	2e^{\beta_{j}t}\,, & \zeta_{j}=1\\
	e^{\beta_{j}t}+e^{\gamma_{j}t}\,, & \zeta_{j}>1
	\end{cases},\quad j=1,\dots,n\,.\nonumber \\
	\det\mathbf{N}^{(j)}(t) & =\begin{cases}
	e^{2\alpha_{j}t}\,, & \zeta_{j}<1\\
	e^{2\beta_{j}t}\,, & \zeta_{j}=1\\
	e^{\left(\beta_{j}+\gamma_{j}\right)t}\,, & \zeta_{j}>1
	\end{cases},\quad j=1,\dots,n\,.\label{eq:N invariants}
	\end{align}
	Thus, we can explicitly compute the particular periodic solution given
	in \eqref{particular_per_sol} using \eqref{eq:fund_sol} as
	\[
	\mathbf{y}(t)=\int_{0}^{T}\mathbf{L}(t-s,T)\boldsymbol{\varphi}(s)\,ds,\quad\mathbf{L}(t-s,T)=\mathrm{diag}\left(L_{1}(t-s,T),\ldots,L_{n}(t-s,T)\right)\in\mathbb{R}^{n\times n},
	\]
	\[
	\dot{\mathbf{y}}(t)=\int_{0}^{T}\mathbf{J}(t-s,T)\boldsymbol{\varphi}(s)\,ds,\quad\mathbf{J}(t-s,T)=\mathrm{diag}\left(J_{1}(t-s,T),\ldots,J_{n}(t-s,T)\right)\in\mathbb{R}^{n\times n},
	\]
	with the diagonal elements of the Green's function matrices $ \mathbf{L,J} $ defined in
	\eqref{eq:G_pos} and \eqref{eq:J_vel},i.e., 
	\[
	L_{j}(t,T)=\begin{cases}
	\frac{e^{\alpha_{j}t}}{\omega_{j}}\left[\frac{e^{\alpha_{j}T}\left[\sin\omega_{j}(T+t)-e^{\alpha_{j}T}\sin\omega_{j}t\right]}{1+e^{2\alpha_{j}T}-2e^{\alpha_{j}T}\cos\omega_{j}T}+h(t)\sin\omega_{j}t\right], & \zeta_{j}<1\\
	\frac{e^{\alpha_{j}(T+t)}\left[\left(1-e^{\alpha_{j}T}\right)t+T\right]}{\left(1-e^{\alpha_{j}T}\right)^{2}}+h(t)te^{\alpha_{j}t}\,, & \zeta_{j}=1\\
	\frac{1}{(\beta_{j}-\gamma_{j})}\left[\frac{e^{\beta_{j}(T+t)}\left(1-e^{\gamma_{j}T}\right)-e^{\gamma_{j}(T+t)}\left(1-e^{\beta_{j}T}\right)}{1-e^{\gamma_{j}T}-e^{\beta_{j}T}+e^{\left(\gamma_{j}+\beta_{j}\right)T}}+h(t)\left(e^{\beta_{j}t}-e^{\gamma_{j}t}\right)\right], & \zeta_{j}>1
	\end{cases},
	\]
	\[
	\begin{split}J_{j}(t,T) & =\begin{cases}
	\begin{array}{c}
	\frac{e^{\alpha_{j}t}}{\omega_{j}}\left[\frac{e^{\alpha_{j}T}\left[\omega_{j}\left(\cos\omega_{j}(T+t)-e^{\alpha_{j}T}\cos\omega_{j}t\right)+\alpha_{j}\left(\sin\omega_{j}(T+t)-e^{\alpha_{j}T}\sin\omega_{j}t\right)\right]}{1+e^{2\alpha_{j}T}-2e^{\alpha_{j}T}\cos\omega_{j}T}+\right.\\
	\left.h(t)\left(\frac{1}{\omega_{j}}\cos\omega_{j}t+\alpha_{j}\sin\omega_{j}t\right)\right]
	\end{array}\quad, & \zeta_{j}<1\\
	\frac{e^{\alpha_{j}(T+t)}\left[\left(1-e^{\alpha_{j}T}\right)\left(1+\alpha_{j}t\right)+\alpha_{j}T\right]}{\left(1-e^{\alpha_{j}T}\right)^{2}}+h(t)\left(e^{\alpha_{j}t}+\alpha_{j}te^{\alpha_{j}t}\right)\,, & \zeta_{j}=1\\
	\frac{1}{(\beta_{j}-\gamma_{j})}\left[\frac{\beta_{j}e^{\beta_{j}(T+t)}\left(1-e^{\gamma_{j}T}\right)-\gamma_{j}e_{j}^{\gamma_{j}(T+t)}\left(1-e^{\beta_{j}T}\right)}{1-e^{\gamma_{j}T}-e^{\beta_{j}T}+e^{\left(\gamma_{j}+\beta_{j}\right)T}}+h(t)\left(\beta_{j}e^{\beta_{j}t}-\gamma_{j}e^{\gamma_{j}t}\right)\right], & \zeta_{j}>1
	\end{cases}\end{split}
	\,,\quad j=1,\dots,n\,
	\,.
	\]
	
	Finally, the linear periodic response $\mathbf{x}_{P}(t)$ in the original
	system coordinates can then obtained by the linear transformation
	$\mathbf{x}_{P}(t)=\mathbf{Uy}(t)$ as 
	\[
	\mathbf{x}_{P}(t)=\mathbf{U}\int_{0}^{T}\mathbf{L}(t-s,T)\mathbf{U}^{\top}\mathbf{f}(s)\,ds.
	\]
	
	\section{Derivative of Green's function with respect to $ T $}
	\label{app:derivative}
	The derivative with respect to the time period $ T $ of the first-order periodic Green's function $ \mathbf{G} $ given in~\eqref{eq:G_firstorder} is simply given by
	\begin{equation}
	\frac{\partial G_{j}}{\partial T}(t,T)=
	\lambda_je^{\lambda_{j}t}\frac{e^{\lambda_{j}T}}{(1-e^{\lambda_{j}T})^2},\quad j=1,\dots,2n\,.\label{eq:G_derv}
	\end{equation}
	We also provide the derivative of the Green's function $\mathbf{L}$
	with respect to $T$ to ease the computation of the Jacobian of the
	zero function in during numerical continuation. This is obtained by simply differentiating \eqref{eq:G_pos} with respect to T. We use a symbolic toolbox for this procedure: 
	{\small
		\[\frac{dL_{j}}{dT}(t,T)=\begin{cases}
		\begin{array}{c}
		\frac{e^{\alpha_{j}(t+T)}}{\omega_{j}\left(1+e^{2\alpha_{j}T}-2e^{\alpha_{j}T}\cos\omega_{j}T\right)^{2}}\left[\omega_{j}\cos\omega_{j}(T+t)+\alpha_{j}\sin\omega_{j}(T+t)\right.-\\
		\left.2e^{\alpha_{j}T}\left(\omega_{j}\cos\omega_{j}t+\alpha_{j}\sin\omega_{j}t\right)+e^{2\alpha_{j}T}\left(\alpha_{j}\sin\omega_{j}(t-T)+\omega_{j}\cos\omega_{j}(t-T)\right)\right]
		\end{array}, & \zeta_{j}<1\\
		\frac{\alpha_{j}e^{\alpha_{j}(t+T)}}{(e^{\alpha_{j}T}-1)^{2}}\left[t+T-2te^{\alpha_{j}T}+1-\frac{2e^{\alpha_{j}T}(T-t(e^{\alpha_{j}T}-1))}{(e^{\alpha_{j}T}-1)}\right] & \zeta_{j}=1\\
		\begin{array}{c}
		\frac{1}{(\beta_{j}-\gamma_{j})}\left[\frac{\left(\beta_{j}-\left(\gamma_{j}+\beta_{j}\right)e^{\gamma_{j}T}\right)e^{\beta_{j}(T+t)}+\left(\left(\gamma_{j}+\beta_{j}\right)e^{\beta_{j}T}-\gamma_{j}\right)e^{\gamma_{j}(T+t)}}{1-e^{\gamma_{j}T}-e^{\beta_{j}T}+e^{\left(\gamma_{j}+\beta_{j}\right)T}}+\right.\\
		\left.\frac{\left(e^{\beta_{j}(T+t)}\left(1-e^{\gamma_{j}T}\right)-e^{\gamma_{j}(T+t)}\left(1-e^{\beta_{j}T}\right)\right)\left(\gamma_{j}e^{\gamma_{j}T}+\beta_{j}e^{\beta_{j}T}-\left(\gamma_{j}+\beta_{j}\right)e^{\left(\gamma_{j}+\beta_{j}\right)T}\right)}{\left(1-e^{\gamma_{j}T}-e^{\beta_{j}T}+e^{\left(\gamma_{j}+\beta_{j}\right)T}\right)^{2}}\right]
		\end{array}\,, & \zeta_{j}>1
		\end{cases},\quad j=1,\dots,n\,.\]
	}
	
	\section{Proof of Remark \ref{rmk:Gamma_T}\label{app:Gamma_T_estimate}}
	
	We derive an estimate for the sup norm of the integral of the operator
	norm of the Green's function, i.e., for $\int_{0}^{T}\left\Vert G_{j}(t-s,T)\right\Vert \,ds$ defined in equation~\eqref{eq:G_firstorder}. For $t>s$, we start
	by noting that
	
	\begin{align*}
	\left|G_{j}(t-s,T)\right|  =\left|e^{\lambda_{j}(t-s)}\left(\frac{1}{1-e^{\lambda_{j}T}}\right)\right|\leq\left|e^{\lambda_{j}(t-s)}\right|\frac{1}{\left|1-e^{\lambda_{j}T}\right|}\leq\frac{\max(\left|e^{\lambda_{j}t}\right|,1)}{\left|1-e^{\lambda_{j}T}\right|},\quad0\leq s\leq t<T.
	\end{align*}
	For the case $T>s>t$, we obtain

	\begin{align*}
	\left|G_{j}(t-s,T)\right|  =\left|e^{\lambda_{j}(t-s)}\left(\frac{e^{\lambda_{j}T}}{1-e^{\lambda_{j}T}}\right)\right|&\leq\max\left(\left|\left(\frac{e^{\lambda_{j}T}}{1-e^{\lambda_{j}T}}\right)\right|,\left|e^{\lambda_{j}(t-T)}\left(\frac{e^{\lambda_{j}T}}{1-e^{\lambda_{j}T}}\right)\right|\right)\\
	&\leq\frac{\max(\left|e^{\lambda_{j}t}\right|,1)}{\left|1-e^{\lambda_{j}T}\right|},\quad 0\leq s\leq t<T.
	\end{align*}
	The upper bounds on the Green's function in the two intervals are
	equal and we therefore obtain
	\begin{align*}
	\left\Vert \int_{0}^{T}\left\Vert \mathbf{G}(t-s,T)\right\Vert \,ds\right\Vert _{0} & =\max_{t\in[0,T]}\int_{0}^{T}\left\Vert \mathbf{G}(t-s,T)\right\Vert \,ds\\
	& \leq\max_{t\in[0,T]}\int_{0}^{T}\max_{1\leq j\leq n}\frac{\max(\left|e^{\lambda_{j}t}\right|,1)}{\left|1-e^{\lambda_{j}T}\right|}\,ds\\
	& \leq\max_{1\leq j\leq n}\frac{T\max(\left|e^{\lambda_{j}T}\right|,1)}{\left|1-e^{\lambda_{j}T}\right|}\eqqcolon\Gamma(T).
	\end{align*}

	\section{Proof of Theorem \ref{Thm:Picard_it_periodic} \label{App:Proof_Picard_it_per}}
	In the following, we derive conditions under which the mapping~$ \mathcal{\boldsymbol{\mathcal{H}}} $ defined in equation \eqref{eq:IE_periodic} is a contraction mapping. We rewrite \eqref{eq:IE_periodic} as 
	\[
	\mathbf{z}(t)=\Upsilon_{P}(\mathbf{F}(t)-\mathbf{R}(\mathbf{z}(t))):=\int_{0}^{T}\mathbf{V}\mathbf{G}(t-s,T)\mathbf{V}^{-1}\left[\mathbf{F}(s)-\mathbf{R}(\mathbf{z}(s))\right]\,ds,
	\]
	where $\Upsilon_{P}$ is a linear map representing the convolution
	operation with the Green's function. \textcolor{black}{Specifically,
		we define the space of $n$-dimensional periodic $T$-periodic functions
		as
		\begin{equation}
		\mathcal{P}_{n}:=\{\mathbf{p}:\mathbb{R}\to\mathbb{R}^{n},\mathbf{p}\in C^{0},\mathbf{p}(t)=\mathbf{p}(t+T)\forall t\in\mathbb{R}\}.\label{eq:P_k}
		\end{equation}
		Under the non-resonance condition \eqref{eq:nonresonance-1}, the
		linear map \[\Upsilon_{P}:\mathcal{P}_{2n}\to\mathcal{P}_{2n},\,\Upsilon_{P}\mathbf{p}=\mathbf{V}\int_{0}^{T}\mathbf{G}(t-s,T)\mathbf{V}^{-1}\mathbf{p}(s)\,ds\]
		is well-defined, i.e., $\Upsilon_{P}$ maps $T$-periodic functions
		into $T$-periodic functions. Indeed, for any $\mathbf{p}\in\mathcal{P}_{2n}$,
		let $\mathbf{q}=\Upsilon_{P}\mathbf{p}$. We have
		\begin{align*}
		\mathbf{q}(t) & =\mathbf{V}\int_{0}^{T}\mathbf{G}(t-s,T)\mathbf{V}^{-1}\mathbf{p}(s)\,ds=\mathbf{V}\int_{0}^{T}\mathbf{G}(t+T-(s+T),T)\mathbf{V}^{-1}\mathbf{p}(s)\,ds\\
		& =\mathbf{V}\int_{0}^{T}\mathbf{G}(t+T-(s+T),T)\mathbf{V}^{-1}\mathbf{p}(s+T)\,ds=\mathbf{V}\int_{T}^{2T}\mathbf{G}(t+T-\sigma,T)\mathbf{V}^{-1}\mathbf{p}(\sigma)\,d\sigma\\
		& =\mathbf{q}(t+T)\,,
		\end{align*}
		i.e., $\mathbf{q}\in\mathcal{P}_{2n}$.}
	
	Since the space space \eqref{eq:Space_per_fcn}
	consists of periodic functions, we know that it is well-defined in
	the space $C_{\delta}^{\mathbf{z}_{0}}[0,T]$. Therefore, by the Banach
	fixed point theorem, the integral equation \eqref{eq:IE_periodic}
	has a unique solution if the mapping $\mathcal{\boldsymbol{\mathcal{H}}}$
	is a contraction of the complete metric space $C_{\delta}^{\mathbf{z}_{0}}[0,T]$
	into itself for an appropriate choice of the radius $\delta>0$ and
	the initial guess $\mathbf{z}_{0}$. 
	
	To find a condition under which this holds, we first note that for
	$\left\Vert \mathbf{z-}\mathbf{z}_{0}\right\Vert _{0}\leq\delta,$
	eq. \eqref{eq:IE_periodic} gives
	{
		\small
		\begin{align*}
		\left|\mathcal{\boldsymbol{\boldsymbol{\mathcal{H}}}}(\mathbf{z}(t))\right| & =\left|\int_{0}^{T}\mathbf{V}\mathbf{G}(t-s,T)\mathbf{V}^{-1}\left[\mathbf{F}(s)-\mathbf{R}(\mathbf{z}_{0}(s))+\mathbf{R}(\mathbf{z}_{0}(s))-\mathbf{R}(\mathbf{z}(s))\right]\,ds\right|\\
		& =\left\Vert \int_{0}^{T}\mathbf{V}\mathbf{G}(t-s,T)\mathbf{V}^{-1}\left[\mathbf{F}(s)-\mathbf{R}(\mathbf{z}_{0}(s))\right]\,ds\right\Vert _{0}+\left\Vert \int_{0}^{T}\mathbf{V}\mathbf{G}(t-s,T)\mathbf{V}^{-1}\left[\mathbf{R}(\mathbf{z}_{0}(s))-\mathbf{R}(\mathbf{z}(s))\right]\,ds\right\Vert _{0}\\
		& \leq\left\Vert \boldsymbol{\mathcal{E}}(\mathbf{z}_{0},t)\right\Vert _{0}+\left\Vert \mathbf{V}\right\Vert \left\Vert \mathbf{V}^{-1}\right\Vert L_{\delta}^{\mathbf{z}_{0}}\left\Vert \mathbf{z}-\mathbf{z}_{0}\right\Vert _{0}\int_{0}^{T}\left\Vert \mathbf{G}(t-s,T)\right\Vert \,ds\\
		& \leq\left\Vert \boldsymbol{\mathcal{E}}(\mathbf{z}_{0},t)\right\Vert _{0}+\left\Vert \mathbf{V}\right\Vert \left\Vert \mathbf{V}^{-1}\right\Vert L_{\delta}^{\mathbf{z}_{0}}\delta\left\Vert \int_{0}^{T}\left\Vert \mathbf{H}(t-s,T)\right\Vert \,ds\right\Vert _{0}\\
		& \leq\left\Vert \boldsymbol{\mathcal{E}}(\mathbf{z}_{0},t)\right\Vert _{0}+\delta\left\Vert \mathbf{V}\right\Vert \left\Vert \mathbf{V}^{-1}\right\Vert ^{2}L_{\delta}^{\mathbf{z}_{0}}\Gamma(T),
		\end{align*}}
	where $L_{\delta}^{\mathbf{z}_{0}}$ denotes a uniform-in-time Lipschitz
	constant for the function $\mathbf{S}(\mathbf{z})$ with respect
	to its argument $\mathbf{z}$ within the ball $\left|\mathbf{z}-\mathbf{z}_{0}\right|\leq\delta,$
	and $\Gamma(T)$ is the constant defined in \eqref{eq:GammaT}. The
	initial error term $\boldsymbol{\mathcal{E}}(t)$ is defined in eq.
	\eqref{eq:initial_error}. Taking the sup norm of both sides, we obtain
	that $\left\Vert \mathcal{\boldsymbol{\boldsymbol{\mathcal{H}}}}(\mathbf{z})\right\Vert _{0}\leq\delta,$
	and hence 
	\[
	\boldsymbol{\mathcal{H}}\colon C_{\delta}^{\mathbf{z}_{0}}[0,T]\to C_{\delta}^{\mathbf{z}_{0}}[0,T]
	\]
	holds, whenever condition \eqref{eq:delta_cond} holds. 
	
	Similarly, for two functions $\mathbf{z},\tilde{\mathbf{z}}\in C_{\delta}^{\mathbf{z}_{0}}[0,T],$
	eq. \eqref{eq:IE_periodic} gives the estimate
	\begin{align*}
	\left|\mathcal{\boldsymbol{\mathcal{H}}}(\mathbf{z}(t))-\boldsymbol{\mathcal{H}}(\tilde{\mathbf{z}}(t))\right| & \leq\left|\int_{0}^{T}\mathbf{V}\mathbf{G}(t-s,T)\mathbf{V}^{-1}\left[\mathbf{R}(\mathbf{z}(s))-\mathbf{R}(\tilde{\mathbf{z}}(s))\right]\,ds\right|\\
	& \leq2\left\Vert \mathbf{V}\right\Vert \left\Vert \mathbf{V}^{-1}\right\Vert L_{\delta}^{\mathbf{z}_{0}}\int_{0}^{T}\left\Vert \mathbf{G}(t-s,T)\right\Vert \,ds\left|\mathbf{z}(t)-\tilde{\mathbf{z}}(t)\right|\\
	& \leq2\left\Vert \mathbf{V}\right\Vert \left\Vert \mathbf{V}^{-1}\right\Vert L_{\delta}^{\mathbf{z}_{0}}\left\Vert \int_{0}^{T}\left\Vert \mathbf{G}(t-s,T)\right\Vert \,ds\right\Vert _{0}\left\Vert \mathbf{z}-\tilde{\mathbf{z}}\right\Vert _{0}\\
	& \leq2\left\Vert \mathbf{V}\right\Vert \left\Vert \mathbf{V}^{-1}\right\Vert L_{\delta}^{\mathbf{z}_{0}}\Gamma(T)\left\Vert \mathbf{z}-\tilde{\mathbf{z}}\right\Vert _{0}.
	\end{align*}
	Taking the sum norm of both sides then gives that $\boldsymbol{\mathcal{H}}$
	is a contraction mapping on $C_{\delta}^{\mathbf{z}_{0}}[0,T]$ if
	\begin{equation}
	2\left\Vert \mathbf{V}\right\Vert \left\Vert \mathbf{V}^{-1}\right\Vert L_{\delta}^{\mathbf{z}_{0}}\Gamma(T)<1/a,\label{eq:cond2}
	\end{equation}
	holds for some real number $a\geq1$. Solving equation \eqref{eq:cond2}
	for $L_{\delta}^{\mathbf{z}_{0}}$, we obtain condition \eqref{eq:Lipschitz_cond}.
	
	\section{Proof of Theorem \ref{Thm:Picard_it_qper}\label{App:Proof_Picard_it_qper}}
	We show here that the mapping $\boldsymbol{\mathcal{H}}$ defined
	in the quasi-periodic case (cf. equation \eqref{eq:Map_quasiper})
	is a contraction on the space \eqref{eq:Space_qper_funct} if the
	conditions \eqref{eq:Lipschitz_cond_qper} and \eqref{eq:delta_cond_qper}
	hold. The convergence estimate for the iteration \eqref{eq:Picard_it_qper}
	is then similar in spirit to the periodic case (cf. Appendix~\ref{App:Proof_Picard_it_per}). 
	
	We rewrite \eqref{eq:IE_QP} as 
	\[
	\mathbf{z}(t)=\Upsilon_{Q}(\mathbf{F}(t)-\mathbf{R}(\mathbf{z}(t))):=\mathbf{V}\sum_{\boldsymbol{\kappa}\in\mathbb{Z}^{k}}\mathbf{H}(T_{\kappa})\mathbf{V}^{-1}\left(\mathbf{F}_{\boldsymbol{\kappa}}-\mathbf{R}_{\boldsymbol{\kappa}}\{\mathbf{z}\}\right)e^{i\left\langle \boldsymbol{\kappa},\boldsymbol{\Omega}\right\rangle t},
	\]
	where $\Upsilon_{Q}$ is a linear map representing the convolution
	operation with the Green's function.\textcolor{black}{{} Similarly to
		the periodic case, we define the space of $n$-dimensional quasi-periodic
		functions with frequency base vector $\boldsymbol{\Omega}$ as
		\begin{equation}
		\mathcal{Q}_{2n}:=\{\mathbf{p}:\mathbb{T}^{k}\to\mathbb{R}^{n},\mathbf{p}\in C^{0}\}.\label{eq:quasip_functionspace}
		\end{equation}
		Furthermore, we note that under the non-resonance condition \eqref{eq:nonresonance-1},
		the linear map 
		\[\Upsilon_{Q}:\mathcal{Q}_{2n}\to\mathcal{Q}_{2n},\quad\Upsilon_{Q}\mathbf{q}=\mathbf{V}\sum_{\boldsymbol{\kappa}\in\mathbb{Z}^{k}}\int_{0}^{T_{\boldsymbol{\kappa}}}\mathbf{G}(t-s,T_{\boldsymbol{\kappa}})\mathbf{V}^{-1}\mathbf{q}(s)\,ds\]
		is well-defined, i.e., $\Upsilon_{Q}$ maps any quasi periodic function
		$\mathbf{q}$ with frequency base vector $\boldsymbol{\Omega}$ to
		quasi-periodic functions with the same frequency base vector $\boldsymbol{\Omega}$.
		This is a direct consequence of the linearity of the $\Upsilon_{Q}$
		and definition of $\Upsilon_{P}$ in Appendix \ref{App:Proof_Picard_it_per}.}

	Since the mapping \eqref{eq:Map_quasiper} is well-defined in the space
	$C_{\delta}^{\mathbf{z}_{0}}(\boldsymbol{\Omega})$ defined in \eqref{eq:Space_qper_funct}, we have by the
	Banach fixed point theorem that the integral equation \eqref{eq:Map_quasiper}
	has a unique solution if the mapping $\boldsymbol{\mathcal{H}}$ is
	a contraction of the complete metric space $C_{\delta}^{\mathbf{z}_{0}}(\boldsymbol{\Omega})$
	into itself for an appropriate choice of the radius $\delta>0.$ In
	a similar spirit as in the periodic case we search for conditions
	under which the space $C_{\delta}^{\mathbf{z}_{0}}(\boldsymbol{\Omega})$
	is mapped to itself. Therefore, we take the sup norm of the mapping
	\eqref{eq:Map_quasiper} applied to an element from $C_{\delta}^{\mathbf{z}_{0}}(\boldsymbol{\Omega})$
	and obtain
	{\small
		\begin{align*}
		\left|\mathcal{\boldsymbol{\mathcal{H}}}(\mathbf{z}(t))\right| & =\left|\mathbf{V}\sum_{\kappa\in\mathbb{Z}^{k}}\mathbf{H}(T_{\kappa})\mathbf{V}^{-1}\left[\mathbf{F_{\boldsymbol{\kappa}}}-\mathbf{R}_{\boldsymbol{\kappa}}\{\mathbf{z}_{0}\}+\mathbf{R}_{\boldsymbol{\kappa}}\{\mathbf{z}_{0}\}-\mathbf{R}_{\boldsymbol{\kappa}}\{\mathbf{z}\}\right]e^{i\left\langle \boldsymbol{\kappa},\boldsymbol{\Omega}\right\rangle t}\right|\\
		& \leq\left\Vert \mathbf{V}\sum_{\kappa\in\mathbb{Z}^{k}}\mathbf{H}(T_{\kappa})\mathbf{V}^{-1}\left[\mathbf{F_{\boldsymbol{\kappa}}}-\mathbf{R}_{\boldsymbol{\kappa}}\{\mathbf{z}_{0}\}\right]e^{i\left\langle \boldsymbol{\kappa},\boldsymbol{\Omega}\right\rangle t}\right\Vert _{0}+\left\Vert \mathbf{V}\sum_{\boldsymbol{\kappa}\in\mathbb{Z}^{k}}\mathbf{H}(T_{\kappa})\mathbf{V^{-1}}\left[\mathbf{R}_{\boldsymbol{\kappa}}\{\mathbf{z}_{0}\}-\mathbf{R}_{\boldsymbol{\kappa}}\{\mathbf{z}\}\right]e^{i\left\langle \boldsymbol{\kappa},\boldsymbol{\Omega}\right\rangle t}\right\Vert _{0}\\
		& \leq\left\Vert \boldsymbol{\mathcal{E}}(\mathbf{z}_{0},t)\right\Vert _{0}+\left\Vert \mathbf{V}\right\Vert \left\Vert \mathbf{V}^{-1}\right\Vert h_{max}\left\Vert \sum_{\boldsymbol{\kappa}\in\mathbb{Z}^{k}}\left[\mathbf{R}_{\boldsymbol{\kappa}}\{\mathbf{z}_{0}\}-\mathbf{R}_{\boldsymbol{\kappa}}\{\mathbf{z}\}\right]e^{i\left\langle \boldsymbol{\kappa},\boldsymbol{\Omega}\right\rangle t}\right\Vert _{0}\\
		& \leq\left\Vert \boldsymbol{\mathcal{E}}(\mathbf{z}_{0},t)\right\Vert _{0}+\left\Vert \mathbf{V}\right\Vert \left\Vert \mathbf{V}^{-1}\right\Vert h_{max}\left\Vert \mathbf{R}(\mathbf{z}_{0}(s),s)-\mathbf{R}(\mathbf{z}(s),s)\right\Vert _{0}\\
		& \leq\left\Vert \boldsymbol{\mathcal{E}}(\mathbf{z}_{0},t)\right\Vert _{0}+\delta\left\Vert \mathbf{V}\right\Vert \left\Vert \mathbf{V}^{-1}\right\Vert h_{max}L_{\delta}^{\mathbf{z}_{0}},
		\end{align*}}
	where we have used that the Fourier series of the nonlinearity $\sum_{k}\mathbf{R_{\boldsymbol{\kappa}}}\{\mathbf{z}\}e^{i\left\langle \boldsymbol{\kappa},\boldsymbol{\Omega}\right\rangle t}$converges
	to the function $\mathbf{R}(\mathbf{z},t)$. Due to the Lipschitz
	continuity of the nonlinearity and the forcing, this holds. We finally
	conclude, that $C_{\delta}^{\mathbf{z}_{0}}(\boldsymbol{\Omega})$
	is mapped to itself, if condition \eqref{eq:delta_cond_qper} holds. 
	
	Similarly, for two function $\mathbf{z},\tilde{\mathbf{z}}$ in $C_{\delta}^{\mathbf{z}_{0}}(\boldsymbol{\Omega})$,
	we obtain 
	\begin{align*}
	\left|\boldsymbol{\mathcal{H}}(\mathbf{z}(t))-\boldsymbol{\mathcal{H}}(\tilde{\mathbf{z}}(t))\right| & \leq\left|\sum_{\boldsymbol{\kappa}\in\mathbb{Z}^{k}}\mathbf{V}\mathbf{H}(T_{\kappa})\mathbf{V}^{-1}\left[\mathbf{\mathbf{R}_{\boldsymbol{\kappa}}\{\mathbf{z}\}}-\mathbf{\mathbf{R}_{\boldsymbol{\kappa}}\{\tilde{\mathbf{z}}\}}\right]e^{i\left\langle \boldsymbol{\kappa},\boldsymbol{\Omega}\right\rangle t}\right|\\
	& \leq2\left\Vert \mathbf{V}\right\Vert \left\Vert \mathbf{V}^{-1}\right\Vert h_{max}L_{\delta}^{\mathbf{z}_{0}}\left\Vert \mathbf{z}-\tilde{\mathbf{z}}\right\Vert _{0}\,.
	\end{align*}
	Therefore the iteration \eqref{eq:Map_quasiper} is a contraction
	on the space $C_{\delta}^{\mathbf{z}_{0}}(\boldsymbol{\Omega})$,
	if the condition 
	\begin{equation}
	2\left\Vert \mathbf{V}\right\Vert \left\Vert \mathbf{V}^{-1}\right\Vert h_{max}L_{\delta}^{\mathbf{z}_{0}}<\frac{1}{a}\qquad a\in\mathbb{R},\:a>1,\label{eq:conv_cond_quasiper}
	\end{equation}
	holds, which we reformulate in \eqref{eq:Lipschitz_cond_qper}. 
	
	\section{Explicit expressions for Fourier coefficients in Remark \ref{rem:Fourier_qper_posdep}}
	\label{sec:Proof4}
	To obtain the amplifications factors given in \eqref{eq:ampl_fkt_posdep}, we carry out the integration
	explicitly, we diagonalize the system with the matrix of the undamped
	modeshapes $\mathbf{U}$, (i.e., let $\mathbf{x}=\mathbf{U}\mathbf{y}$) and introduce the notation $\boldsymbol{\psi_{\boldsymbol{\kappa}}}=\mathbf{U}^{\top}\mathbf{f}_{\boldsymbol{\kappa}}$. Assuming an underdamped configuration ($\zeta_{j}<1$), we obtain for the $j^{\text{th}}$ degree of freedom 
	
	{\small
		\begin{align*}
		w_{j}(t) & =\sum_{\boldsymbol{\kappa}\in\mathbb{Z}^{k}}\int_{0}^{T_{\boldsymbol{\kappa}}}L_{j}(t-s,T_{\boldsymbol{\kappa}})\psi_{j,\boldsymbol{\kappa}}e^{i\left\langle \boldsymbol{\kappa},\boldsymbol{\Omega}\right\rangle s}\,ds\\
		& =\sum_{\boldsymbol{\kappa}\in\mathbb{Z}^{k}}\int_{0}^{T_{\boldsymbol{\kappa}}}\frac{e^{\alpha_{j}(t-s)}}{\omega_{j}}\left[\frac{e^{\alpha_{j}T_{\boldsymbol{\kappa}}}\left[\sin\omega_{j}(T_{\boldsymbol{\kappa}}+t-s)-e^{\alpha_{j}T_{\boldsymbol{\kappa}}}\sin\omega_{j}(t-s)\right]}{1+e^{2\alpha_{j}T_{\boldsymbol{\kappa}}}-2e^{\alpha_{j}T_{\boldsymbol{\kappa}}}\cos\omega_{j}T_{\boldsymbol{\kappa}}}+h(t-s)\sin\omega_{j}(t-s)\right]\psi_{j,\boldsymbol{\kappa}}e^{i\left\langle \boldsymbol{\kappa},\boldsymbol{\Omega}\right\rangle s}\,ds\\
		& =\sum_{\boldsymbol{\kappa}\in\mathbb{Z}^{k}}\int_{0}^{t}\frac{e^{\alpha_{j}(t-s)}}{\omega_{j}}\sin\omega_{j}(t-s)\psi_{j,\boldsymbol{\kappa}}e^{i\left\langle \boldsymbol{\kappa},\boldsymbol{\Omega}\right\rangle s}\,ds+
		\\
		&\quad\sum_{\boldsymbol{\kappa}\in\mathbb{Z}^{k}}\int_{0}^{T_{\boldsymbol{\kappa}}}\frac{e^{\alpha_{j}(t-s)}}{\omega_{j}}\left[\frac{e^{\alpha_{j}T_{\boldsymbol{\kappa}}}\left[\sin\omega_{j}(T_{\boldsymbol{\kappa}}+t-s)-e^{\alpha_{j}T_{\boldsymbol{\kappa}}}\sin\omega_{j}(t-s)\right]}{1+e^{2\alpha_{j}T_{\boldsymbol{\kappa}}}-2e^{\alpha_{j}T_{\boldsymbol{\kappa}}}\cos\omega_{j}T_{\boldsymbol{\kappa}}}\right]\psi_{j,\boldsymbol{\kappa}}e^{i\left\langle \boldsymbol{\kappa},\boldsymbol{\Omega}\right\rangle s}\,ds
		\end{align*}
	}
	{\small
		\begin{align*} \qquad&=\sum_{\boldsymbol{\kappa}\in\mathbb{Z}^{k}}\frac{e^{\alpha_{j}t}}{\omega_{j}}\psi_{j,\boldsymbol{\kappa}}\int_{0}^{t}\sin\omega_{j}(t-s)e^{(i\left\langle \boldsymbol{\kappa},\boldsymbol{\Omega}\right\rangle -\alpha_{j})s}\,ds+\\
		& \quad\sum_{\boldsymbol{\kappa}\in\mathbb{Z}^{k}}\frac{e^{\alpha_{j}(T_{\boldsymbol{\kappa}}+t)}}{\omega_{j}}\left(\frac{1}{1+e^{2\alpha_{j}T_{\boldsymbol{\kappa}}}-2e^{\alpha_{j}T_{\boldsymbol{\kappa}}}\cos\omega_{j}T_{\boldsymbol{\kappa}}}\right)\psi_{j,\boldsymbol{\kappa}}\left[\int_{0}^{T_{\boldsymbol{\kappa}}}e^{(i\left\langle \boldsymbol{\kappa},\boldsymbol{\Omega}\right\rangle -\alpha_{j})s}\sin\omega_{j}(T_{\boldsymbol{\kappa}}+t-s)\,ds - \right.\\
		&\quad 
		\left.  e^{\alpha_{j}T_{\boldsymbol{\kappa}}}\int_{0}^{T_{\boldsymbol{\kappa}}}e^{(i\left\langle \boldsymbol{\kappa},\boldsymbol{\Omega}\right\rangle -\alpha_{j})s}\sin\omega_{j}(t-s)\,ds\right]\\
		&=\sum_{\boldsymbol{\kappa}\in\mathbb{Z}^{k}}\frac{e^{\alpha_{j}t}}{\omega_{j}}\psi_{j,\boldsymbol{\kappa}}\left.\frac{e^{(i\left\langle \boldsymbol{\kappa},\boldsymbol{\Omega}\right\rangle -\alpha_{j})s}\left(\omega_{j}\cos\omega_{j}(t-s)+(i\left\langle \boldsymbol{\kappa},\boldsymbol{\Omega}\right\rangle -\alpha_{j})\sin\omega_{j}(t-s)\right)}{((i\left\langle \boldsymbol{\kappa},\boldsymbol{\Omega}\right\rangle -\alpha_{j})^{2}+\omega_{j}^{2})}\right|_{s=0}^{s=t}+\\
		&\quad\sum_{\boldsymbol{\kappa}\in\mathbb{Z}^{k}}\frac{e^{\alpha_{j}(T_{\boldsymbol{\kappa}}+t)}}{\omega_{j}}\left[\frac{\left.e^{(i\left\langle \boldsymbol{\kappa},\boldsymbol{\Omega}\right\rangle -\alpha_{j})s}\left(\omega_{j}\cos\omega_{j}(T_{\boldsymbol{\kappa}}+t-s)+(i\left\langle \boldsymbol{\kappa},\boldsymbol{\Omega}\right\rangle -\alpha_{j})\sin\omega_{j}(T_{\boldsymbol{\kappa}}+t-s)\right)\right|_{s=0}^{s=T_{\boldsymbol{\kappa}}}}{\left(1+e^{2\alpha_{j}T_{\boldsymbol{\kappa}}}-2e^{\alpha_{j}T_{\boldsymbol{\kappa}}}\cos\omega_{j}T_{\boldsymbol{\kappa}}\right)\left((i\left\langle \boldsymbol{\kappa},\boldsymbol{\Omega}\right\rangle -\alpha_{j})^{2}+\omega_{j}^{2}\right)}\right]\psi_{j,\boldsymbol{\kappa}}-\\
		&\sum_{\boldsymbol{\kappa}\in\mathbb{Z}^{k}}\frac{e^{\alpha_{j}(2T_{\boldsymbol{\kappa}}+t)}}{\omega_{j}}\left[\frac{\left.e^{(i\left\langle \boldsymbol{\kappa},\boldsymbol{\Omega}\right\rangle -\alpha_{j})s}\left(\omega_{j}\cos\omega_{j}(t-s)+(i\left\langle \boldsymbol{\kappa},\boldsymbol{\Omega}\right\rangle -\alpha_{j})\sin\omega_{j}(t-s)\right)\right|_{s=0}^{s=T_{\boldsymbol{\kappa}}}}{\left(1+e^{2\alpha_{j}T_{\boldsymbol{\kappa}}}-2e^{\alpha_{j}T_{\boldsymbol{\kappa}}}\cos\omega_{j}T_{\boldsymbol{\kappa}}\right)\left((i\left\langle \boldsymbol{\kappa},\boldsymbol{\Omega}\right\rangle -\alpha_{j})^{2}+\omega_{j}^{2}\right)}\right]\psi_{j,\boldsymbol{\kappa}}\\
		&=\sum_{\boldsymbol{\kappa}\in\mathbb{Z}^{k}}\frac{1}{(i\left\langle \boldsymbol{\kappa},\boldsymbol{\Omega}\right\rangle -\alpha_{j})^{2}+\omega_{j}^{2}}\psi_{j,\boldsymbol{\kappa}}e^{i\left\langle \boldsymbol{\kappa},\boldsymbol{\Omega}\right\rangle t}.
		\end{align*}
	}
	For the critically damped configuration ($\zeta_{j}=1$), we obtain
	{\small
		\begin{align*}
		w_{j}(t)	&=\sum_{\boldsymbol{\kappa}\in\mathbb{Z}^{k}}\int_{0}^{T_{\boldsymbol{\kappa}}}L_{j}(t-s,T_{\boldsymbol{\kappa}})\psi_{j,\boldsymbol{\kappa}}e^{i\left\langle \boldsymbol{\kappa},\boldsymbol{\Omega}\right\rangle s}\,ds
		\\
		&=\sum_{\boldsymbol{\kappa}\in\mathbb{Z}^{k}}\int_{0}^{T_{\boldsymbol{\kappa}}}\left[\frac{e^{\alpha_{j}(T_{\boldsymbol{\kappa}}+t-s)}\left[\left(1-e^{\alpha_{j}T}\right)(t-s)+T_{\boldsymbol{\kappa}}\right]}{\left(1-e^{\alpha_{j}T_{\boldsymbol{\kappa}}}\right)^{2}}+h(t-s)(t-s)e^{\alpha_{j}(t-s)}\right]\psi_{j,\boldsymbol{\kappa}}e^{i\left\langle \boldsymbol{\kappa},\boldsymbol{\Omega}\right\rangle s}\,ds
		\\
		&=\sum_{\boldsymbol{\kappa}\in\mathbb{Z}^{k}}e^{\alpha_{j}t}\psi_{j,\boldsymbol{\kappa}}\int_{0}^{t}(t-s)e^{(i\left\langle \boldsymbol{\kappa},\boldsymbol{\Omega}\right\rangle -\alpha_{j})s}\,ds+
		\\
		&\quad\sum_{\boldsymbol{\kappa}\in\mathbb{Z}^{k}}\frac{e^{\alpha_{j}(T_{\boldsymbol{\kappa}}+t)}}{\left(1-e^{\alpha_{j}T_{\boldsymbol{\kappa}}}\right)^{2}}\psi_{j,\boldsymbol{\kappa}}\int_{0}^{T_{\boldsymbol{\kappa}}}\left[\left(1-e^{\alpha_{j}T_{\boldsymbol{\kappa}}}\right)(t-s)+T_{\boldsymbol{\kappa}}\right]e^{(i\left\langle \boldsymbol{\kappa},\boldsymbol{\Omega}\right\rangle -\alpha_{j})s}\,ds
		\\
		&=\sum_{\boldsymbol{\kappa}\in\mathbb{Z}^{k}}\frac{1}{(i\left\langle \boldsymbol{\kappa},\boldsymbol{\Omega}\right\rangle -\alpha_{j})^{2}}\psi_{j,\boldsymbol{\kappa}}e^{i\left\langle \boldsymbol{\kappa},\boldsymbol{\Omega}\right\rangle t}.
		\end{align*}
	}
	Finally, for the overdamped configuration ($\zeta_{j}>1$), we obtain
	{\small
		\begin{align*}
		w_{j}(t)	&=\sum_{\boldsymbol{\kappa}\in\mathbb{Z}^{k}}\int_{0}^{T_{\boldsymbol{\kappa}}}L_{j}(t-s,T_{\boldsymbol{\kappa}})\psi_{j,\boldsymbol{\kappa}}e^{i\left\langle \boldsymbol{\kappa},\boldsymbol{\Omega}\right\rangle s}\,ds\\
		&=\sum_{\boldsymbol{\kappa}\in\mathbb{Z}^{k}}\int_{0}^{T_{\boldsymbol{\kappa}}}\left[\frac{e^{\beta_{j}(T_{\boldsymbol{\kappa}}+t-s)}\left(1-e^{\gamma_{j}T_{\boldsymbol{\kappa}}}\right)-e^{\gamma_{j}(T_{\boldsymbol{\kappa}}+t-s)}\left(1-e^{\beta_{j}T_{\boldsymbol{\kappa}}}\right)}{(\beta_{j}-\gamma_{j})\left(1-e^{\gamma_{j}T_{\boldsymbol{\kappa}}}-e^{\beta_{j}T_{\boldsymbol{\kappa}}}+e^{\left(\gamma_{j}+\beta_{j}\right)T_{\boldsymbol{\kappa}}}\right)}+\frac{h(t-s)\left(e^{\beta_{j}(t-s)}-e^{\gamma_{j}(t-s)}\right)}{(\beta_{j}-\gamma_{j})}\right]\psi_{j,\boldsymbol{\kappa}}e^{i\left\langle \boldsymbol{\kappa},\boldsymbol{\Omega}\right\rangle s}\,ds\\
		&=\sum_{\boldsymbol{\kappa}\in\mathbb{Z}^{k}}\frac{\psi_{j,\boldsymbol{\kappa}}}{(\beta_{j}-\gamma_{j})}\left[e^{\beta_{j}t}\int_{0}^{t}e^{(i\left\langle \boldsymbol{\kappa},\boldsymbol{\Omega}\right\rangle -\beta_{j})s}\,ds-e^{\gamma_{j}t}\int_{0}^{t}e^{(i\left\langle \boldsymbol{\kappa},\boldsymbol{\Omega}\right\rangle -\gamma_{j})s}\,ds\right]\\
		&\quad+\sum_{\boldsymbol{\kappa}\in\mathbb{Z}^{k}}\frac{\psi_{j,\boldsymbol{\kappa}}\left[\left(1-e^{\gamma_{j}T}\right)e^{\beta_{j}(T+t)}\int_{0}^{T_{\boldsymbol{\kappa}}}e^{(i\left\langle \boldsymbol{\kappa},\boldsymbol{\Omega}\right\rangle -\beta_{j})s}\,ds-\left(1-e^{\beta_{j}T}\right)e^{\gamma_{j}(T+t)}\int_{0}^{T_{\boldsymbol{\kappa}}}e^{(i\left\langle \boldsymbol{\kappa},\boldsymbol{\Omega}\right\rangle -\gamma_{j})s}\,ds\right]}{(\beta_{j}-\gamma_{j})(1-e^{\gamma_{j}T_{\boldsymbol{\kappa}}}-e^{\beta_{j}T_{\boldsymbol{\kappa}}}+e^{\left(\gamma_{j}+\beta_{j}\right)T_{\boldsymbol{\kappa}}})}\\
		&=\sum_{\boldsymbol{\kappa}\in\mathbb{Z}^{k}}\frac{1}{(\beta_{j}-i\left\langle \boldsymbol{\kappa},\boldsymbol{\Omega}\right\rangle )(\gamma_{j}-i\left\langle \boldsymbol{\kappa},\boldsymbol{\Omega}\right\rangle )}\psi_{j,\boldsymbol{\kappa}}e^{i\left\langle \boldsymbol{\kappa},\boldsymbol{\Omega}\right\rangle t}.
		\end{align*}
		
	}

\end{document}